\documentclass[UTF-8,reqno]{amsart}
\usepackage{enumerate}
\usepackage{mhequ}
\usepackage{amssymb,url,color, booktabs,nccmath}
\usepackage[left=3.2cm,right=3.2cm,top=4cm,bottom=4cm]{geometry}
\usepackage{mathrsfs}
\usepackage{enumitem,dsfont}

\usepackage{color}
\usepackage[colorlinks=true]{hyperref}
\hypersetup{
    linkcolor=blue,          
    citecolor=red,        
    filecolor=blue,      
    urlcolor=cyan
}

\definecolor{darkergreen}{rgb}{0.0, 0.5, 0.0}

\numberwithin{equation}{section}
\def\theequation{\arabic{section}.\arabic{equation}}
\newcommand{\be}{\begin{eqnarray}}
\newcommand{\ee}{\end{eqnarray}}
\newcommand{\ce}{\begin{eqnarray*}}
\newcommand{\de}{\end{eqnarray*}}
\newtheorem{theorem}{Theorem}[section]
\newtheorem{lemma}[theorem]{Lemma}
\newtheorem{proposition}[theorem]{Proposition}
\newtheorem{Examples}[theorem]{Example}
\newtheorem{corollary}[theorem]{Corollary}

\newtheorem{definition}[theorem]{Definition}
\theoremstyle{definition}
\newtheorem{remark}[theorem]{Remark}

\DeclareMathOperator{\supp}{supp}

\def\eps{\varepsilon}

\def\u{\mathbf{u}}

\def\p{\partial}

\def\[{{\Big[}}
\def\]{{\Big]}}
\def\<{{\langle}}
\def\>{{\rangle}}
\def\({{\Big(}}
\def\){{\Big)}}

\def\bx{{\mathbf{x}}}

\def\dif{{\mathord{{\rm d}}}}

\def\no{\nonumber}
\def\={&\!\!=\!\!&}
\DeclareMathOperator*{\esssup}{esssup}

\def\cB{{\mathcal B}}

\def\mN{{\mathbb N}}

\def\mP{{\mathbb P}}

\def\mR{{\mathbb R}}
\def\mS{{\mathbb S}}

\def\1{{\mathbf{1}}}

\def\sC{{\mathscr C}}

\def\sP{{\mathscr P}}

\def\E{\mathbf E}

\def\geq{\geqslant}
\def\leq{\leqslant}

\def\div{\mathord{{\rm div}}}

\def\eps{\varepsilon}

\def\u{\mathbf{u}}

\def\p{\partial}

\def\[{{\Big[}}
\def\]{{\Big]}}
\def\<{{\langle}}
\def\>{{\rangle}}
\def\({{\Big(}}
\def\){{\Big)}}

\def\bx{{\mathbf{x}}}

\def\dif{{\mathord{{\rm d}}}}

\def\no{\nonumber}
\def\={&\!\!=\!\!&}
\def\bt{\begin{theorem}}
\def\et{\end{theorem}}
\def\bl{\begin{lemma}}
\def\el{\end{lemma}}
\def\br{\begin{remark}}
\def\er{\end{remark}}
\def\bx{\begin{Examples}}
\def\ex{\end{Examples}}
\def\bd{\begin{definition}}
\def\ed{\end{definition}}
\def\bp{\begin{proposition}}
\def\ep{\end{proposition}}
\def\bc{\begin{corollary}}
\def\ec{\end{corollary}}

\def\geq{\geqslant}
\def\leq{\leqslant}

\def\div{\mathord{{\rm div}}}

\def\Id{\textrm{Id}}

 \def\R{\mathbb R}
 \def\R{\mathbb R}    
\def\N{\mathbb N}  
   
\def\<{\langle} \def\>{\rangle}

\allowdisplaybreaks

\begin{document}

\title[Probabilistically strong and  Markov solutions  to stochastic 3D NSE]{Global-in-time probabilistically strong and  Markov solutions  to stochastic 3D Navier--Stokes equations: existence and non-uniqueness}

\author{Martina Hofmanov\'a}
\address[M. Hofmanov\'a]{Fakult\"at f\"ur Mathematik, Universit\"at Bielefeld, D-33501 Bielefeld, Germany}
\email{hofmanova@math.uni-bielefeld.de}

\author{Rongchan Zhu}
\address[R. Zhu]{Department of Mathematics, Beijing Institute of Technology, Beijing 100081, China; Fakult\"at f\"ur Mathematik, Universit\"at Bielefeld, D-33501 Bielefeld, Germany}
\email{zhurongchan@126.com}

\author{Xiangchan Zhu}
\address[X. Zhu]{ Academy of Mathematics and Systems Science,
Chinese Academy of Sciences, Beijing 100190, China; Fakult\"at f\"ur Mathematik, Universit\"at Bielefeld, D-33501 Bielefeld, Germany}
\email{zhuxiangchan@126.com}
\thanks{
This project has received funding from the European Research Council (ERC) under the European Union's Horizon 2020 research and innovation programme (grant agreement No. 949981).  The financial support by the DFG through the CRC 1283 ``Taming uncertainty and profiting
 from randomness and low regularity in analysis, stochastics and their applications'' is greatly acknowledged.
 R.Z. is grateful to the financial supports of the NSFC (No.  11922103).
  X.Z. is grateful to the financial supports  in part by National Key R\&D Program of China (No. 2020YFA0712700) and the NSFC (No. 11771037,  12090014, 11688101) and
  the support by key Lab of Random Complex Structures and Data Science,
 Youth Innovation Promotion Association (2020003), Chinese Academy of Science.
}

\begin{abstract}
We are concerned with the  three dimensional incompressible  Navier--Stokes equations driven by an additive stochastic forcing of trace class. First, for every divergence free initial condition in $L^{2}$ we establish existence of infinitely many  global-in-time probabilistically strong and analytically weak solutions, solving one of the  open problems in the field. This result in particular implies non-uniqueness in law. Second,
 we prove non-uniqueness of the associated Markov processes in a suitably chosen class of analytically weak solutions satisfying a relaxed form of an energy inequality. Translated to the deterministic setting, we obtain non-uniqueness of the associated semiflows.
 \end{abstract}


\date{\today}

\maketitle

\tableofcontents

\section{Introduction}

The research focused on well-posedness of the three dimensional incompressible Navier--Stokes equations has experienced an immense  breakthrough recently: Buckmaster and Vicol \cite{BV19a} established non-uniqueness of weak solutions with finite kinetic energy. More precisely, the authors showed that for any prescribed smooth and non-negative function $e$ there is a weak solution whose kinetic energy is given by $e$.  Remarkably,  Buckmaster, Colombo and Vicol \cite{BCV18} were even able to connect two arbitrary strong solutions via a weak solution. Burczak, Modena and Sz{\'e}kelyhidi~Jr. \cite{BMS20} then obtained a number of new ill-posedness results for power-law fluids and in particular also non-uniqueness of weak solutions to the Navier--Stokes equations for every given divergence free initial condition in $L^{2}$. Sharp non-uniqueness results for the Navier--Stokes equations in dimension $d\geq 2$ were obtained by Cheskidov and Luo \cite{CL20}.
All the above results rely on the method of convex integration introduced to fluid dynamics by  De~Lellis and Sz{\'e}kelyhidi~Jr. \cite{DelSze2,DelSze3,DelSze13}. This method has already led to a number of groundbreaking results concerning the incompressible Euler equations,  culminating  in the proof of  Onsager's conjecture by Isett \cite{Ise} and by Buckmaster, De~Lellis, Sz{\'e}kelyhidi~Jr. and Vicol \cite{BDSV}. The interested reader is referred to the nice reviews \cite{BV19, BV20} for further details and references.

In light of these exciting but rather negative developments, the question of a possible regularizing effect provided by a suitable stochastic perturbation becomes even more prominent. Recall that some properties of the Navier--Stokes system have indeed been shown to improve under the presence of a stochastic noise. In the deterministic setting, a selection of solutions depending  continuously on the initial condition has not been obtained. However, the probabilistic counterpart, i.e. the Feller property and even the strong Feller property which corresponds to a smoothing with respect to the initial condition, was established for a sufficiently non-degegenerate noise  by Da~Prato and Debussche \cite{DD03} and by Flandoli and Romito \cite{FR08}. A linear multiplicative noise as well as a transport type noise  has been shown to prevent blow up of strong solutions in a certain sense with large probability, see R\"ockner, Zhu and Zhu \cite{RZZ14} and Flandoli, Luo \cite{FL19}. The latter result has been generalized to  provide  regularization by highly oscillating but deterministic vector fields by Flandoli, Hofmanov\'a, Luo and Nilssen \cite{FHLN20}.

\subsection{Main results}

In the present paper, we are concerned with  stochastic Navier--Stokes equations  on $\mathbb{T}^3$ driven by an additive stochastic noise. The  equations govern the time evolution of the fluid velocity $u$ and  read as
\begin{equation}
\label{1}
\aligned
 \dif u+\div(u\otimes u)\,\dif t+\nabla P\,\dif t&=\Delta u \,\dif t+\dif B,
\\
\div u&=0,\\
u(0)&=u_{0}.
\endaligned
\end{equation}
Here $P$ is the associated pressure, $B$ is a $GG^*$-Wiener process on some probability space $(\Omega, \mathcal{F}, \mathbf{P})$ and $G$ is a Hilbert--Schmidt operator from $U$ to $L^{2}$. We do not require any further space regularity of the noise.

Unlike  their  deterministic counterpart, the stochastic equations possess additional structural features beside mere existence, uniqueness and continuous dependence on initial conditions. First of all, we distinguish between probabilistically strong and probabilistically weak (also called  martingale) solutions.
Probabilistically strong solutions are constructed on a given probability space and are adapted with respect to the given  noise, which means the solutions at time $t$ only depend on   the noise from $0$ to $t$. Probabilistically weak solutions do not have this property: they are typically obtained by the method of compactness where the noise as well as the probability space becomes part of the construction.  This is also the case for the stochastic counterpart of the so-called  Leray solutions, meaning solutions satisfying an energy inequality and belonging to $C_{\rm{weak}}(0,T;L^{2})\cap L^2(0,T;H^1)$. They are obtained by a compactness argument and in the stochastic setting they are only probabilistically weak. Indeed, on the one hand, it is necessary to take expectation to control the noise and obtain uniform energy estimates, which then leads to probabilistically weak solutions. Moreover, due to the lack of uniqueness we cannot apply Yamada--Watanabe's theorem to obtain probabilistically strong solutions.  On the other hand, if we analyze the equation $\omega$-wise, then the converging subsequence from compactness argument may depend on $\omega$. This destroys adaptedness and measurable selection only applies to a given $\sigma$-field, not a filtration.

Consequently, it has been a long standing open problem to construct probabilistically strong solutions to the Navier--Stokes system \eqref{1}, see page 84 in \cite{F08}.
In our previous work \cite{HZZ19}, we made a first step in this direction and  obtained probabilistically strong and analytically weak solutions, i.e. not Leray, before a suitable stopping time. The initial value was a part of the construction and could not have been prescribed.
As a first goal of the present paper we intend to overcome these two limitations: We prove the existence of global-in-time non-unique probabilistically strong and analytically weak solutions for every given divergence free initial condition in $L^{2}$.

\bt\label{th:ma4}
Let $u_{0}\in L^2$ $\mathbf{P}$-a.s. be a  divergence free initial condition independent of the Wiener process $B$. There exist infinitely many probabilistically strong and analytically weak solutions to the Navier--Stokes system \eqref{1} on $[0,\infty)$. The solutions belong to  $L^p_{\rm loc}([0,\infty);L^2)\cap C([0,\infty),W^{\frac12,\frac{31}{30}})$ $\mathbf{P}$-a.s. for all $p\in[1,\infty)$.
\et

Furthermore, our construction  directly implies the following result, essentially proving part of the classical Yamada--Watanabe--Engelbert's theorem  (see Kurtz \cite{K07}, Cherny \cite{C03} for a  version applicable to finite dimensional stochastic differential equations).
It strengthens \cite[Theorem~1.2]{HZZ19} in the sense that it shows  non-uniqueness in law for arbitrary divergence free initial condition in $L^{2}$. However, the considered solution class is larger and hence \cite[Theorem 1.2]{HZZ19} does not follow from our present result. More precisely,  even for $\gamma$ sufficiently small, Theorem~\ref{th:ma4} does not give solutions in $L^{\infty}_{\rm loc}([0,\infty);L^{2})\cap L^{2}_{\rm loc}([0,\infty);H^{\gamma})$ as  considered in \cite[Theorem 1.2]{HZZ19}.

\begin{corollary}\label{cor:law}
Non-uniqueness in law holds for the Navier--Stokes system \eqref{1} for every given initial law  supported on divergence free vector fields in  $L^{2}$.
\end{corollary}

Another desired property of stochastic systems is the Markov property, which in the deterministic setting translates to the usual semigroup or semiflow property for solutions. It is an immediate consequence of uniqueness. However,  also for evolution systems where uniqueness is either unknown or  not valid, it may be possible to select a Markov solution by an abstract selection procedure introduced by Krylov \cite{K73}. A similar approach developed by Cardona and Kapitanskii \cite{CorKap} leads to a selection of a semiflow for deterministic equations.  In the context of Leray solutions to the Navier--Stokes system \eqref{1}, Markov selection was done by Flandoli and Romito \cite{FR08}, further models were considered by Goldys, R\"ockner and Zhang \cite{GRZ09}, stochastic compressible Navier--Stokes system was treated by Breit, Feireisl and Hofmanov\'a \cite{BFH18markov} and recently we presented an application to  stochastic Euler equations in \cite{HZZ20}.

The second main result of the present paper establishes non-uniqueness of Markov selections.

\bt\label{th:ma2}
Markov families associated to the Navier--Stokes system \eqref{1} are not unique.
\et

In the context of deterministic Navier--Stokes system, i.e. \eqref{1} with $G=0$, we may therefore deduce the following new result.

\begin{corollary}\label{cor:det}
The deterministic Navier--Stokes system generates non-unique semiflow solutions.
\end{corollary}

\subsection{Ideas of the proofs}

Both our main results, Theorem~\ref{th:ma4} and Theorem~\ref{th:ma2}, make use of the convex integration method. An adaptation of this method to the stochastic setting has already appeared in a number of works, proving results of ill-posedness in various settings, see \cite{CFF19,BFH20,HZZ19,HZZ20, HZZ21,  Ya20a, Ya20b, Ya21, Ya21b, Ya21c, RS21}. In particular, our previous work \cite{HZZ19} was also concerned with the Navier--Stokes system \eqref{1} and we put forward a general approach leading to non-uniqueness in law. The~core lay in the extension of a convex integration solution beyond a stopping time by connecting it to a general martingale solution. This was then applied to  other models in  \cite{HZZ20, HZZ21, Ya20a, Ya20b, Ya21, Ya21b, Ya21c, RS21}. Among the results of \cite{HZZ20}, we  also used convex integration to obtain non-unique strong Markov selections to the stochastic Euler equations. Combining paracontrolled distributions \cite{GIP15, ZZ15} together with  convex integration, we could construct infinitely many global-in-time probabilistically strong solutions to 3D Navier-Stokes equations driven by space-time white noise in \cite{HZZ21}.

In order to derive the results of the present paper, it is necessary to significantly strengthen the convex integration procedure from \cite{HZZ19}. We are inspired by the recent work \cite{BMS20}, but we formulate the main iteration rather in the spirit of \cite{BV19} and use the intermittent jets from \cite{BCV18}. This seems to be better suited for the stochastic setting, see Remark~\ref{r:5.3}.  More precisely, for the non-uniqueness of Markov selections in Theorem~\ref{th:ma2}, the key point is to obtain non-uniqueness of solutions which are stable under approximations of the initial datum and satisfying a suitable shift and concatenation property. That is, a class of solutions with a suitable version of an energy inequality leading to  compactness. As non-uniqueness of Leray solutions, which are typically used for Markov selections, is out of reach of the current techniques, we formulated a relaxed version of the energy inequality (see the discussion in Section~\ref{s:martsol} and in particular Definition~\ref{martingale solution}),
 which we prove to  be satisfied by both Leray as well as convex integration solutions.
To this end,  we use convex integration to  construct  solutions with a prescribed deterministic kinetic energy. This is actually rather surprising as  the  energies of solutions to the stochastic Navier--Stokes equations \eqref{1} are  a priori random.

\bt\label{Main results1}
Let  $\bar{e}\geq \underline{e}>4 $ and $\tilde e>0$ be given. Then there  exist  $\gamma\in (0,1)$ and a $\mathbf{P}$-a.s. strictly positive stopping time $\mathfrak{t}$   such that the following holds true: For every $e:[0,1]\to [\underline{e},\infty)$  belonging to $C^{1}_{b}$ with $\|e\|_{C^{0}}\leq \bar{e}$ and $\|e'\|_{C^{0}}\leq \tilde e$,  there exist a deterministic initial value $u_{0}$ and
	a probabilistically strong  and analytically weak solution $u\in C([0,\mathfrak{t}];H^{\gamma})$ $\mathbf{P}$-a.s. to \eqref{1}
satisfying
	\begin{equation}\label{eq:vv}
	\esssup_{\omega\in\Omega}\sup_{t\in[0,\mathfrak{t}]}\|u(t)\|_{H^{\gamma}}<\infty,
	\end{equation}
	and for $t\in [0,\mathfrak{t}]$
	\begin{equation}\label{eq:env}
	\|u(t)\|_{L^{2}}^2=e(t).
	\end{equation}
	Furthermore, if $e_{1},$ $e_{2}$ are given  as above with the same constants $\underline{e},\bar{e},\tilde e$, and additionally $e_{1}=e_{2}$ on  $[0,t]$ for some $t\in[0,1]$, then the corresponding solutions $u_{1}$ and $u_{2}$ coincide on $[0,t\wedge \mathfrak{t}]$.
	
\et

The choice  $e(t)=c_{0}+c_{1}t$ then permits to obtain convex integration solutions starting from the same initial condition and satisfying the relaxed energy inequality in Definition~\ref{martingale solution}. This is then  compatible with the Markov selection procedure and leads to the proof of Theorem~\ref{th:ma2} by using approach developed in \cite{HZZ19}. Note that, similarly to \cite{BMS20}, we require the energy $e$ to be strictly positive. This seems to be even more important in the stochastic setting and in particular it is not clear how to generalize the method of \cite{BV19a} which applies  also to vanishing energies.

As the main ingredient of the construction of global probabilistically strong solutions in  Theorem~\ref{th:ma4} we need to strengthen the convex integration method in another direction. Namely, it is necessary to  prescribe a general divergence free initial condition in $L^{2}$. Note that the initial condition was a part of the construction in Theorem~\ref{Main results1} as well as in \cite{HZZ19}.  We again  profit from the ideas in \cite{BMS20} and establish the existence of infinitely many probabilistically strong and analytically weak solutions defined up to a suitable stopping time. Based on this result, we are able to extend every such convex integration solution by another convex integration solution, proving thereby Theorem~\ref{th:ma4} after countably many extensions.

\begin{remark}
The above results can  easily be extended to a linear multiplicative  noise of the form $u\,\dif B$ where  $B$ is  $\mR$-valued Brownian motion. Indeed, similarly to \cite{HZZ19} we may perform a  transformation $v:=u\exp(-B)$ with $v$ satisfying the following random PDE
\begin{align}
\p_tv+\frac12v-\Delta v+\theta \div (v\otimes v)+\theta^{-1} \nabla P=&0,\label{eq:lin}
\\\div v=&0,\no\end{align}
with $\theta=\exp(B)$. Then we may use similar arguments as in Section \ref{s:1.1} and Section \ref{s:in} to construct convex integration solutions to \eqref{eq:lin} with a prescribed energy and for a given initial condition, respectively. Consequently, existence and non-uniqueness of Markov families as well global existence and non-uniqueness of probabilistically strong solutions in the spirit of Theorem~\ref{th:ma4} and Theorem~\ref{th:ma2}, respectively, both hold in this case.
\end{remark}

\noindent{\bf Organization of the paper.}
In Section~\ref{s:not} we collect the basic notations used throughout the~paper. Section~\ref{s:1.1} is devoted to our first convex integration result, Theorem~\ref{Main results1}. This is then used in Section~\ref{sec:mar} in order to prove Theorem~\ref{th:ma2} and Corollary~\ref{cor:det}. Section~\ref{s:in} is concerned with the proof of Theorem~\ref{th:ma4} and Corollary~\ref{cor:law}. In Appendix we collect several auxiliary results.

\section{Notations}
\label{s:not}

\subsection{Function spaces}

  Throughout the paper, we employ the notation $a\lesssim b$ if there exists a constant $c>0$ such that $a\leq cb$, and we write $a\simeq b$ if $a\lesssim b$ and $b\lesssim a$. $\mN_{0}:=\mN\cup \{0\}$. Given a Banach space $E$ with a norm $\|\cdot\|_E$ and $T>0$, we write $C_TE=C([0,T];E)$ for the space of continuous functions from $[0,T]$ to $E$, equipped with the supremum norm $\|f\|_{C_TE}=\sup_{t\in[0,T]}\|f(t)\|_{E}$. We also use $CE$ or $C([0,\infty);E)$ to denote the space of continuous functions from $[0,\infty)$ to $E$. For $\alpha\in(0,1)$ we  define $C^\alpha_TE$ as the space of $\alpha$-H\"{o}lder continuous functions from $[0,T]$ to $E$, endowed with the norm $\|f\|_{C^\alpha_TE}=\sup_{s,t\in[0,T],s\neq t}\frac{\|f(s)-f(t)\|_E}{|t-s|^\alpha}+\sup_{t\in[0,T]}\|f(t)\|_{E}.$ Here we use $C_T^\alpha$ to denote the case when $E=\mathbb{R}$.  We  also use $C_\mathrm{loc}^\alpha E$ to denote the space of functions from $[0,\infty)$ to $E$ satisfying $f|_{[0,T]}\in C_T^\alpha E$ for all $T>0$. For $p\in [1,\infty]$ we write $L^p_TE=L^p([0,T];E)$ for the space of $L^p$-integrable functions from $[0,T]$ to $E$, equipped with the usual $L^p$-norm. We also use $L^p_{\mathrm{loc}}([0,\infty);E)$ to denote the space of functions $f$ from $[0,\infty)$ to $E$ satisfying $f|_{[0,T]}\in L^p_T E$ for all $ T>0$.
    We use $L^p$ to denote the set of  standard $L^p$-integrable functions from $\mathbb{T}^3$ to $\mathbb{R}^3$. For $s>0$, $p>1$ we set $W^{s,p}:=\{f\in L^p; \|(I-\Delta)^{\frac{s}{2}}f\|_{L^p}<\infty\}$ with the norm  $\|f\|_{W^{s,p}}=\|(I-\Delta)^{\frac{s}{2}}f\|_{L^p}$. Set $L^{2}_{\sigma}=\{f\in L^2; \int_{\mathbb{T}^{3}} f\,\dif x=0,\div f=0\}$. For $s>0$, we define $H^s:=W^{s,2}\cap L^2_\sigma$ and we normalize the corresponding norm so that  $\|f\|_{H^s}\leq \|\nabla f\|_{L^2}$. This is used in order to  include Leray solutions into our class of solutions without worrying about an implicit constant. For $s<0$ define $H^s$ to be the dual space of $H^{-s}$.
For $T>0$ and a domain $D\subset\R^{+}$ we denote by  $C^{N}_{T,x}$ and $C^{N}_{D,x}$, respectively, the space of $C^{N}$-functions on $[0,T]\times\mathbb{T}^{3}$ and on $D\times\mathbb{T}^{3}$, respectively,  $N\in\N_{0}$. The spaces are equipped with the norms
$$
\|f\|_{C^N_{T,x}}=\sum_{\substack{0\leq n+|\alpha|\leq N\\ n\in\N_{0},\alpha\in\N^{3}_{0} }}\|\partial_t^n D^\alpha f\|_{L^\infty_T L^\infty},\qquad \|f\|_{C^N_{D,x}}=\sum_{\substack{0\leq n+|\alpha|\leq N\\ n\in\N_{0},\alpha\in\N^{3}_{0} }}\sup_{t\in D}\|\partial_t^n D^\alpha f\|_{ L^\infty}.
$$
 For a Polish space $H$ we denote by $\mathcal{B}(H)$  the $\sigma$-algebra of Borel sets in $H$.

\subsection{Probabilistic elements}\label{s:2.2}

Let  $\Omega_0:=C([0,\infty);H^{-3})\cap L^2_{\rm{loc}}([0,\infty);L^2_\sigma)$
and let $\mathscr{P}(\Omega_0)$ denote the set of all probability measures on $(\Omega_0,\mathcal{B})$ with $\mathcal{B}$ being the Borel $\sigma$-algebra coming from the topology of locally uniform convergence on $\Omega_0$. Let  $x:\Omega_0\rightarrow H^{-3}$ denote the canonical process on $\Omega_{0}$ given by
$x_t(\omega)=\omega(t).$
Similarly, for $t\geq 0$ we define $\Omega_{0}^t:=C([t,\infty);H^{-3})\cap L^2_{\rm{loc}}([t,\infty);L^2_\sigma)$ equipped with its Borel $\sigma$-algebra $\mathcal{B}^{t}$ which coincides with $\sigma\{ x(s),s\geq t\}$.
Finally,  we define the canonical filtration  $\mathcal{B}_t^0:=\sigma\{ x(s),s\leq t\}$, $t\geq0$, as well as its right continuous version $\mathcal{B}_t:=\cap_{s>t}\mathcal{B}^0_s$, $t\geq 0$. For a given probability measure $P$ we denote by $\mathbf{E}^P$  the expectation under $P$.

Regarding the driving noise, we assume that $B$ is a $GG^*$-Wiener process on some probability space $(\Omega, \mathcal{F}, \mathbf{P})$ and $G$ is a Hilbert--Schmidt operator from $U$ to $L_{\sigma}^2(\mathbb{T}^{3})$ for some Hilbert space $U$. We denote the square of its Hilbert--Schmidt norm by $C_{G}:=\|G\|^{2}_{L_2(U;L^{2}_{\sigma})}$. 
For notational convenience, we work under the assumption $G\neq 0$ and  we outline  a possible extension to the deterministic setting in Remark~\ref{r:CG}.

\section{Construction of solutions with a prescribed energy}
\label{s:1.1}

This section is devoted to the proof of Theorem~\ref{Main results1}. More precisely, by means of  the convex integration method we construct solutions to the Navier--Stokes system \eqref{1} with a prescribed kinetic energy: for a given $C_{b}^{1}$-strictly positive function $e$ we construct an analytically weak and probabilistically strong solution before a stopping time such that $\|u(t)\|_{L^2}^2=e(t)$ $\mathbf P$-a.s. We note that the energy $e$ is deterministic which leads to a rather surprising result as one would expect that kinetic energies of solutions to the stochastic Navier--Stokes system are a priori random. Results of this section are then employed in Section~\ref{sec:mar} to construct non-unique Markov solutions and to prove Theorem~\ref{th:ma2}.

Note that the initial value $u(0)$ in the present section is part of the construction, that is, it cannot be prescribed in advance. We focus on this issue later in Section~\ref{s:in} where a modified convex integration scheme permits to construct solutions for a given initial condition, leading to the existence of infinitely many probabilistically strong and global-in-time solutions for every divergence free initial condition in Theorem~\ref{th:ma4}. However, the latter solutions cannot have a prescribed energy, that is, they are not suitable for Markov selections.

In this and the following section we fix a probability space $(\Omega,\mathcal{F},\mathbf{P})$ with a  $GG^*$-Wiener process $B$. We let $(\mathcal{F}_t)_{t\geq0}$ be the normal filtration generated by $B$, that is, the  canonical right continuous filtration augmented by all the $\mathbf{P}$-negligible sets.  In order to verify that the solution constructed in this section is a martingale solution before a suitable stopping time, it is essential that the solution is adapted to this filtration, which corresponds to a probabilistically strong solution.

We intend to develop an iteration procedure leading to the proof of Theorem \ref{Main results1}. To this end, we decompose the Navier--Stokes system \eqref{1} into two parts, one is linear and contains the stochastic integral, whereas the second one is a nonlinear but random PDE. More precisely, we consider
\begin{equation}\label{linear}
\aligned
\dif z -\Delta z \,\dif t+\nabla P_1\,\dif t&=\dif B,
\\\div z&=0,
\\z(0)&=0,
\endaligned
\end{equation}
and
\begin{equation}\label{nonlinear}
\aligned
\partial_tv -\Delta v+\div((v+z)\otimes (v+z))+\nabla P_{2}&=0,
\\
\div v&=0,
\endaligned
\end{equation}
where by $P_{1}$ and $P_{2}$ we denoted the associated pressure terms.

Then, we apply the convex integration method to the nonlinear equation (\ref{nonlinear}). The iteration is indexed by a parameter $q\in\mathbb{N}_{0}$. At each step $q$, a pair $(v_q, \mathring{R}_q)$ is constructed solving the following system
\begin{equation}\label{induction}
\aligned
\partial_tv_q-\Delta v_q +\div((v_q+z_q)\otimes (v_q+z_q))+\nabla p_q&=\div \mathring{R}_q,
\\
\div v_q&=0.
\endaligned
\end{equation}
{Here $z_q=\mP_{\leq f(q)}$ for $f(q)=\lambda_{q+1}^{\alpha/8}$ and $\mathbb{P}_{\leq f(q)}$ is the projection onto Fourier frequencies smaller  than $ f(q)$ in absolute value. We require $\alpha b\in 8\mN$.}
We consider an increasing sequence $\{\lambda_q\}_{q\in\mathbb{N}_{0}}\subset \mathbb{N}$ which diverges to $\infty$, and a sequence $\{\delta_q\}_{q\in \mathbb{N}}\subset(0,1)$  which is decreasing to $0$. We choose $a\in\mathbb{N},$ $ b\in\mathbb{N},$ $  \beta\in (0,1)$ and let
$$\lambda_q=a^{(b^q)}, \quad \delta_q=\lambda_1^{2\beta}\lambda_q^{-2\beta},$$
where $\beta$ will be chosen sufficiently small and $a$ as well as $b$ will be chosen sufficiently large. More details on the choice of these parameters will be given below in the course of the construction.

Using the  factorization method it is standard to derive  regularity of the stochastic convolution $z$ which solves the linear equation \eqref{linear} on $(\Omega, \mathcal{F},(\mathcal{F}_{t})_{t\geq0},\mathbf{P})$.  In particular, the following result follows from \cite[Theorem 5.14]{DPZ92} together with the Kolmogorov continuity criterion.

\bp\label{fe z}
	Suppose that $\mathrm{Tr}(GG^*)<\infty$ for some $\sigma>0$. Then for all $\delta\in (0,1)$ and $T>0$
	$$
	\E^{\mathbf{P}}\left[\|z\|_{C_{T}H^{1-\delta}}+\|z\|_{C_T^{1/2-\delta}L^2}\right]<\infty.
	$$
	
\ep

By the Sobolev embedding we know that $\|f\|_{L^\infty}\leq C_S\|f\|_{H^{\frac{3+\sigma}{2}}}$ for $\sigma>0$ and some constant $C_{S}\geq1$.
We fix   $0<\delta<1/12$ and define
\begin{equation}\label{stopping time}
\aligned
\mathfrak{t}:=&\inf\{t\geq0, \|z(t)\|_{H^{1-\delta}}\geq 1/ C_S\}\wedge \inf\{t\geq0,\|z\|_{C_t^{1/2-2\delta}L^2}\geq 1/C_S\}
\\&\qquad\wedge \inf\{t\geq0,\|z(t)\|_{L^2}\geq a^{\beta b-b^2\beta}/\sqrt{12}\}\wedge 1.
\endaligned
\end{equation}
Then $\mathfrak{t}$  is $\mathbf{P}$-a.s. strictly positive stopping time.
Moreover, for  $t\in[0, \mathfrak{t}]$ we have
\begin{equation}\label{z}
\| z_q(t)\|_{L^\infty}\leq \lambda_{q+1}^{\frac\alpha8},\quad \| z_q(t)\|_{L^2}\leq a^{\beta b-b^2\beta}/\sqrt{12}, \quad\|\nabla z_q(t)\|_{L^\infty}\leq \lambda_{q+1}^{\frac\alpha4}, \quad \|z_{q}\|_{C_t^{\frac{1}{2}-2\delta}L^\infty}\leq \lambda_{q+1}^{\frac\alpha4}.
\end{equation}

To handle the mollification around $t=0$ needed in the convex integration below, we require that the iterative equation \eqref{induction} is satisfied also for some negative times, namely, it holds on an interval $[t_{q},\mathfrak{t}]$ for certain $t_{q}<0$. More precisely, we let $t_q:=-2+\sum_{1\leq r\leq q}\delta_r^{1/2}$. To this end, we assume $\sum_{r\geq 1}\delta_{r}^{1/2}\leq \sum_{r\geq1}a^{b\beta-rb\beta}=\frac{1}{1-a^{-\beta b}}\leq2$ which boils down  to
\begin{equation}\label{aaa}
a^{\beta b}\geq 2
\end{equation}
assumed from now on.
We introduce the following norms for  and $p\in [1,\infty]$
\begin{align*}
	\|f\|_{C_{t,x,q}^N}:=\|f\|_{C^N_{[t_q,t],x}},\quad \|f\|_{C_{t,q}L^p}:=\|f\|_{C_{[t_q,t]}L^p}.
\end{align*}

Recall that $e:[0,1]\to [\underline{e},\infty)$ with $\underline{e}\geq 4$, $\|e\|_{C^{0}}\leq \bar e $ and $\|e'\|_{C^{0}}\leq \tilde e$ is the given energy. We also extend $e$ and $z$ to $[-2,0]$ by setting $e(t)=e(0), z(t)=z(0)$ for $t\in [-2,0]$.
By induction on $q$ we  assume the following bounds for the iterations $(v_q,\mathring{R}_q)$: if $t\in[{t_q}, \mathfrak{t}]$ then
\begin{equation}\label{inductionv}
\aligned
\|v_q\|_{{C_{t,q}}L^2}&\leq M_0(1+\sum_{1\leq r\leq q}\delta_{r}^{1/2}){\bar e}^{1/2}\leq 3M_0 {\bar e}^{1/2} ,
\\ \|v_q\|_{{C^1_{t,x,q}}}&\leq \lambda_q^4{\bar e}^{1/2},
\\
\|\mathring{R}_q\|_{{C_{t,q}}L^1}&\leq \frac{1}{48}\delta_{q+2}e(t).
\endaligned
\end{equation}
Here, we defined $\sum_{1\leq r\leq 0}:=0$ and  $M_0\geq 1$ is a   universal constant given below.
In this iterative assumption, we have  intentionally chosen $\bar e$ in the first two bounds and the point evaluation at $e(t)$ in the last bound. The reason for this will become apparent in Section~\ref{s:energy} where we control the error of the energy.

In addition, we assume that the energy of the iterations $v_{q}+z_{q}$ gradually approaches the given energy $e$. More precisely, if $t\in[{t_q}, \mathfrak{t}]$ then we require
\begin{equation}\label{inductionve}
\frac{3}{4}\delta_{q+1} e(t)\leq  e(t)-\|(v_q+{z_q})(t)\|_{L^2}^2\leq \frac{5}{4}\delta_{q+1}  e(t).
\end{equation}
Under the condition \eqref{aaa} we obtain  in particular that $a^{\beta b-b^2\beta}/\sqrt{12}\leq 1$ which by the definition of the stopping time implies for $t\in[0,\mathfrak{t}]$ that $\|z_{q}(t)\|_{L^{2}}\leq\|z(t)\|_{L^{2}}\leq 1$. We start the iteration off from $q=0$, $v_0\equiv0$ and $\mathring{R}_{0}=z_{0}\mathring{\otimes} z_{0}$ on $[-2,\mathfrak{t}]$. Using $\delta_1=1$, \eqref{inductionv} and \eqref{inductionve} leads to the condition
\begin{equation}\label{c:1}
\frac{3}{4}e(t)\leq e(t)-1\leq \frac{5}{4}e(t),
\end{equation}
and for $t\in [-2,\mathfrak{t}]$
\begin{align}\label{c:2}
\|\mathring{R}_0(t)\|_{L^1}\leq \frac1{12}a^{2\beta b-2b^2\beta}\leq \frac1{48}a^{2\beta b-2b^2\beta}\underline{e},
\end{align}
i.e. we require
\begin{align}\label{c:3}
e(t)\geq \underline{e}\geq 4.
\end{align}

The main ingredient in the proof of Theorem \ref{Main results1} is the following iteration.

\bp\label{main iteration}
There exists a choice of parameters $a,b,\beta$ such that the following holds true: If for some $q\in\N_{0}$
a process  $(v_q,\mathring{R}_q)$ is an $(\mathcal{F}_t)_{t\geq0}$-adapted solution to \eqref{induction} on $[t_{q},\mathfrak{t}]$ satisfying \eqref{inductionv} and \eqref{inductionve}, then
  there exists an $(\mathcal{F}_t)_{t\geq0}$-adapted process
	$(v_{q+1},\mathring{R}_{q+1})$ which solves \eqref{induction} on $[t_{q+1},\mathfrak{t}]$, obeys \eqref{inductionv} and \eqref{inductionve} at level $q+1$ and  for $t\in[t_{q+1},\mathfrak{t}]$ we have
	\begin{equation}\label{iteration}
	\|v_{q+1}(t)-v_q(t)\|_{L^2}\leq M_0 \bar{e}^{1/2}\delta_{q+1}^{1/2},
	\end{equation}
	for some universal constant $M_{0}$.
	Furthermore, if $(v_q,\mathring{R}_q)(t)$ is deterministic for $t\in [t_q,0]$, so is $(v_{q+1},\mathring{R}_{q+1})(t)$ for $t\in [t_{q+1},0]$.
\ep

We present the  proof of  Proposition \ref{main iteration}  in Section \ref{ss:it} below. Based on   Proposition \ref{main iteration} we may  proceed with the proof of Theorem \ref{Main results1}.

\begin{proof}[Proof of Theorem \ref{Main results1}]

	Starting from $(v_0,\mathring{R}_0)=(0,{z_{0}\mathring\otimes z_{0}})$,  Proposition \ref{main iteration} yields a sequence $(v_q, \mathring{R}_q)$ satisfying (\ref{inductionv}), \eqref{inductionve} and (\ref{iteration}).
	By interpolation we deduce for $\gamma\in (0,\frac{\beta}{4+\beta})$, $t\in [0,\mathfrak{t}]$
	\begin{equation*}
	\sum_{q\geq0}\|v_{q+1}(t)-v_q(t)\|_{H^{\gamma}}\lesssim \sum_{q\geq0}\|v_{q+1}(t)-v_q(t)\|_{L^2}^{1-\gamma}\|v_{q+1}(t)-v_q(t)\|_{H^1}^{\gamma}\lesssim \bar{e}^{1/2}\sum_{q\geq0}\delta_{q+1}^{\frac{1-\gamma}{2}}\lambda_{q+1}^{4\gamma}\lesssim \bar{e}^{1/2}.
	\end{equation*}
As a consequence, a limit $v=\lim_{q\rightarrow\infty}v_q$ exists and lies in $C([0,\mathfrak{t}],H^{\gamma})$.  Since $v_q$ is $(\mathcal{F}_t)_{t\geq0}$-adapted for every $q\in\mathbb{N}_{0}$, the limit
	$v$ is $(\mathcal{F}_t)_{t\geq0}$-adapted as well.
	Furthermore, it follows from the third estimate in \eqref{inductionv} that $\lim_{q\rightarrow\infty}\mathring{R}_q=0$ in $C([0,\mathfrak{t}];L^1)$ {and $\lim_{q\rightarrow\infty}z_q=z$ in $C([0,\mathfrak{t}];L^2)$}. Thus $v$ is an analytically  weak solution to (\ref{nonlinear}). In addition, there exists a deterministic constant $c_{2}$ such that
	\begin{equation}\label{eq:vvv}
	\|v(t)\|_{H^{\gamma}}\leq c_2
	\end{equation}
	holds true for all $t\in[0,\mathfrak{t}]$. Hence letting $u=v+z$ we obtain an $(\mathcal{F}_{t})_{t\geq 0}$-adapted analytically weak solution to \eqref{1} satisfying \eqref{eq:vv}.

	Finally,
by taking the limit in \eqref{inductionve} we deduce the desired energy equality \eqref{eq:env}.
The last claim of the theorem follows from the construction in Section \ref{ss:it}.  Indeed, we find that the value of $v(t)$ for a given $t\in[0,\mathfrak{t}]$ is determined by the value of $v_{q+1}(t)$ at each step. Due to the (one-sided) mollification in time, each iteration $v_{q+1}(t)$ in turn depends on the following objects on the full time interval $[0,t]$: the stochastic convolution $z$, the previous iteration $v_{q}$, the previous stress $\mathring{R}_{q}$ and the energy $e$. As a consequence, if two energies $e_{1},$ $e_{2}$ are as in the statement of the theorem, i.e.  have the same $\underline{e},$ $\bar{e}$ and $\tilde e$ and  coincide on $[0,t]$ then the corresponding  solutions $u_{1}$, $u_{2}$ coincide on $[0,t\wedge\mathfrak{t}]$ as well.
\end{proof}

\subsection{Proof of Proposition \ref{main iteration}}\label{ss:it}
	
	\subsubsection{Choice of  parameters}\label{s:c}
	
	In the sequel, additional parameters will be indispensable and their value has to be carefully chosen in order to respect all the compatibility conditions appearing in the estimations below. First, for a sufficiently small  $\alpha\in (0,1)$ to be chosen below, we let $\ell\in (0,1)$ be  a small parameter  satisfying
	\begin{equation}\label{ell}
	\ell \lambda_q^4\leq \lambda_{q+1}^{-\alpha},\quad \ell^{-1}\leq \lambda_{q+1}^{2\alpha},\quad \bar e\leq \ell^{-1}.
	\end{equation}
	In particular, we define
	\begin{equation}\label{ell1}
	\ell:=\lambda_{q+1}^{-\frac{3\alpha}{2}}\lambda_q^{-2}.
	\end{equation}
	Combined with \eqref{c:3} we need $\alpha b>4$ and
	 \begin{equation}\label{aaa2}
	   4 \leq\underline{e}\leq \bar e\leq a^{\frac32b\alpha+2},
	   \end{equation}
	 which can be obtained by choosing $a$ large enough.	
	
	In the sequel, we  use the following bounds
	$$
	\alpha>{18}\beta b^2,\quad \frac17-50\alpha>2\beta b^2,\quad \frac17-160\alpha>2\beta b,
	$$
	which can be obtained by choosing $\alpha$ small  such that
	$\frac17-160\alpha>\alpha,$
	and choosing $b\in {7\cdot 8}\mathbb{N}$ large enough such that
	{$\alpha b\in 8\mN$} and finally choosing $\beta$ small such that
	$\alpha>{18}\beta b^2$. We also need
	$$M_0\lambda_{q+1}^{-\frac\alpha8\cdot \frac{11}{12}}\bar{e}\leq \frac{1}{10}\lambda_{q+1}^{-2b^2\beta}, $$
	i.e.
		\begin{equation}\label{eq:aaa20}
		10M_0\bar{e}\leq \lambda_{1}^{\frac{11}{12}\cdot\frac18\alpha-2b^2\beta},
		\end{equation}
	which can be obtained by choosing $a$ large enough.  Various estimates of this form are needed for the final  control the new stress $\mathring{R}_{q+1}$: the right hand side is  given by $\lambda_{q+1}^{-2\beta b^2}$ and the left hand side is typically given by an expression of the form $\lambda_{q+1}^{{50}\alpha-\frac17}$. Hence, we shall choose $\alpha$ small first and $b$ large, then $\beta$  small enough. The last free parameter is $a$ which satisfies the lower bounds given through \eqref{aaa}, \eqref{aaa2}, \eqref{eq:aaa20}. In the sequel, we increase $a$ in order to absorb various implicit and universal constants.

	\subsubsection{Mollification}\label{s:p}
	
	We intend to replace $v_q$ by a mollified velocity field $v_\ell$. To this end, let $\{\phi_\varepsilon\}_{\varepsilon>0}$ be a family of standard mollifiers on $\mathbb{R}^3$, and let $\{\varphi_\varepsilon\}_{\varepsilon>0}$ be a family of  standard mollifiers with support in $(0,1)$. The one sided mollifier here is used in order to preserve adaptedness. We define a mollification of $v_q$, $\mathring{R}_q$ and $z$ in space and time by convolution as follows
	$$v_\ell=(v_q*_x\phi_\ell)*_t\varphi_\ell,\qquad
	\mathring{R}_\ell=(\mathring{R}_q*_x\phi_\ell)*_t\varphi_\ell,\qquad
	z_\ell=({z_q}*_x\phi_\ell)*_t\varphi_\ell,$$
	where $\phi_\ell=\frac{1}{\ell^3}\phi(\frac{\cdot}{\ell})$ and $\varphi_\ell=\frac{1}{\ell}\varphi(\frac{\cdot}{\ell})$.
	Since the mollifier $\varphi_\ell$ is supported on $(0,1)$, it is easy to see that $z_\ell$ is $(\mathcal{F}_t)_{t\geq0}$-adapted and so are $v_\ell$ and $\mathring{R}_\ell$. As $\ell\leq \delta_{q+1}^{1/2}$, if $v_q(t), \mathring{R}_q(t)$ are deterministic for $t\in [t_q,0]$, so are $v_\ell(t)$ and $\mathring{R}_\ell(t), \partial_t\mathring{R}_\ell(t)$ for $t\in [t_{q+1},0]$. Moreover, $z(t)=0$ for $t\in [-2,0]$ implies that $z_\ell(t)$ as well as ${R}_{\mathrm{com}}(t)$ on $t\in [t_{q+1},0]$ given below is deterministic as well.
	
	Since (\ref{induction}) holds  on $[t_{q},\mathfrak{t}]$, it follows that $(v_\ell,\mathring{R}_\ell)$ satisfies on $[t_{q+1},\mathfrak t]$
	\begin{equation}\label{mollification}
	\aligned
	\partial_tv_\ell -\Delta v_\ell+\div((v_\ell+z_\ell)\otimes (v_\ell+z_\ell))+\nabla p_\ell&=\div (\mathring{R}_\ell+R_{\textrm{com}})
	\\\div v_\ell&=0,
	\endaligned
	\end{equation}
	where
	\begin{equation*}
	R_{\textrm{com}}=(v_\ell+z_\ell)\mathring{\otimes}(v_\ell+z_\ell)-((v_q+{z_q})\mathring{\otimes}(v_q+{z_q}))*_x\phi_\ell*_t\varphi_\ell,
	\end{equation*}
	\begin{equation*}
	p_\ell=(p_q*_x\phi_\ell)*_t\varphi_\ell-\frac{1}{3}\big(|v_\ell+z_\ell|^2-(|v_q+{z_q}|^2*_x\phi_\ell)*_t\varphi_\ell\big).
	\end{equation*}

	In view of  (\ref{inductionv})  we obtain for $t\in[t_{q+1},\mathfrak t]$
	\begin{equation}\label{error}
	\|v_q-v_\ell\|_{C_{t,q+1}L^2}\lesssim\|v_q-v_\ell\|_{C^0_{t,x,q+1}}\lesssim \ell\|v_q\|_{C^1_{t,x,q}}\leq \ell\lambda_q^4\bar {e}^{1/2} \leq  \lambda_{q+1}^{-\alpha}\bar {e}^{1/2}\leq \frac{1}{4} \delta_{q+1}^{1/2}\bar {e}^{1/2},
	\end{equation}
	where we  used the fact that $\alpha>\beta$ and we  chose $a$ large enough in order to absorb the implicit constant.
	In addition, it holds for $t\in[t_{q+1}, \mathfrak{t}]$
\begin{equation}\label{eq:vl}
	\|v_\ell\|_{C_{t,q+1}L^2}\leq \|v_q\|_{C_{t,q}L^2}\leq   M_0(1+\sum_{1\leq r\leq q}\delta_{r}^{1/2})\bar {e}^{1/2},
	\end{equation}
	and for $N\geq1$
\begin{equation}\label{eq:vl2}
	\|v_\ell\|_{C^N_{t,x,q+1}}\lesssim \ell^{-N+1}\|v_q\|_{C^1_{t,x,q}}\leq \ell^{-N+1}\lambda_q^4\bar{e}^{1/2} \leq  \ell^{-N}\lambda_{q+1}^{-\alpha}\bar{e}^{1/2},
	\end{equation}
	where we chose $a$ large enough to absorb the implicit constant.

\subsubsection{Construction of  $v_{q+1}$}\label{s:con}

Let us now proceed with the construction of the perturbation $w_{q+1}$ which then defines the next iteration by $v_{q+1}:=v_{\ell}+w_{q+1}$.
To this end, we employ  the intermittent jets introduced in \cite{BCV18} and presented in \cite[Section 7.4]{BV19}, which we recall in Appendix~\ref{s:B}. In particular, the building blocks $W_{(\xi)}=W_{\xi,r_\perp,r_\|,\lambda,\mu}$ for $\xi\in\Lambda$ are defined in (\ref{intermittent}) and the set $\Lambda$ is introduced in Lemma \ref{geometric}.
The necessary estimates are collected  in \eqref{bounds}.  We choose the following parameters
\begin{equation}\label{parameter}
\aligned
\lambda&=\lambda_{q+1},
\qquad
r_\|=\lambda_{q+1}^{-4/7},
\qquad r_\perp=r_\|^{-1/4}\lambda_{q+1}^{-1}=\lambda_{q+1}^{-6/7},
\qquad
\mu=\lambda_{q+1}r_\|r_\perp^{-1}=\lambda_{q+1}^{9/7}.
\endaligned
\end{equation}
It is required that $b$ is a multiple of $7$ to ensure that $\lambda_{q+1}r_\perp= a^{(b^{q+1})/7}\in\mathbb{N}$.

As the next step, we shall  define certain amplitude functions used in the definition of the perturbations $w_{q+1}$. This is where one of the main differences  compared to  \cite{HZZ19} lies and we follow the ideas of \cite{BMS20}. In particular, we shift and normalize $\mathring{R}_\ell$ by
$$\Id -\frac{\mathring{R}_\ell}{\rho},$$
with
$$\rho:=2\sqrt{\ell^2+|\mathring{R}_\ell|^2}+\gamma_\ell,\quad t\in [t_{q+1},\mathfrak t],$$
$$\gamma_q(t):=\frac1{3\cdot (2\pi)^3} \Big[ e(t)(1-\delta_{q+2})-\|v_q(t)+{z_q}(t)\|_{L^2}^2\Big],\quad t\in [t_q,\mathfrak{t}],$$
and
$$ \gamma_\ell:=\gamma_q*_t\varphi_\ell,\quad t\in [t_{q+1},\mathfrak t].$$
The function $\gamma_{q}$ is used in order to pump energy into the system and  permits to achieve \eqref{inductionve}. We observe that  \eqref{aaa} and $b\geq 2$ implies in particular
$
\frac43\leq a^{2\beta b(b-1)},
$
i.e. $\frac34\delta_{q+1}\geq \delta_{q+2}$, it follows that $\gamma_{q}\geq 0$.
Since the mollifier $\varphi_\ell$ is supported on $\mR^+$ it follows that $\rho$ is $(\mathcal{F}_t)_{t\geq0}$-adapted.
As a consequence of  $\gamma_q\geq0$ we know $\rho\geq 2|\mathring{R}_\ell|$ and hence
$$ \Big|\Id-\frac{\mathring{R}_\ell}{\rho}-\Id\Big|\leq \frac12.$$

Note that since $\mathring{R}_\ell(t,x), \partial_t\mathring{R}_\ell(t,x)$  and $\gamma_\ell(t), \p_t\gamma_\ell(t)$ are deterministic for $t\in [t_{q+1},0]$, so are $\rho$ and $\partial_t\rho$. In the following we consider all the norms are on $[t_{q+1},t]$ and we omit the subscript $q+1$ if there is no danger of confusion.
Moreover, we use \eqref{inductionve} to obtain for any $p\in [1,\infty]$, $t\in [t_{q+1},\mathfrak{t}]$
\begin{equation}\label{rho}
\|\rho\|_{C_tL^p}\leq 2\ell(2\pi)^{3/p}+2\|\mathring{R}_\ell\|_{C_tL^p}+\frac12\delta_{q+1}\bar {e}.
\end{equation}
Furthermore, by mollification estimates, the embedding $W^{4,1}\subset L^\infty$ and \eqref{inductionv} we obtain for $N\geq0$
$t\in[t_{q+1}, \mathfrak{t}]$
$$
\|\mathring{R}_\ell\|_{C^N_{t,x}}\lesssim \ell^{-4-N}\delta_{q+2}\bar{e},
$$
which in particular leads to
\begin{equation}\label{rho0}
\|\rho\|_{C^{0}_{t,x}}\lesssim \ell+\ell^{-4}\delta_{q+2}\bar {e}+\delta_{q+1}\bar {e}\lesssim \ell^{-4}\delta_{q+1}\bar{e},
\end{equation}
by using $\ell<\delta_{q+1}\bar {e}$.

As the next step, we aim at estimating the $C^{N}_{t,x}$-norm of $\rho$ for $N\in\mathbb{N}$. To this end, we first apply  the chain rule  \cite[Proposition C.1]{BDLIS16} to the function $\Psi(z)=\sqrt{\ell^{2}+z^{2}}$, $|D^{m}\Psi(z)|\lesssim \ell^{-m+1}$ to  obtain 
	$$
	\left\| \sqrt{\ell^2 + | \mathring{R}_\ell |^2} \right\|_{C^N_{t,x}} \lesssim \left\| \sqrt{\ell^2 + | \mathring{R}_\ell |^2} \right\|_{C^{0}_{t,x}}+\|D \Psi \|_{C^{0}}
	\| \mathring{R}_\ell \|_{C^N_{t,x}} + \| D\Psi \|_{C^{N-1}} \|
	\mathring{R}_\ell \|^N_{C^1_{t,x}} $$
	$$
	 \lesssim \ell^{- 4 - N} \delta_{q + 2} \bar{e} + \ell^{- N + 1}
	\ell^{(- 4 - 1) N} \delta_{q + 2} \bar{e}^N. $$
Then we apply  $\bar{e}\leq \ell^{-1}$ and deduce for $N\geq1$
\begin{equation}\label{rhoN}
\aligned
\|\rho\|_{C^N_{t,x}}&\lesssim
	\left\| \sqrt{\ell^2 + | \mathring{R}_\ell |^2} \right\|_{C^N_{t,x}}+\|\gamma_{\ell}\|_{C^{N}_{t}}\\
	&\lesssim (\ell^{- 4 - N} + \ell^{- N + 1} \ell^{- 5 N} \ell^{- N + 1})
	\delta_{q + 2} \bar{e} +\frac12\ell^{-N}\delta_{q+1}\bar {e} \lesssim
	\ell^{2 - 7 N} \delta_{q + 1} \bar {e},
\endaligned
\end{equation}
where we used the fact that $2-7N\leq-4-N$ for $N\geq1$.

Now, we define the amplitude functions
\begin{equation}\label{amplitudes}a_{(\xi)}(\omega,t,x):=a_{\xi,q+1}(\omega,t,x):=\rho(\omega,t,x)^{1/2}\gamma_\xi\left(\Id
-\frac{\mathring{R}_\ell(\omega,t,x)}{\rho(\omega,t,x)}\right)(2\pi)^{-\frac{3}{4}},\end{equation}
where $\gamma_\xi$ is introduced in  Lemma \ref{geometric}. Since $\rho$ and $\mathring{R}_\ell$ are $(\mathcal{F}_t)_{t\geq0}$-adapted, we know that also $a_{(\xi)}$ is $(\mathcal{F}_t)_{t\geq0}$-adapted. If $\mathring{R}_\ell(t,x), \partial_t\mathring{R}_\ell(t,x)$, and $\rho(t,x), \p_t\rho(t,x)$ are deterministic for $t\in [t_{q+1},0]$, so are $a_{(\xi)}(t,x)$ and $\partial_ta_{(\xi)}( t,x)$ for $t\in [t_{q+1},0]$.
By  (\ref{geometric equality}) we have
\begin{equation}\label{cancellation}(2\pi)^{-\frac{3}{2}}\sum_{\xi\in\Lambda}a_{(\xi)}^2\int_{\mathbb{T}^3}W_{(\xi)}\otimes W_{(\xi)}\dif x=\rho \Id-\mathring{R}_\ell,
\end{equation}
and using (\ref{rho}) for $t\in[t_{q+1}, \mathfrak{t}]$
\begin{equation}\label{estimate a}
\begin{aligned}
\|a_{(\xi)}\|_{C_tL^2}&\leq \|\rho\|_{C_tL^1}^{1/2}\|\gamma_\xi\|_{C^0(B_{1/2}(\Id))}\leq \frac{M}{8|\Lambda|(1+8\pi^{3})^{1/2}}\left(2\ell (2\pi)^3+\frac{2}{48}\delta_{q+2}\bar {e}+\frac{1}2 \delta_{q+1}\bar {e}\right)^{1/2}\\
&\leq\frac{M}{4|\Lambda|}\bar{e}^{1/2}\delta_{q+1}^{1/2},
\end{aligned}
\end{equation}
where  we use $2\ell\leq\delta_{q+1}\bar {e}/2$ and  $M$ denotes the universal constant from Lemma~\ref{geometric}.

Let us now estimate the $C^{N}_{t,x}$-norm of $a_{(\xi)}$. By Leibniz rule, we get
\begin{equation}\label{aa}
\|a_{(\xi)}\|_{C^{N}_{t,x}}\lesssim\sum_{m=0}^N\|\rho^{\frac12}\|_{C^m_{t,x}}\left\|\gamma_{\xi}\left(\Id
-\frac{\mathring{R}_\ell}{\rho}\right)\right\|_{C^{N-m}_{t,x}}
\end{equation}
and estimate each norm separately.
First, by \eqref{rho0}
\begin{equation*}
\| \rho^{1 / 2} \|_{C^0_{t,x}} \lesssim \ell^{-2}\delta_{q+1}^{1/2}\bar {e}^{1/2},
\end{equation*}
and by Lemma~\ref{geometric}
$$
\left\|\gamma_{\xi}\left(\Id
-\frac{\mathring{R}_\ell}{\rho}\right)\right\|_{C^{0}_{t,x}}\lesssim 1.
$$
Second, applying \cite[Proposition C.1]{BDLIS16} to the function
	$
	 \Psi (z) = z^{1 / 2},$ $  | D^{m}\Psi(z) | \lesssim |z|^{1 / 2 - m}, $ for $m=1,\dots, N$, and using \eqref{rhoN} and $\rho\geq \ell$ we obtain 
	\begin{equation}\label{eq:rho12}
	\begin{aligned}
	 \| \rho^{1 / 2} \|_{C^m_{t,x}} &\lesssim \| \rho^{1 / 2} \|_{C^0_{t,x}}+\ell^{- 1 / 2} \|\rho\|_{C^{m}_{t,x}} + \ell^{1 / 2 -
		m} \|\rho\|_{C^{1}_{t,x}}^m\\
		& \lesssim \ell^{-2}\delta_{q+1}^{1/2}\bar {e}^{1/2}+\ell^{- 1 / 2} \ell^{2 - 7 m} \delta_{q + 1} \bar {e} +
	\ell^{1 / 2 - m} \ell^{- 5 m} \delta^m_{q + 1} \bar {e}^m\\
	&\lesssim  \ell^{1-7m}\delta^{1/2}_{q + 1} \bar {e}^{1/2},
	\end{aligned}
	\end{equation}
where we used \eqref{ell}.
	
We proceed with a bound for $\left\|\gamma_{\xi}\left(\Id
-\frac{\mathring{R}_\ell}{\rho}\right)\right\|_{C^{N-m}_{t,x}}$ for $m=0,\dots,N-1$. Keeping \cite[Proposition C.1]{BDLIS16} as well as Lemma~\ref{geometric} in mind, we need to estimate
	\begin{equation}\label{gamma}
	 \left\| \frac{\mathring{R}_\ell}{\rho} \right\|_{C^{N - m}_{t,x}} + \left\| \frac{\nabla_{t,x}\mathring{R}_\ell}{\rho}
	\right\|_{C^0_{t,x}}^{N - m} + \left\| \frac{\mathring{R}_\ell}{\rho^{2}} \right\|_{C^0_{t,x}}^{N - m} \|
	\rho \|_{C^1_{t,x}}^{N - m} .
	\end{equation}
	We use $\rho\geq \ell$ to have
	$$ \left\| \frac{\nabla_{t,x}\mathring{R}_\ell}{\rho} \right\|_{C^0_{t,x}}^{N - m} \lesssim \ell^{- (N - m)}
	\ell^{(- 4 - 1) (N - m)} \delta^{N - m}_{q + 2} \bar {e}^{N - m} \lesssim \ell^{-
		7 (N - m)} , $$
	and	in view of $|\frac{\mathring{R}_\ell}{\rho}|\leq 1$
	$$ \left\| \frac{\mathring{R}_\ell}{\rho^2} \right\|_{C^0_{t,x}}^{N - m} \lesssim \left\| \frac{1}{\rho} \right\|_{C^0_{t,x}}^{N - m}\lesssim \ell^{-(N-m)},$$
and by \eqref{rhoN} and $\bar{e}\leq \ell^{-1}$
$$ \| \rho \|_{C^1_{t,x}}^{N - m} \lesssim \ell^{-5 (N - m)} \delta^{N -
		m}_{q + 1} \bar{e}^{N - m} \lesssim \ell^{- 6 (N - m)}. $$

To estimate the  first term in \eqref{gamma}, we write		
	\begin{equation}\label{RR}
	 \left\| \frac{\mathring{R}_\ell}{\rho} \right\|_{C^{N - m}_{t,x}} \lesssim \sum_{k = 0}^{N - m}
	\| \mathring{R}_\ell \|_{C^k_{t,x}} \left\| \frac{1}{\rho} \right\|_{C^{N - m - k}_{t,x}},
	\end{equation}
	where for $N-m-k=0$ we have
	$$
	\left\| \frac{1}{\rho} \right\|_{C^{0}_{t,x}} \lesssim \ell^{-1}
	$$
	and for $k=0,\dots,N-m-1$ using \eqref{rhoN} and $\bar{e}\leq \ell^{-1}$
	$$
	 \left\| \frac{1}{\rho} \right\|_{C^{N - m - k}_{t,x}} \lesssim \left\|\frac1\rho\right\|_{C^{0}_{t,x}}+\ell^{- 2} \| \rho
	\|_{C^{N - m - k}_{t,x}} + \ell^{- (N - m - k){-1}} \| \rho \|_{C^1_{t,x}}^{N - m - k} $$
	$$
	 \lesssim \ell^{- 2} \ell^{{2 - 7 (N - m - k)-1}} + \ell^{- (N - m - k)-1}
	\ell^{- 5 (N - m - k)} \delta_{q + 1}^{N - m - k} \bar {e}^{N - m - k}
	 \lesssim {\ell^{-1}}\ell^{ - 7 (N - m - k)}. $$
	Altogether, we therefore bound \eqref{RR} as
	\begin{equation}\label{RR2}
	\begin{aligned}
	\left\| \frac{\mathring{R}_\ell}{\rho} \right\|_{C^{N - m}_{t,x}} &\lesssim \sum_{k = 0}^{N - m-1} \ell^{-5-k}\ell^{-7(N-m-k){-1}}+ \ell^{-5-(N-m)}\ell^{-1} \\
	&\lesssim \ell^{-6-7(N-m)}+\ell^{-6-(N-m)}\lesssim \ell^{-{6}-7(N-m)}.
	\end{aligned}
	\end{equation}

Finally, plugging \eqref{RR2}, and the other bounds into  \eqref{gamma} leads to
$$
\left\|\gamma_{\xi}\left(\Id
-\frac{\mathring{R}_\ell}{\rho}\right)\right\|_{C^{N-m}_{t,x}}\lesssim \ell^{-6-7(N-m)} + \ell^{-7(N-m)}\lesssim \ell^{-6-7(N-m)}.
$$
Combining this
with the bounds for $\rho^{1/2}$ above and plugging into \eqref{aa} yields for $N\in\mathbb{N}$
\begin{equation}\label{estimate aN}
\begin{aligned}
\|a_{(\xi)}\|_{C^N_{t,x}}& \lesssim \left(\ell^{-2}\ell^{-6-7N}+\sum^{N-1}_{m=1}\ell^{1-7m}\ell^{-6-7(N-m)}+\ell^{1-7N}\right)\delta_{q+1}^{1/2}\bar {e}^{1/2}\\
&\lesssim \ell^{-8-7N}\delta_{q+1}^{1/2}\bar{e}^{1/2},
\end{aligned}
\end{equation}
where the final bound is also valid for $N=0$.

\medskip

With these preparations in hand,  we define the principal part $w_{q+1}^{(p)}$ of the perturbation $w_{q+1}$ as
\begin{equation}\label{principle}
w_{q+1}^{(p)}:=\sum_{\xi\in\Lambda} a_{(\xi)}W_{(\xi)}.
\end{equation}
If $\mathring{R}_\ell(t,x), \partial_t\mathring{R}_\ell(t,x)$ are deterministic for $t\in [t_{q+1},0]$, so are $w_{q+1}^{(p)}$ and $\partial_tw_{q+1}^{(p)}$.
Since the coefficients $a_{(\xi)}$ are $(\mathcal{F}_t)_{t\geq0}$-adapted and $W_{(\xi)}$ is a deterministic function we deduce that
$w_{q+1}^{(p)}$ is also $(\mathcal{F}_t)_{t\geq0}$-adapted.
Moreover, according to (\ref{cancellation}) and (\ref{Wxi}) it follows that
\begin{equation}\label{can}w_{q+1}^{(p)}\otimes w_{q+1}^{(p)}+\mathring{R}_\ell=\sum_{\xi\in \Lambda}a_{(\xi)}^2 \mathbb{P}_{\neq0}(W_{(\xi)}\otimes W_{(\xi)})+\rho \Id,
\end{equation}
where we use the notation $\mathbb{P}_{\neq0}f:=f-\mathcal{F}f(0)$.

We also define the incompressibility corrector by
\begin{equation}\label{incompressiblity}
w_{q+1}^{(c)}:=\sum_{\xi\in \Lambda}\textrm{curl}(\nabla a_{(\xi)}\times V_{(\xi)})+\nabla a_{(\xi)}\times \textrm{curl}V_{(\xi)}+a_{(\xi)}W_{(\xi)}^{(c)},\end{equation}
with $W_{(\xi)}^{(c)}$ and $V_{(\xi)}$ being given in (\ref{corrector}).
Since $a_{(\xi)}$ is $(\mathcal{F}_t)_{t\geq0}$-adapted and $W_{(\xi)}, W_{(\xi)}^{(c)}$ and $V_{(\xi)}$ are  deterministic it follows that
$w_{q+1}^{(c)}$ is also $(\mathcal{F}_t)_{t\geq0}$-adapted. If $a_{(\xi)}(t,x), \partial_ta_{(\xi)}(t,x)$ are deterministic for $t\in [t_{q+1},0]$, so are $w_{q+1}^{(c)}$ and $\partial_tw_{q+1}^{(c)}$.
By a direct computation we deduce that
\begin{equation*}
w_{q+1}^{(p)}+w_{q+1}^{(c)}=\sum_{\xi\in\Lambda}\textrm{curl}\,\textrm{curl}(a_{(\xi)}V_{(\xi)}),
\end{equation*}
hence
\begin{equation*}\div(w_{q+1}^{(p)}+w_{q+1}^{(c)})=0.\end{equation*}
Next, we introduce the temporal corrector
\begin{equation}\label{temporal}w_{q+1}^{(t)}:=-\frac{1}{\mu}\sum_{\xi\in \Lambda}\mathbb{P}\mathbb{P}_{\neq0}\left(a_{(\xi)}^2\phi_{(\xi)}^2\psi_{(\xi)}^2\xi\right),\end{equation}
where $\mathbb{P}$ is the Helmholtz projection. If $a_{(\xi)}(t,x), \partial_ta_{(\xi)}(t,x)$ are deterministic for $t\in [t_{q+1},0]$, so is $w_{q+1}^{(t)}$.  Similarly as above, $w_{q+1}^{(t)}$ is $(\mathcal{F}_t)_{t\geq0}$-adapted and by a direct computation (see \cite[(7.20)]{BV19})  we obtain
\begin{equation}\label{equation for temporal}
\aligned
&\partial_t w_{q+1}^{(t)}+\sum_{\xi\in\Lambda}\mathbb{P}_{\neq0}\left(a_{(\xi)}^2\div(W_{(\xi)}\otimes W_{(\xi)})\right)
\\
&\qquad= -\frac{1}{\mu}\sum_{\xi\in\Lambda}\mathbb{P}\mathbb{P}_{\neq0}\partial_t\left(a_{(\xi)}^2\phi_{(\xi)}^2\psi_{(\xi)}^2\xi\right)
+\frac{1}{\mu}\sum_{\xi\in\Lambda}\mathbb{P}_{\neq0}\left( a^2_{(\xi)}\partial_t(\phi^2_{(\xi)}\psi^2_{(\xi)}\xi)\right)
\\&\qquad= (\Id-\mathbb{P})\frac{1}{\mu}\sum_{\xi\in\Lambda}\mathbb{P}_{\neq0}\partial_t\left(a_{(\xi)}^2\phi_{(\xi)}^2\psi_{(\xi)}^2\xi\right)
-\frac{1}{\mu}\sum_{\xi\in\Lambda}\mathbb{P}_{\neq0}\left(\partial_t a^2_{(\xi)}(\phi^2_{(\xi)}\psi^2_{(\xi)}\xi)\right).
\endaligned
\end{equation}
Note that the first term on the right hand side can be viewed as a pressure term $\nabla p_{1}$.

Finally, the total perturbation $w_{q+1}$ is defined by
\begin{equation}\label{wq}w_{q+1}:=w_{q+1}^{(p)}+w_{q+1}^{(c)}+w_{q+1}^{(t)},\end{equation}
which is mean zero, divergence free and $(\mathcal{F}_t)_{t\geq0}$-adapted. If $a_{(\xi)}(t,x), \partial_ta_{(\xi)}(t,x)$ are deterministic for $t\in[t_{q+1},0]$, so is $w_{q+1}$. The new velocity $v_{q+1}$ is defined as
\begin{equation}\label{vq}
v_{q+1}:=v_\ell+w_{q+1}.
\end{equation}
Thus, it is also $(\mathcal{F}_t)_{t\geq0}$-adapted. If $\mathring{R}_q(t,x), v_q(t,x)$ are deterministic for $t\in[t_{q},0]$, so is $v_{q+1}(t,x)$ for $t\in [t_{q+1},0]$.

\subsubsection{Inductive estimates for $v_{q+1}$}
\label{sss:v}
Next, we verify the inductive estimates (\ref{inductionv}) on the level $q+1$ for $v$ and we prove (\ref{iteration}).
First, we  recall the following result from \cite[Lemma~7.4]{BV19}.

\bl\label{lem:Lp}
	Fix integers $N, \kappa\geq1$ and let $\zeta>1$ be such that
	\begin{equation*}
	\frac{2\pi \sqrt{3}\zeta}{\kappa}\leq\frac{1}{3}\quad \textrm{ and } \quad\zeta^4\frac{(2\pi \sqrt{3}\zeta)^N}{\kappa^N}\leq1.
	\end{equation*}
	Let $p\in \{1,2\}$ and let $f$ be a $\mathbb{T}^3$-periodic function such that there exists a constant $C_f>0$ such that
	$$\|D^jf\|_{L^p}\leq C_f\zeta^j,$$
	holds for all $0\leq j\leq N+4$. In addition, let $g$ be a $(\mathbb{T}/\kappa)^3$-periodic function. Then it holds that
	$$\|fg\|_{L^p}\lesssim C_f\|g\|_{L^p},$$
	where the implicit constant is universal.
\el

This result is applied  to bound $w_{q+1}^{(p)}$ in $L^{2}$ whereas for the other $L^{p}$-norms we use a different approach.  By (\ref{estimate a}) and (\ref{estimate aN}) we obtain for $t\in[t_{q+1}, \mathfrak{t}]$
\begin{equation*}
\|D^ja_{(\xi)}\|_{C_tL^2}\lesssim \frac{M\bar {e}^{1/2}}{4|\Lambda|}\delta_{q+1}^{1/2}\ell^{-15j},
\end{equation*}
which combined with Lemma \ref{lem:Lp} for $\zeta=\ell^{-15}\leq\lambda_{q+1}^{30\alpha}$
implies that for $t\in[t_{q+1}, \mathfrak{t}]$
\begin{equation}\label{estimate wqp}
\|w_{q+1}^{(p)}\|_{C_tL^2}\lesssim \sum_{\xi\in\Lambda}\frac{1}{4|\Lambda|}M\bar {e}^{1/2}\delta_{q+1}^{1/2}\|W_{(\xi)}\|_{C_tL^2}\leq \frac{M_0}{2}\bar {e}^{1/2}\delta_{q+1}^{1/2},
\end{equation}
where we used  the fact that due to (\ref{intermittent}) together with the normalizations \eqref{eq:phi}, \eqref{eq:psi} it holds  $\|W_{(\xi)}\|_{L^2}\simeq 1$ uniformly in all the involved parameters, and we use $\zeta^5<\kappa=\lambda_{q+1}r_\perp$ by $150\alpha<\frac17$. Here, we may  choose $M_0=cM\geq1$ with a universal constant $c$.

For a general $L^p$-norm we apply (\ref{bounds}) and (\ref{estimate aN}) to deduce for $t\in[t_{q+1}, \mathfrak{t}]$, $p\in(1,\infty)$
\begin{equation}\label{principle est1}
\aligned
\|w_{q+1}^{(p)}\|_{C_tL^p}&\lesssim \sum_{\xi\in \Lambda}\|a_{(\xi)}\|_{C^0_{t,x}}\|W_{(\xi)}\|_{C_tL^p}\lesssim \bar {e}^{1/2}\delta_{q+1}^{1/2}\ell^{-8}r_\perp^{2/p-1}r_\|^{1/p-1/2},
\endaligned
\end{equation}
\begin{equation}\label{correction est}
\aligned
\|w_{q+1}^{(c)}\|_{C_tL^p}&\lesssim\sum_{\xi\in \Lambda}\left(\|a_{(\xi)}\|_{C^0_{t,x}}\|W_{(\xi)}^{(c)}\|_{C_tL^p}+\|a_{(\xi)}\|_{C^2_{t,x}}\|V_{(\xi)}\|_{C_tW^{1,p}}\right)
\\&\lesssim \bar {e}^{1/2}\delta_{q+1}^{1/2}\ell^{-22}r_\perp^{2/p-1}r_\|^{1/p-1/2}\left(r_\perp r_\|^{-1}+\lambda_{q+1}^{-1}\right)\lesssim \bar {e}^{1/2}\delta_{q+1}^{1/2}\ell^{-22}r_\perp^{2/p}r_\|^{1/p-3/2},
\endaligned
\end{equation}
and
\begin{equation}\label{temporal est1}
\aligned
\|w_{q+1}^{(t)}\|_{C_tL^p}&\lesssim \mu^{-1}\sum_{\xi\in\Lambda}\|a_{(\xi)}\|_{C^0_{t,x}}^2\|\phi_{(\xi)}\|_{L^{2p}}^2\|\psi_{(\xi)}\|_{C_tL^{2p}}^2
\\
&\lesssim\delta_{q+1}\bar {e} \ell^{-16}r_\perp^{2/p-1}r_\|^{1/p-2}(\mu^{-1}r_\perp^{-1}r_\|)= \bar {e}\delta_{q+1}\ell^{-16}r_\perp^{2/p-1}r_\|^{1/p-2}\lambda_{q+1}^{-1}.
\endaligned
\end{equation}
We note that for $p=2$ \eqref{principle est1} provides a worse bound than \eqref{estimate wqp} which was based on Lemma \ref{lem:Lp}.
Hence, we obtain
\begin{equation}\label{corr temporal}
\aligned
&\|w_{q+1}^{(c)}\|_{C_tL^p}+\|w_{q+1}^{(t)}\|_{C_tL^p}\\
&\quad\lesssim  \bar {e}^{1/2}\delta_{q+1}^{1/2}\ell^{-2}r_\perp^{2/p-1}r_\|^{1/p-1/2}\left(\ell^{-20}r_\perp r_\|^{-1}+\bar {e}^{1/2}\delta_{q+1}^{1/2}\ell^{-14}r_\|^{-3/2}\lambda_{q+1}^{-1}\right)
\\ &\quad\lesssim \bar {e}^{1/2}\delta_{q+1}^{1/2}\ell^{-2}r_\perp^{2/p-1}r_\|^{1/p-1/2},
\endaligned
\end{equation}
where we use (\ref{ell}) and the fact that $\lambda_{q+1}^{29\alpha-\frac{1}{7}}<1$ by our choice of $\alpha$. The bound \eqref{corr temporal} will be used below in the estimation of the Reynolds stress.

Combining (\ref{estimate wqp}), (\ref{correction est}) and (\ref{temporal est1}) we obtain for $t\in[t_{q+1}, \mathfrak{t}]$
\begin{equation}\label{estimate wq}
\aligned
\|w_{q+1}\|_{C_tL^2}&\leq \bar {e}^{1/2}\delta_{q+1}^{1/2}\left(\frac{M_0}{2}+C\ell^{-22}r_\perp r_\|^{-1}+C\bar {e}^{1/2}\delta_{q+1}^{1/2}\ell^{-16}r_\|^{-3/2}\lambda_{q+1}^{-1}\right)
\\&\leq \bar {e}^{1/2}\delta_{q+1}^{1/2}\left(\frac{M_0}{2}+C\lambda_{q+1}^{44\alpha-2/7}+C\bar {e}^{1/2}\delta_{q+1}^{1/2}\lambda_{q+1}^{32\alpha-1/7}\right)
\leq \frac34M_{0}\bar {e}^{1/2}\delta_{q+1}^{1/2},
\endaligned
\end{equation}
where  we chose  $ \beta$ small enough and $a$ large enough such that
$$
C\lambda_{q+1}^{44\alpha-2/7}\leq 1/8,\quad\text{and}\quad C\delta_{q+1}^{1/2}\lambda_{q+1}^{33\alpha-1/7}\leq 1/8.
$$
The bound \eqref{estimate wq} can be directly combined with \eqref{eq:vl} and the definition of the velocity $v_{q+1}$ \eqref{vq} to deduce the first bound in \eqref{inductionv} on the level $q+1$. Indeed, for $t\in[t_{q+1}, \mathfrak{t}]$ we use \eqref{eq:vl} and  \eqref{estimate wq} to have
$$
\|v_{q+1}\|_{C_{t}L^{2}}\leq \|v_{\ell}\|_{C_{t}L^{2}}+\|w_{q+1}\|_{C_{t}L^{2}} \leq M_0\bar {e}^{1/2}(1+\sum_{1\leq r\leq q+1}\delta_{r}^{1/2}).
$$
In addition, \eqref{estimate wq} together with \eqref{error} yields for $t\in[t_{q+1}, \mathfrak{t}]$
$$
\|v_{q+1}-v_{q}\|_{C_{t}L^{2}}\leq \|w_{q+1}\|_{C_{t}L^{2}}+\|v_{\ell}-v_{q}\|_{C_{t}L^{2}}\leq M_0\bar {e}^{1/2} \delta_{q+1}^{1/2},
$$
hence \eqref{iteration} holds.

As the next step, we shall verify the second bound in \eqref{inductionv}.
Using (\ref{estimate aN}) and (\ref{bounds}) we have for $t\in[t_{q+1}, \mathfrak{t}]$
\begin{equation}\label{principle est2}
\aligned
\|w_{q+1}^{(p)}\|_{C^1_{t,x}}&\leq \sum_{\xi\in\Lambda}\|a_{(\xi)}\|_{C^1_{t,x}}\|W_{(\xi)}\|_{C^1_{t,x}}\\
&\lesssim \bar{e}^{1/2} \ell^{-15}r_\perp^{-1}r_\|^{-1/2}\lambda_{q+1}\left(1+\frac{r_\perp \mu}{r_\|}\right)\lesssim \bar{e}^{1/2} \ell^{-15}r_\perp^{-1}r_\|^{-1/2}\lambda_{q+1}^2,
\endaligned
\end{equation}
\begin{equation}\label{correction est2}
\aligned
\|w_{q+1}^{(c)}\|_{C^1_{t,x}}
&\lesssim  \sum_{\xi\in\Lambda}\left(
\|a_{(\xi)}\|_{C^3_{t,x}}(\|W_{(\xi)}^{(c)}\|_{C^1_{t,x}}+\|V_{(\xi)}\|_{C^1_{t}C^1_{x}}
+\|V_{(\xi)}\|_{C_tC^2_{x}})\right)\\
&\lesssim \bar {e}^{1/2}\ell^{-29}r_\|^{-3/2}\lambda_{q+1}\frac{r_\perp \mu}{r_\|}\lesssim \bar {e}^{1/2}\ell^{-29}r_\|^{-3/2}\lambda_{q+1}^2,
\endaligned
\end{equation}
and
\begin{equation}\label{temporal est2}
\aligned
\|w_{q+1}^{(t)}\|_{C^1_{t,x}}&\leq \frac{1}{\mu}\sum_{\xi\in\Lambda}[\|a^2_{(\xi)}
\phi^2_{(\xi)}\psi^2_{(\xi)}\|_{C_tW^{1+\alpha,p}}+\|a^2_{(\xi)}\phi^2_{(\xi)}\psi^2_{(\xi)}\|_{C^1_tW^{\alpha,p}}]
\\
&\leq\frac{1}{\mu}\sum_{\xi\in\Lambda}
\Big((\|a_{(\xi)}\|_{C^0_{t,x}}\|a_{(\xi)}\|_{C^{1+\alpha}_{t,x}}+\|a_{(\xi)}\|_{C^1_{t,x}}\|a_{(\xi)}\|_{C^{\alpha}_{t,x}})\\
&\qquad\Big[\|\phi_{(\xi)}\|_{L^{\infty}}\| \phi_{(\xi)}\|_{W^{1+\alpha,\infty}}\|\psi_{(\xi)}\|_{C_tL^{\infty}}^2+\|\phi_{(\xi)}\|_{L^{\infty}}^2\big(\|\psi_{(\xi)}\|_{C_tL^{\infty}}\|\psi_{(\xi)}\|_{C_tW^{1+\alpha,p}}\\&\quad+\|\psi_{(\xi)}\|_{C^1_tL^{\infty}}\|\psi_{(\xi)}\|_{C_{t}W^{\alpha,p}}
+\|\psi_{(\xi)}\|_{C_tL^{\infty}}\|\psi_{(\xi)}\|_{C^1_tW^{\alpha,p}}\big)\Big]\Big)
\\&\lesssim \bar {e} \ell^{-24}r_\perp^{-1}r_\|^{-2}\lambda_{q+1}^{1+\alpha}\lambda_{q+1}^{-1}\left(1+\frac{r_\perp \mu}{r_\|}\right)
\lesssim \bar {e}\ell^{-24}r_\perp^{-1}r_\|^{-2}\lambda_{q+1}^{1+\alpha},
\endaligned
\end{equation}
where we chose $p$ large enough and applied the Sobolev embedding in the first inequality. This was needed because $\mathbb{P}\mathbb{P}_{\neq0}$ is not a bounded operator on $C^0$. In the last inequality in \eqref{temporal est2}, we used interpolation and an extra $\lambda_{q+1}^\alpha$ appeared.
Combining   \eqref{eq:vl2} and \eqref{principle est2}, \eqref{correction est2}, \eqref{temporal est2} with (\ref{ell}) we obtain for $t\in[t_{q+1}, \mathfrak{t}]$
\begin{equation*}
\aligned
\|v_{q+1}\|_{C^1_{t,x}}&\leq \|v_\ell\|_{C^1_{t,x}}+\|w_{q+1}\|_{C^1_{t,x}}\\
&\leq \bar {e}^{1/2}\left(\lambda_{q+1}^\alpha+C\lambda_{q+1}^{30\alpha+22/7}+C\lambda_{q+1}^{58\alpha+20/7}+C\lambda_{q+1}^{50\alpha+3}\right)
\leq \bar {e}^{1/2}\lambda_{q+1}^4,
\endaligned
\end{equation*}
where we used  the fact that $\bar {e}^{1/2}\leq \ell^{-1/2} \leq\lambda_{q+1}^{\alpha}$. Thus, the second estimate in \eqref{inductionv} holds true on the level $q+1$.

We conclude this part with further estimates of the perturbations $w^{(p)}_{q+1}$, $w^{(c)}_{q+1}$ and $w^{(t)}_{q+1}$, which will be used below in order to bound the Reynolds stress $\mathring{R}_{q+1}$ and to establish the last estimate in \eqref{inductionv} on the level $q+1$.
By a similar approach as in \eqref{principle est1}, \eqref{correction est}, \eqref{temporal est1}, we derive  the following estimates: for $t\in[t_{q+1}, \mathfrak{t}]$ by using (\ref{ell}), (\ref{estimate aN}) and (\ref{bounds})
\begin{equation}\label{principle est22}
\aligned
\|w_{q+1}^{(p)}+w_{q+1}^{(c)}\|_{C_tW^{1,p}}&\leq\sum_{\xi\in\Lambda}
\|\textrm{curl\,}\textrm{curl}(a_{(\xi)}V_{(\xi)})\|_{C_tW^{1,p}}
\\
&\lesssim \sum_{\xi\in \Lambda}\Big(\|a_{(\xi)}\|_{C^3_{t,x}}\|V_{(\xi)}\|_{C_tL^p}+\|a_{(\xi)}\|_{C^2_{t,x}}\|V_{(\xi)}\|_{C_tW^{1,p}}\\
&\quad+\|a_{(\xi)}\|_{C^1_{t,x}}\|V_{(\xi)}\|_{C_tW^{2,p}}+\|a_{(\xi)}\|_{C^0_{t,x}}\|V_{(\xi)}\|_{C_tW^{3,p}}\Big)
\\
&\lesssim \bar {e}^{1/2}r_\perp^{2/p-1}r_\|^{1/p-1/2}\left(\ell^{-29}
\lambda_{q+1}^{-2}+\ell^{-22}\lambda_{q+1}^{-1}+\ell^{-15}+\ell^{-8}\lambda_{q+1}\right)\\
&\lesssim \bar {e}^{1/2}r_\perp^{2/p-1}r_\|^{1/p-1/2}\ell^{-8}\lambda_{q+1},
\endaligned
\end{equation}
and
\begin{equation}\label{corrector est2}
\aligned
\|w_{q+1}^{(t)}\|_{C_tW^{1,p}}
&\leq\frac{1}{\mu}\sum_{\xi\in\Lambda}
\Big(\|a_{(\xi)}\|_{C^0_{t,x}}\|a_{(\xi)}\|_{C^1_{t,x}}\|\phi_{(\xi)}\|_{L^{2p}}^2\|\psi_{(\xi)}\|_{C_tL^{2p}}^2\\
&\quad+\|a_{(\xi)}\|_{C^0_{t,x}}^2\|\phi_{(\xi)}\|_{L^{2p}}\|\nabla \phi_{(\xi)}\|_{L^{2p}}\|\psi_{(\xi)}\|_{C_tL^{2p}}^2\\
&\quad+\|a_{(\xi)}\|_{C^0_{t,x}}^2\|\phi_{(\xi)}\|_{L^{2p}}^2\|\nabla \psi_{(\xi)}\|_{C_tL^{2p}}\|\psi_{(\xi)}\|_{C_tL^{2p}}\Big)
\\
&\lesssim \frac{\bar {e}}{\mu}r_\perp^{2/p-2}r_\|^{1/p-1}\left(\ell^{-23}+\ell^{-16}\lambda_{q+1}\right)\lesssim \bar {e}r_\perp^{2/p-2}r_\|^{1/p-1}\ell^{-16}\lambda_{q+1}^{-2/7}.
\endaligned
\end{equation}

\subsubsection{Control of  the energy}
\label{s:energy}

In this section, we follow the ideas of \cite{BMS20} to control the energy.
Since $W_{(\xi)}$ is $(\mathbb{T}/(r_\perp \lambda_{q+1}))^3$ periodic, we deduce that
\begin{equation}\label{eq:W}\mathbb{P}_{\neq0}(W_{(\xi)}\otimes W_{(\xi)})=\mathbb{P}_{\geq r_\perp \lambda_{q+1}/2}(W_{(\xi)}\otimes W_{(\xi)}),\end{equation}
where $\mathbb{P}_{\geq r}=\Id-\mathbb{P}_{<r}$ and $\mathbb{P}_{< r}$ denotes the Fourier multiplier operator, which projects a function onto its Fourier frequencies $<r$ in absolute value.

We define
$$\delta E(t):=\Big| e(t)(1-\delta_{q+2})-\|v_{q+1}(t)+{z_{q+1}}(t)\|_{L^2}^2\Big|.$$

\bp It holds for $t\in[t_{q+1}, \mathfrak{t}]$
\begin{align}
\delta E(t)\leq& \frac14\delta_{q+2}e(t).
\end{align}
\ep
\begin{proof}
By definition of $\gamma_q$ we find
\begin{align}\label{eq:deltaE}
\begin{aligned}
\delta E(t)&\leq \big|\|w_{q+1}^{(p)}\|_{L^2}^2-3\gamma_q(2\pi)^3\big|+\|w_{q+1}^{(c)}+w_{q+1}^{(t)}\|_{L^2}^2+2\|(v_\ell+z_{q+1})(w_{q+1}^{(c)}+w_{q+1}^{(t)})\|_{L^1}\\
&\qquad+2\|(v_\ell+z_{q+1})w_{q+1}^{(p)}\|_{L^1}+2\|w_{q+1}^{(p)}(w_{q+1}^{(c)}+w_{q+1}^{(t)})\|_{L^1}\\
&\qquad+\|v_\ell-v_q+z_{q+1}-z_q\|_{L^2}^2+2\|(v_\ell-v_q+z_{q+1}-z_q)(v_q+z_q)\|_{L^1},
\end{aligned}
\end{align}
which shall be estimated.

Let us begin with the bound of the first term on the right hand side of \eqref{eq:deltaE}. We use \eqref{can} and the fact that $\mathring{R}_\ell$ is traceless to deduce
\begin{align*}
|w_{q+1}^{(p)}|^2-3\gamma_q=6\sqrt{\ell^2+|\mathring{R}_\ell|^2}+3(\gamma_\ell-\gamma_q)+\sum_{\xi\in \Lambda}a_{(\xi)}^2P_{\neq0}|W_{(\xi)}|^2
,
\end{align*}
hence
\begin{align}\label{eq:gg}
 \big|\|w_{q+1}^{(p)}\|_{L^2}^2-3\gamma_q(2\pi)^3\big|\leq 6\cdot(2\pi)^3\ell+6\|\mathring{R}_\ell\|_{L^1}+3\cdot(2\pi)^3|\gamma_\ell-\gamma_q|+\sum_{\xi\in \Lambda}\Big|\int_{\mathbb{T}^{3}} a_{(\xi)}^2P_{\neq0}|W_{(\xi)}|^2 \dif x\Big|.
\end{align}
Here we estimate each term separately. Using \eqref{ell1}  we find
\begin{align*}
6\cdot(2\pi)^3\ell\leq 6\cdot(2\pi)^3\lambda_{q+1}^{-\frac{3\alpha}2}\leq  \frac{1}{ 48}\lambda_{q+1}^{-2\beta b}e(t)\leq\frac{1}{ 48}\delta_{q+2}e(t),
\end{align*}
which requires $2\beta b<\frac{3\alpha}2$ and choosing $a$ large to absorb the constant. Here we used also $e(t)\geq \underline{e}\geq 4$.
 Using \eqref{inductionv} on $\mathring{R}_q$ gives
\begin{align*}
6\|\mathring{R}_\ell(t)\|_{L^1}\leq  \frac{1}{8}\delta_{q+2}e(t).
\end{align*}

For the third term in \eqref{eq:gg} we use \eqref{z} and \eqref{inductionv} to have
\begin{align*}
3\cdot(2\pi)^3|\gamma_\ell-\gamma_q|&\lesssim \ell \|e'\|_{C^{0}_{t}}+\ell  \|v_q\|_{C^1_{t,x}}(\|v_q\|_{C_{t}L^2}+{\|z_q\|_{C_{t}L^2}})\\
&\qquad+\ell^{\frac12-2\delta}\|{z_q}\|_{C_t^{\frac12-2\delta}L^\infty}(\|v_q\|_{C_{t}L^2}+\|{z_q}\|_{C_{t}L^2})
\\
&\lesssim \ell \tilde e+\ell \lambda_q^4\bar{e}^{1/2}(1+3M_0\bar{e}^{1/2})+\ell^{\frac12-2\delta}{\lambda_{q+1}^{\frac\alpha4}}(1+3M_0\bar{e}^{1/2})\\
&\lesssim \lambda_{q+1}^{-3\alpha/2}\tilde{e}+\lambda_{q+1}^{-\alpha} \bar{e}M_0+\lambda_{q+1}^{-\frac{3\alpha}{2}(\frac12-2\delta){+\frac\alpha4}}\bar{e}^{1/2}M_0
\\&\lesssim M_0\lambda_{q+1}^{-\frac{\alpha}{2}{+\frac\alpha4}} (\bar{e}+\tilde e)\leq \frac{1}{ 48}\delta_{q+2}e(t),
\end{align*}
where $\delta<\frac1{12}$, we require  $\lambda_{q+1}^{-\frac{\alpha}{4}}< \frac{1}{ 48}\lambda_{q+1}^{-2\beta b}\lambda_1^{2\beta}$ implied by {$\alpha>18\beta b$} and we use   $e(t)\geq \underline{e}\geq 4$ and  choose $a$ large enough to absorb $\bar{e}$,  $\tilde{e}$ and $M_0$. Here, we see in particular that it is  not only the upper and lower bound on $e$ which influences the choice of the parameters and consequently the length of the stopping time $\mathfrak{t}$, but also the full $C^{1}$-norm of $e$.

For the last term in \eqref{eq:gg} we apply \cite[Proposition C.1]{BDLIS16} and \eqref{eq:W} to bound
\begin{align*}
\sum_{\xi\in\Lambda}&\Big|\int_{\mathbb{T}^{3}} a_{(\xi)}^2\mathbb{P}_{\neq0}|W_{(\xi)}|^2 \dif x\Big|=\sum_{\xi\in\Lambda}\Big|\int_{\mathbb{T}^{3}} a_{(\xi)}^2P_{\geq r_\perp\lambda_{q+1}/2}|W_{(\xi)}|^2\dif x\Big|\\
&=\sum_{\xi\in\Lambda}\Big|\int_{\mathbb{T}^{3}} |\nabla|^La_{(\xi)}^2|\nabla|^{-L}\mathbb{P}_{\geq r_\perp\lambda_{q+1}/2}|W_{(\xi)}|^2 \dif x\Big|
\\
&\lesssim\|a_{(\xi)}^2\|_{C^L}(r_\perp\lambda_{q+1})^{-L}\||W_{(\xi)}|^2\|_{L^2}\lesssim\ell^{-7L-16}\delta_{q+1}\bar{e}(r_\perp\lambda_{q+1})^{-L}r_\perp^{-1}r_{\|}^{-\frac12}\\
&\leq\ell^{-7L-16}\delta_{q+1}\bar{e}\lambda_{q+1}^{\frac{8-L}7}
\lesssim \lambda_1^{2\beta}\lambda_{q+1}^{(14L+32)\alpha+\frac{8-L}7-2\beta }\bar{e}\leq\frac1{24}\lambda_1^{2\beta}\lambda_{q+1}^{-2\beta b}e(t)=\frac1{48}\delta_{q+2}e(t).
\end{align*}
Here we may choose $L=9$, $\alpha, b\beta$ small enough to have $2\beta b<\frac1{7}-160\alpha$, which gives
$$\lambda_1^{2\beta}\lambda_{q+1}^{(14L+32)\alpha+\frac{8-L}7-2\beta }\bar{e}\leq \lambda_{q+1}^{160 \alpha-\frac17}\leq\lambda_{q+1}^{-2\beta b}e(t),$$
where $\bar{e}\leq \ell^{-1}\leq \lambda_{q+1}^{2\alpha}$ and we use $e(t)\geq \underline{e}\geq 4$.
This completes the bound for \eqref{eq:gg}.

Going back to \eqref{eq:deltaE}, it remains to control
\begin{align*}
&\|w_{q+1}^{(c)}+w_{q+1}^{(t)}\|_{L^2}^2+2\|(v_\ell+z_{q+1})(w_{q+1}^{(c)}+w_{q+1}^{(t)})\|_{L^1}\\
&\qquad+2\|(v_\ell+z_{q+1})w_{q+1}^{(p)}\|_{L^1}+2\|w_{q+1}^{(p)}(w_{q+1}^{(c)}+w_{q+1}^{(t)})\|_{L^1}\\
&\qquad+\|v_\ell-v_q+z_{q+1}-z_q\|_{L^2}^2+2\|(v_\ell-v_q+z_{q+1}-z_q)(v_q+z_q)\|_{L^1}.
\end{align*}
We use the estimates \eqref{correction est} \eqref{temporal est1} and \eqref{ell} to have
\begin{align*}
\|w_{q+1}^{(c)}+w_{q+1}^{(t)}\|_{L^2}^2&\lesssim \bar{e}\delta_{q+1}\ell^{-44}r_\perp^2r_{\|}^{-2}+\bar{e}^2\delta_{q+1}^2\ell^{-32}r_{\|}^{-3}\lambda_{q+1}^{-2}
\\
&\lesssim \bar{e}\lambda_{q+1}^{88\alpha-\frac47}
+\bar{e}^2\lambda_{q+1}^{64\alpha-\frac27}\leq \frac{1}{48}
\lambda_{q+1}^{-2\beta b}e(t)\leq \frac{\delta_{q+2}}{48}e(t),
\end{align*}
where we chose $a$ large enough to absorb $\bar{e}$ and we used the  bounds for the  parameters $\alpha,\beta$ specified in Section~\ref{s:c}. Similarly, we use $\|v_\ell(t)\|_{L^2}\leq \|v_q\|_{C_{t,q}L^2}$ together with \eqref{estimate wqp} to get
\begin{align*}
&2\|(v_\ell+{z_{q+1}})(w_{q+1}^{(c)}+w_{q+1}^{(t)})\|_{L^1}+2\|w_{q+1}^{(p)}(w_{q+1}^{(c)}+w_{q+1}^{(t)})\|_{L^1}
\\
&\qquad\lesssim \bar{e}^{1/2}M_0\|w_{q+1}^{(c)}+w_{q+1}^{(t)}\|_{L^2}\lesssim M_0\Big(\bar{e}\delta_{q+1}^{1/2}\ell^{-22}r_\perp r_{\|}^{-1}+\bar{e}^{\frac32}\delta_{q+1}\ell^{-16}r_{\|}^{-\frac32}\lambda_{q+1}^{-1}\Big)
\\
&\qquad\lesssim \bar{e}\lambda_{q+1}^{44\alpha-\frac27}+\bar{e}^{\frac32}\lambda_{q+1}^{32\alpha-\frac17}\leq \frac{1}{48}
\lambda_{q+1}^{-2\beta b}e(t)\leq \frac{\delta_{q+2}}{48}e(t).
\end{align*}
Next, we apply \eqref{ell} and \eqref{principle est1} and $\|v_\ell\|_{C^1_{t,x,q+1}}\leq \|v_q\|_{C^1_{t,x,q}} $, $H^{3/4}\subset L^4$  to obtain for $\delta$ in \eqref{stopping time}
\begin{align*}
2\|(v_\ell+{z_{q+1}})w_{q+1}^{(p)}\|_{L^1}&
\lesssim\|v_\ell+{z_{q}}\|_{L^\infty}\|w_{q+1}^{(p)}\|_{L^1}+\|z_{q+1}-z_q\|_{L^4}\|w_{q+1}^{(p)}\|_{L^{4/3}}
\\&\lesssim {(\lambda_q^4\bar{e}^{1/2}+\lambda_{q+1}^{\frac\alpha8})}\delta_{q+1}^{1/2}\ell^{-8}r_\perp^{1-2\beta}r_{\|}^{1/2-\beta}+\bar{e}^{1/2}\delta_{q+1}^{1/2}\ell^{-8}\lambda_{q+1}^{-\frac\alpha8(\frac14-\delta)}r_\perp^{1/2}r_{\|}^{1/4}
\\&\lesssim \bar{e}\lambda_{q+1}^{17\alpha-\frac87(1-2\beta)}+\bar{e}^{1/2}\delta_{q+1}^{1/2}\lambda_{q+1}^{16\alpha-\frac47}\leq \frac{1}{96}
\lambda_{q+1}^{-2\beta b}e(t)\leq \frac{\delta_{q+2}}{96}e(t),
\end{align*}
where we used 
$$\|z_{q+1}-z_q\|_{L^4}\lesssim\|z_{q+1}-z_q\|_{H^{3/4}}\lesssim \lambda_{q+1}^{-\frac\alpha8(\frac14-\delta)}.$$

For the last  terms we have
\begin{align*}
&\|v_\ell-v_q+z_{q+1}-z_q\|_{L^2}^2+2\|(v_\ell-v_q+z_{q+1}-z_q)(v_q+z_q)\|_{L^1}
\\&\lesssim \|v_\ell-v_q\|_{L^2}^2+\|z_{q+1}-z_q\|_{L^2}^2+(\|v_q\|_{L^2}+\|z_q\|_{L^2})(\|v_\ell-v_q\|_{L^2}+\|z_{q+1}-z_q\|_{L^2})
\\&\lesssim M_0\bar{e}^{1/2}(\|v_\ell-v_q\|_{L^2}+\|z_{q+1}-z_q\|_{L^2})
\\
&\lesssim M_0(\ell\lambda_q^4+\lambda_{q+1}^{-\frac\alpha8(1-\delta)})\bar{e}\leq M_0(\lambda_{q+1}^{-\alpha}+\lambda_{q+1}^{-\frac\alpha8(1-\delta)})\bar{e}\leq \frac{1}{96}
\lambda_{q+1}^{-2\beta b}e(t)\leq \frac{\delta_{q+2}}{96}e(t) .
\end{align*}

Combining the above estimates the result follows.
\end{proof}

\subsubsection{Definition of the Reynolds stress $\mathring{R}_{q+1}$}\label{s:def}

Subtracting from (\ref{induction}) at level $q+1$ the system (\ref{mollification}), we obtain on $[t_{q+1},\mathfrak t]$
\begin{equation}\label{stress}
\aligned
\div\mathring{R}_{q+1}-\nabla p_{q+1}&=\underbrace{-\Delta w_{q+1}+\partial_t(w_{q+1}^{(p)}+w_{q+1}^{(c)})+\div((v_\ell+z_\ell)\otimes w_{q+1}+w_{q+1}\otimes (v_\ell+z_\ell))}_{\div(R_{\textrm{lin}})+\nabla p_{\textrm{lin}}}
\\&\quad+\underbrace{\div\left((w_{q+1}^{(c)}+w_{q+1}^{(t)})\otimes w_{q+1}+w_{q+1}^{(p)}\otimes (w_{q+1}^{(c)}+w_{q+1}^{(t)})\right)}_{\div(R_{\textrm{cor}})+\nabla p_{\textrm{cor}}}
\\&\quad+\underbrace{\div(w_{q+1}^{(p)}\otimes w_{q+1}^{(p)}+\mathring{R}_\ell)+\partial_tw_{q+1}^{(t)}}_{\div(R_{\textrm{osc}})+\nabla p_{\textrm{osc}}}
\\&\quad+\underbrace{\div\left(v_{q+1}{\otimes}{z_{q+1}}-v_{q+1}{\otimes}z_\ell+{z_{q+1}}{\otimes}v_{q+1}-z_\ell{\otimes}v_{q+1}
	+{z_{q+1}}{\otimes}{z_{q+1}}-z_\ell{\otimes}z_\ell\right)}_{\div(R_{\textrm{com}1})+\nabla p_{\textrm{com}1}}
\\&\quad+\div(R_{\textrm{com}})-\nabla p_\ell.
\endaligned
\end{equation}
We recall the inverse divergence operator $\mathcal{R}$ as in \cite[Section 5.6]{BV19}, which acts on vector fields $v$ with $\int_{\mathbb{T}^3}v\dif x=0$ as
\begin{equation*}
(\mathcal{R}v)^{kl}=(\partial_k\Delta^{-1}v^l+\partial_l\Delta^{-1}v^k)-\frac{1}{2}(\delta_{kl}+\partial_k\partial_l\Delta^{-1})\div\Delta^{-1}v,
\end{equation*}
for $k,l\in\{1,2,3\}$. Then $\mathcal{R}v(x)$ is a symmetric trace-free matrix for each $x\in\mathbb{T}^3$, and $\mathcal{R}$ is a right inverse of the div operator, i.e. $\div(\mathcal{R} v)=v$. By using $\mathcal{R}$ we define
\begin{equation*}
R_{\textrm{lin}}:=-\mathcal{R}\Delta w_{q+1}+\mathcal{R}\partial_t(w_{q+1}^{(p)}+w_{q+1}^{(c)})
+(v_\ell+z_\ell)\mathring\otimes w_{q+1}+w_{q+1}\mathring\otimes (v_\ell+z_\ell),
\end{equation*}
\begin{equation*}
R_{\textrm{cor}}:=(w_{q+1}^{(c)}+w_{q+1}^{(t)})\mathring{\otimes} w_{q+1}+w_{q+1}^{(p)}\mathring{\otimes} (w_{q+1}^{(c)}+w_{q+1}^{(t)}),
\end{equation*}
\begin{equation*}
R_{\textrm{com}1}:=v_{q+1}\mathring{\otimes}{z_{q+1}}-v_{q+1}\mathring{\otimes}z_\ell+{z_{q+1}}\mathring{\otimes}v_{q+1}-z_\ell\mathring{\otimes}v_{q+1}
+{z_{q+1}}\mathring{\otimes}{z_{q+1}}-z_\ell\mathring{\otimes}z_\ell.
\end{equation*}
We observe that if $\mathring{R}_q(t,x), v_q(t,x)$ are deterministic for $t\in [t_{q},0]$, the same is valid for the above defined error terms $R_{\mathrm{lin}}(t,x)$, $R_{\mathrm{cor}}(t,x)$, $R_{\mathrm{com1}}(t,x)$ for $t\in [t_{q+1},0]$.

In order to define the remaining oscillation error from the third line in \eqref{stress}, we apply (\ref{can}) and (\ref{equation for temporal}) to obtain
\begin{align*}
&\div(w_{q+1}^{(p)}\otimes w_{q+1}^{(p)}+\mathring{R}_\ell)+\partial_tw_{q+1}^{(t)}\\
&\quad=\sum_{\xi\in\Lambda}\div\left(a^2_{(\xi)}\mathbb{P}_{\neq0}(W_{(\xi)}\otimes W_{(\xi)})\right)+\nabla \rho+\partial_tw_{q+1}^{(t)}
\\
&\quad=\sum_{\xi\in\Lambda}\mathbb{P}_{\neq0}\left(\nabla a^2_{(\xi)}\mathbb{P}_{\neq0}(W_{(\xi)}\otimes W_{(\xi)})\right)+\nabla \rho+\sum_{\xi\in\Lambda}\mathbb{P}_{\neq0}\left(a^2_{(\xi)}\div(W_{(\xi)}\otimes W_{(\xi)})\right)+\partial_tw_{q+1}^{(t)}
\\
&\quad=\sum_{\xi\in\Lambda}\mathbb{P}_{\neq0}
\left(\nabla a_{(\xi)}^2\mathbb{P}_{\neq0}(W_{(\xi)}\otimes W_{(\xi)}) \right)+\nabla \rho+\nabla p_1-\frac{1}{\mu}\sum_{\xi\in\Lambda}\mathbb{P}_{\neq0}
\left(\partial_t a_{(\xi)}^2(\phi_{(\xi)}^2\psi_{(\xi)}^2\xi) \right)
\end{align*}
Therefore,
\begin{equation*}
\aligned
R_{\textrm{osc}}:=&\sum_{\xi\in\Lambda}\mathcal{R}
\left(\nabla a_{(\xi)}^2\mathbb{P}_{\neq0}(W_{(\xi)}\otimes W_{(\xi)}) \right)-\frac{1}{\mu}\sum_{\xi\in\Lambda}\mathcal{R}
\left(\partial_t a_{(\xi)}^2(\phi_{(\xi)}^2\psi_{(\xi)}^2\xi) \right)=:R_{\textrm{osc}}^{(x)}+R_{\textrm{osc}}^{(t)},
\endaligned
\end{equation*}
which is also deterministic for $t\in [t_{q+1},0]$.
Finally we define the Reynolds stress on the level $q+1$ by
\begin{equation*}\aligned
\mathring{R}_{q+1}:=R_{\textrm{lin}}+R_{\textrm{cor}}+R_{\textrm{osc}}+R_{\textrm{com}}+R_{\textrm{com}1}.
\endaligned
\end{equation*}
We note that by construction $\mathring{R}_{q+1}(t,x)$ is deterministic for $t\in [t_{q+1},0]$.

\subsubsection{Inductive estimate for $\mathring{R}_{q+1}$}
\label{sss:R}
To conclude the proof of Proposition~\ref{main iteration}, we shall verify the third estimate in \eqref{inductionv}. To this end, we estimate each term in the definition of $\mathring{R}_{q+1}$ separately.

In the following we choose $p=\frac{32}{32-7\alpha}>1$ so that it holds  in particular that $r_\perp^{2/p-2}r_\|^{1/p-1}\leq \lambda_{q+1}^\alpha$. In this section, we also consider all the norms are on $[t_{q+1},\mathfrak t]$ and we omit the subscript $q+1$ if there is no danger of confusion. 
For the linear  error we apply \eqref{inductionv} to obtain for $t\in[t_{q+1}, \mathfrak{t} ]$
\begin{equation*}
\aligned
\|R_{\textrm{lin}}\|_{C_tL^p}
&\lesssim\|\mathcal{R}\Delta w_{q+1}\|_{C_tL^p}+\|\mathcal{R}\partial_t(w_{q+1}^{(p)}+w_{q+1}^{(c)})\|_{C_tL^p}\\
&\qquad+\|(v_\ell+z_\ell)\mathring{\otimes}w_{q+1}+w_{q+1}\mathring{\otimes}(v_\ell+z_\ell)\|_{C_tL^p}
\\
&\lesssim\|w_{q+1}\|_{C_tW^{1,p}}+\sum_{\xi\in\Lambda}\|\partial_t\textrm{curl}
(a_{(\xi)}V_{(\xi)})\|_{C_tL^p}+\bar{e}^{1/2}(\lambda_{q}^4+{\lambda_{q+1}^{\frac\alpha8}})\|w_{q+1}\|_{C_tL^p},
\endaligned
\end{equation*}
where by \eqref{bounds} and \eqref{estimate aN}
\begin{equation*}
\aligned
\sum_{\xi\in\Lambda}\|\partial_t\textrm{curl}
(a_{(\xi)}V_{(\xi)})\|_{C_tL^p}&\leq \sum_{\xi\in\Lambda}\left(\|
a_{(\xi)}\|_{C_tC^1_x}\|\partial_t V_{(\xi)}\|_{C_tW^{1,p}}+\|\partial_ta_{(\xi)}\|_{C_tC^1_x}\| V_{(\xi)}\|_{C_tW^{1,p}}\right)\\
&\lesssim \bar{e}^{1/2} \ell^{-15}r_{\perp}^{2/p}r_{\|}^{1/p-3/2}\mu+\bar{e}^{1/2} \ell^{-22}r_{\perp}^{2/p-1}r_{\|}^{1/p-1/2}\lambda_{q+1}^{-1}.
\endaligned
\end{equation*}
In view of (\ref{principle est22}), (\ref{corrector est2}) as well as (\ref{principle est1}), (\ref{corr temporal}),  we deduce  for $t\in{[t_{q+1},\mathfrak{t}]}$
\begin{equation*}
\aligned
\|R_{\textrm{lin}}\|_{C_tL^p}
&\lesssim \bar{e}^{1/2}\ell^{-8}r_\perp^{2/p-1}r_\|^{1/p-1/2}\lambda_{q+1}+\bar{e}\ell^{-16}r_\perp^{2/p-2}r_\|^{1/p-1}\lambda_{q+1}^{-2/7}
\\
&\qquad +\bar{e}^{1/2}\ell^{-15}r_\perp^{2/p}r_\|^{1/p-3/2}\mu
+\bar{e}^{1/2}\ell^{-22}r_\perp^{2/p-1}r_\|^{1/p-1/2}\lambda_{q+1}^{-1}\\
&\qquad+\bar{e}\ell^{-8}r_\perp^{2/p-1}r_\|^{1/p-1/2}(\lambda^4_{q}+{\lambda_{q+1}^{\frac\alpha8}})
\\
&\lesssim \bar{e}^{1/2}\lambda_{q+1}^{17\alpha-1/7}+\bar{e}\lambda_{q+1}^{33\alpha-2/7}+\bar{e}^{1/2}\lambda_{q+1}^{31\alpha-1/7}+\bar{e}^{1/2}\lambda_{q+1}^{45\alpha-15/7}+\bar{e}\lambda_{q+1}^{19\alpha-8/7}
\\
&\leq\frac{e(t)\delta_{q+3}}{5\cdot 48}.
\endaligned
\end{equation*}
Here, we took  $\bar e\lambda_{q+1}^{31\alpha-\frac17}\leq \frac{1}{5\cdot48}\lambda_{q+1}^{-2\beta b^2}e(t)$, which can be done  by choosing  $a$ sufficiently large to absorb the implicit universal constant.

The corrector error  is estimated  using \eqref{principle est1}, \eqref{correction est}, \eqref{temporal est1}, (\ref{corr temporal})  for $t\in[t_{q+1}, \mathfrak{t}]$ as
\begin{equation*}
\aligned
\|R_{\textrm{cor}}\|_{C_tL^p}
&\leq\|w_{q+1}^{(c)}+ w_{q+1}^{(t)}\|_{C_tL^{2p}}\| w_{q+1}\|_{C_tL^{2p}}+\|w_{q+1}^{(c)}+ w_{q+1}^{(t)}\|_{C_tL^{2p}}\| w_{q+1}^{(p)}\|_{C_tL^{2p}}
\\
&\lesssim \bar{e}\left(\ell^{-22}r_\perp^{1/p}r_\|^{1/(2p)-3/2}
+\ell^{-16}\bar{e}^{1/2}r_\perp^{1/p-1}r_\|^{1/(2p)-2}\lambda_{q+1}^{-1}\right)\ell^{-8}r_\perp^{1/p-1}r_\|^{1/(2p)-1/2}
\\
&\lesssim \bar{e}\left(\ell^{-30}r_\perp^{2/p-1}r_\|^{1/p-2}
+\ell^{-24}\bar{e}^{1/2}r_\perp^{2/p-2}r_\|^{1/p-5/2}\lambda_{q+1}^{-1}\right)
\\
&\lesssim \bar{e}\left(\lambda_{q+1}^{61\alpha-2/7}+\bar{e}^{1/2}\lambda_{q+1}^{49\alpha-1/7}\right)
\leq\frac{e(t)\delta_{q+3}}{5\cdot48}.
\endaligned
\end{equation*}
Here, we have taken $\bar{e}^{1/2}\lambda_{q+1}^{49\alpha-1/7}\leq\lambda_{q+1}^{50\alpha-\frac17}\leq \frac{1}{5\cdot48}\lambda_{q+1}^{-2\beta b^2}e(t)$.

We continue with the oscillation error $R_{\textrm{osc}}$. In order to bound $R_{\textrm{osc}}^{(x)}$, we  recall the following result from \cite[Lemma~7.5]{BV19}.

\bl\label{lemma} Fix parameters $1\leq \zeta<\kappa, p\in (1,2]$, and assume there exists $N\in\mathbb{N}$ such that $\zeta^N\leq \kappa^{N-2}$. Let
	$a\in C^N(\mathbb{T}^3)$ be such that there exists $C_a>0$ with
	\begin{equation*}\|D^j a\|_{C^0}\leq C_a \zeta^j,\end{equation*}
	for all $0\leq j\leq N$. Assume that $f\in L^p(\mathbb{T}^3)$ such that $\int_{\mathbb{T}^3}a(x)\mathbb{P}_{\geq\kappa}f(x)\dif x=0$. Then we have
	\begin{equation*}\||\nabla|^{-1}(a \mathbb{P}_{\geq\kappa}f)\|_{L^p}\leq C_a \frac{\|f\|_{L^p}}{\kappa},\end{equation*}
	where the implicit constant depends only on $p$ and $N$.
	
\el

Using \eqref{eq:W} and Lemma \ref{lemma} with $a=\nabla a^2_{(\xi)}$ for $C_a=\bar{e}\ell^{-23}$, $\zeta=\ell^{-7}$, $\kappa=r_{\perp}\lambda_{q+1}/2$ and any $N\geq 3$,
we have $\zeta^3=\ell^{-21}\leq \lambda_{q+1}^{42\alpha}<\kappa$, which implies
\begin{equation*}
\aligned
\|R_{\textrm{osc}}^{(x)}\|_{C_{t}L^p}&\leq \sum_{\xi\in\Lambda}\big\|\mathcal{R}\big(\nabla a^2_{(\xi)}\mathbb{P}_{\geq r_\perp\lambda_{q+1}/2}(W_{(\xi)}\otimes W_{(\xi)})\big)\big\|_{C_{t}L^p}
\\
&\lesssim \bar{e}\ell^{-23}\frac{\|W_{(\xi)}\otimes W_{(\xi)}\|_{C_{t}L^p}}{r_\perp\lambda_{q+1}}\lesssim \bar{e}\ell^{-23}\frac{\|W_{(\xi)}\|^2_{C_{t}L^{2p}}}{r_\perp\lambda_{q+1}}
\\&\lesssim \bar{e}\ell^{-23}r_\perp^{2/p-2}r_\|^{1/p-1}(r_\perp^{-1}\lambda_{q+1}^{-1})\lesssim \bar{e}\ell^{-23}\lambda_{q+1}^\alpha (r_\perp^{-1}\lambda_{q+1}^{-1})\\
&\lesssim \bar{e}\lambda_{q+1}^{47\alpha-1/7}\leq\frac{e(t)\delta_{q+3}}{10\cdot 48}.
\endaligned
\end{equation*}
For the second term $R_{\textrm{osc}}^{(t)}$ we use Fubini's theorem to integrate along the orthogonal directions of $\phi_{(\xi)}$ and $\psi_{(\xi)}$ and apply (\ref{bounds}) to deduce
\begin{equation*}
\aligned
\|R_{\textrm{osc}}^{(t)}\|_{C_{t}L^p}&\leq \mu^{-1}\sum_{\xi\in\Lambda}\|\partial_t a_{(\xi)}^2\|_{C^{0}_{t,x}}\|\phi_{(\xi)}\|_{C_{t}L^{2p}}^2\|\psi_{(\xi)}\|_{C_{t}L^{2p}}^2
\\
&\lesssim \bar{e} \mu^{-1}\ell^{-23}r_\perp^{2/p-2}r_\|^{1/p-1}\lesssim \bar{e} \lambda_{q+1}^{47\alpha-9/7}\leq \frac{e(t)\delta_{q+3}}{10\cdot 48}.
\endaligned
\end{equation*}

In view of the standard mollification estimates and \eqref{z}, \eqref{inductionv} it holds for $t\in[t_{q+1}, \mathfrak{t}]$
\begin{equation*}
\aligned
\|R_{\textrm{com}}\|_{C_{t,q+1}L^1}&\lesssim \ell(\|v_q\|_{C^1_{t,x,q}}+\|z_q\|_{C_{t,q}C^1})(\|v_q\|_{C_{t,q}L^2}+\|z_q\|_{C_{t,q}L^2})\\
&\qquad+\ell^{\frac{1}{2}-2\delta}\|z_q\|_{C_{t,q}^{\frac{1}{2}-2\delta}L^\infty}(\|v_q\|_{C_{t,q}L^2}
+\|z_q\|_{C_{t,q}L^2})\\
&\lesssim M_0\ell  (\lambda_q^4+\lambda_{q+1}^{\frac\alpha4})\bar{e}+\ell^{\frac{1}{2}-2\delta}\lambda_{q+1}^{\frac\alpha4}M_0\bar{e}\leq \frac{e(t)\delta_{q+3}}{5\cdot48},
\endaligned
\end{equation*}
where $\delta<\frac{1}{12}$ and  we require  that
$$
CM_0\bar{e}\ell \lambda_q^4\leq CM_0\bar{e}\lambda_{q+1}^{-\alpha}\leq \frac{1}{10}\lambda_{q+1}^{-2\beta b^2}\leq\frac{e(t)\delta_{q+3} }{{15}\cdot 48},
$$
$$CM_0\ell^{\frac{1}{2}-2\delta}{\lambda_{q+1}^{\frac\alpha4}}\bar{e}\leq C\bar{e}M_0 \ell^{\frac13}{\lambda_{q+1}^{\frac\alpha4}}\leq C M_0\lambda_{q+1}^{-\alpha/2+{\frac\alpha4}}\bar{e}\leq \frac{1}{10}\lambda_{q+1}^{-2b^2\beta}
\leq\frac{e(t)\delta_{q+3}}{{15}\cdot48}.
$$
 With the  choice of $\ell$ in \eqref{ell1} and since we postulated that $\alpha >{18}\beta b^2$, this can indeed be achieved by possibly increasing  $a$.
Finally, we use \eqref{z} to obtain for $t\in[t_{q+1}, \mathfrak{t}]$
{\begin{equation*}
	\aligned
\|R_{\textrm{com}1}\|_{C_tL^1}\lesssim& M_0\bar{e}^{1/2} \|z_\ell-z_{q+1}\|_{C_tL^2}\lesssim M_0\bar{e}^{1/2} (\|z_\ell-z_{q}\|_{C_tL^2}+\|z_{q+1}-z_{q}\|_{C_tL^2})
\\\leq&  M_0\bar{e}^{1/2}(\ell^{\frac{1}{2}-2\delta}+\lambda_{q+1}^{-\frac\alpha8(1-\delta)})\leq \frac{e(t)\delta_{q+3}}{5\cdot48}.
\endaligned
\end{equation*}}
Here, we used $\alpha>{18}\beta b^2$ and {$\frac\alpha8(1-\delta)>2\beta b^2$} in the last inequality.
Summarizing all the above estimates we obtain
\begin{equation*}
\aligned
&\|\mathring{R}_{q+1}\|_{C_{t}L^1}\leq \frac1{48}e(t)\delta_{q+3},
\endaligned
\end{equation*}
which is the desired last bound in \eqref{inductionv}.

The proof of Proposition~\ref{main iteration} is therefore complete.

\section{Non-uniqueness of  Markov solutions}\label{sec:mar}

\subsection{Martingale solutions}
\label{s:martsol}

The proof of non-uniqueness of Markov solutions in Theorem~\ref{th:ma2} requires two ingredients: (i) a stable notion of solution, which is suitable for the abstract procedure of Markov selection introduced by Krylov~\cite{K73} and adapted to the setting of Leray solutions to the Navier--Stokes system \eqref{1} by Flandoli and Romito \cite{FR08}; (ii) a non-uniqueness in law result in the same class of solutions. As non-uniqueness in law for Leray martingale solutions seems out of reach at the moment, we relax the notion of solution accordingly and establish non-uniqueness in law by means of convex integration combined with certain probabilistic arguments. However, we cannot relax the notion of solution too much as  compactness of the set of solutions is indispensable for the stability in (i).

To fulfill these requirements we search for martingale solutions in the class $L^{\infty}_{\rm loc}([0,\infty);L^{2}_{\sigma})\cap L^{2}_{\rm loc}([0,\infty);H^{\gamma})$ for  $\gamma\in (0,1)$ from Theorem~\ref{Main results1}. In addition, a relaxed form of an energy inequality has to be included in the definition of martingale solution. Therefore, both Leray as well as the convex integration solutions constructed in Section~\ref{s:1.1} belong to this class. The formulation of the relaxed energy inequality is essential in order to check the requirements of the  Markov selection procedure. Since the kinetic energy is only lower semicontinuous, it is not possible to obtain stability of the relaxed energy inequality for all times. To overcome this issue we follow  the ideas of Flandoli and Romito \cite{FR08} and employ an almost sure supermartingale formulation.

Let us  recall the key definition from \cite[Definition 3.2]{FR08}.

\bd\label{df:sm}
Let $\theta$ be an $(\mathcal{B}_t)_{t\geq 0}$-adapted real-valued stochastic process on $\Omega_{0}$ and let $P$ be a probability measure on $\Omega_{0}$. We call $\theta$ an almost sure $(\mathcal{B}_t)_{t\geq0}$-supermartingale under $P$ provided $\theta$ is $P$-integrable and
\begin{align}\label{ine:sm}
\E^P[\theta_t1_A]\leq \E^P[\theta_s1_A],
\end{align}
for a.e. $s\geq 0$, for every $t\geq s$ and every  $A\in \mathcal{B}_s$. The time points $s$ for which \eqref{ine:sm} holds are called regular times of $\theta$. The time points $s$ for which \eqref{ine:sm} does not hold are called exceptional times of $\theta$.
\ed

Similarly, we could also introduce almost sure $(\cB_t^0)_{t\geq0}$-supermartingale.  It is easy to check the following result.

\bp\label{prop:0}
Suppose $\theta $ is an almost sure $(\cB_t)_{t\geq0}$-supermartingale and $\theta$ is $(\cB_t^0)_{t\geq0}$-adapted. Then $\theta$ is an almost sure $(\cB_t^0)_{t\geq0}$-supermartingale.
\ep

\begin{remark}
Note that \cite{FR08} only introduced  almost sure supermartingales with respect to the canonical filtration $(\cB_t^0)_{t\geq0}$. This is also the setting we employ in Definition~\ref{martingale solution} and in the Markov selection in Section~\ref{s:mark}. However, for the probabilistic extension of the convex integration solutions in Section~\ref{s:ext}, particularly in Proposition \ref{prop:1}, we additionally require a notion of  martingale solution up to a stopping time with respect to the right continuous filtration $(\mathcal{B}_{t})_{t\geq0}$, see Definition~\ref{def:martsol}. As seen in Proposition~\ref{prop:2}, the extended solutions then satisfy Definition~\ref{martingale solution} with the correct filtration $(\cB_t^0)_{t\geq0}$ and can therefore be used in the Markov selection procedure.
\end{remark}

Using a.s. supermartingales in the  formulation of the (relaxed) energy inequality only leads to the so-called almost sure Markov property. That is, the Markov property only holds  for almost every time, in contrast to the classical Markov or even the strong Markov property. See Section~\ref{s:mark} for precise definitions.

In the case of stochastic Euler equations in \cite{HZZ20} we were able to overcome this limitation. Namely, for a suitable notion of dissipative martingale solution we proved existence as well as non-uniqueness of strong Markov solutions. This was permitted, on the one hand, by including an additional variable, i.e. the energy, into the selection procedure, and, on the other hand, by achieving a stronger form of the energy inequality in the convex integration construction. Since the convex integration in the Navier--Stokes setting is significantly more challenging, we are only able to obtain a weaker form of the (relaxed) energy inequality and consequently only the almost sure Markov property.

In what follows, we fix $\gamma\in (0,1)$ from Theorem~\ref{Main results1}.

\begin{definition}\label{martingale solution}
	Let $s\geq 0$ and $x_{0}\in L^{2}_{\sigma}$. A probability measure $P\in \mathscr{P}(\Omega_0)$ is  a martingale solution to the Navier--Stokes system \eqref{1}  with the initial value $x_0 $ at time $s$ provided
	
	\no{\rm(M1)} $P(x(t)=x_0,  0\leq t\leq s)=1$
	and
$P(x\in L^\infty_{\rm{loc}}([0,\infty);L^2_\sigma)\cap L^2_{\rm{loc}}([0,\infty);H^\gamma))=1.$

	\no{\rm(M2)} For every $e_i\in C^\infty(\mathbb{T}^3)\cap L^2_\sigma$, and for $t\geq s$ the process
	$$M_{t,s}^{i}:=\langle x(t)-x(s),e_i\rangle+\int^t_s\langle \div(x(r)\otimes x(r))-\Delta x(r),e_i\rangle \dif r$$
	is a continuous square integrable $(\mathcal{B}_t^0)_{t\geq s}$-martingale under $P$ with the quadratic variation process
	given by
	$\|G^*e_i\|_{U}^2(\cdot-s)$.

	\no {\rm(M3)}
	For every $p\in\mathbb{N}$ there exists
	$C_{p}\geq 0$ such that the process  $\{E^p(t)-E^p(s)\}_{t\geq s}$ defined through\footnote{Recall that we denoted $C_{G}:=\|G\|_{L_2(U;L^2_\sigma)}^2$.}
	$$
	E^p(t):=\|x(t)\|^{2p}_{L^2}+ 2p\int_0^{t}\|x(l)\|^{2p-2}_{L^2} \|x(l)\|^2_{H^{\gamma}} \dif l- C_pC_{G}\int_0^t\|x(l)\|^{2p-2}_{L^2}\dif l$$
	 is an almost sure $(\mathcal{B}_t^0)_{t\geq s}$-supermartingale under $P$ and $s$ is a regular time.

	\end{definition}

Given a martingale solution $P$, define the set $T_P\subset(0,\infty)$ of exceptional times of $P$ as the union of the sets of exceptional times of the a.s. supermartingales $E^p$, $p\in \mathbb{N}$.

\begin{remark}
Note that the classical Leray martingale solutions $P$ as considered for instance in \cite{FR08} satisfy the energy inequality
$$
\begin{aligned}
&\mathbf{E}^{P}\left[\|x(t)\|_{L^{2}}^{2p}\right]+2p\mathbf{E}^{P}\left[\int_{0}^{t}\|x(l)\|_{L^{2}}^{2p-2}\|x(l)\|_{H^{1}}^{2}\dif l\right]\\
&\qquad\leq \|x(0)\|_{L^{2}}^{2p}+p(2p-1)C_{G}\mathbf{E}^{P}\left[\int_{0}^{t}\|x(l)\|_{L^{2}}^{2p-2}\dif l\right].
\end{aligned}
$$
Moreover, it can be shown that they are martingale solutions in the sense of Definition~\ref{martingale solution} with appropriate constants $C_{p}$.  Nevertheless, in order to include also our convex integration solutions into the same class of solutions,  we formulate the condition (M3) in Definition~\ref{martingale solution} for solutions only in $H^{\gamma}$ and with general constants $C_{p}$.
\end{remark}

\begin{remark}\label{r:CG}
The definition of $E^{p}$ in Definition~\ref{martingale solution} is  notationally convenient in the truly stochastic setting, i.e. for $C_{G}\neq0$. However, we note that it is not consistent with the deterministic setting. Indeed, we show in Section~\ref{s:det} that the non-uniqueness of semiflows in Corollary~\ref{cor:det} follows for solutions satisfying
$$ \| x (t) \|_{L^{2}}^2 + 2 \int_s^t \| x(l) \|_{H^{\gamma}}^2 \dif l- C t \leqslant
   \| x (s) \|^2 - C s , $$
   for some $C>0$, every $t$ and almost every $s$.
In particular, for the convex integration solutions it is necessary to include a non-zero constant $C$.
Therefore, a consistent version of the relaxed energy inequality in the stochastic setting is  obtained through
$$ E^p (t) = \| x (t) \|_{L^{2}}^{2 p} + 2 p \int_0^t \| x(l) \|_{L^2}^{2 p - 2} \| x(l)
   \|_{H^{\gamma}}^2\dif l - (C_{p, 1} + C_{p, 2} C_G) \int_0^t \| x(l) \|_{L^2}^{2 p -
   2}\dif l . $$
Here, we can tune $C_{p, 1}$ in order to include the convex integration solutions (in
both deterministic and stochastic setting, independent of the size of $C_G$)
and we choose  $C_{p, 2}$ to include Leray solutions in the stochastic setting.

Alternatively, we may also keep (M3) in Definition~\ref{martingale solution} as it is and redefine $C_{G}$ to $C_{G}>0$ such that $C_{G}\geq \|G\|_{L_{2}(U;L^{2}_{\sigma})}^{2}$.
\end{remark}

We also define martingale solutions up to a stopping time $\tau:\Omega_{0}\to[0,\infty]$ with respect to the right continuous filtration $(\mathcal{B}_{t})_{t\geq s}$. This is necessary since the convex integration in Section~\ref{s:1.1} only yields solutions up to a stopping time, which  requires the right continuous filtration when transferred to the canonical path space $\Omega_{0}$.
We define the space of trajectories stopped at the time $\tau$ by
$$
\Omega_{0,\tau}:=\{\omega(\cdot\wedge\tau(\omega));\omega\in \Omega_{0}\}.
$$
We note that due to the Borel measurability of $\tau$, the set $\Omega_{0,\tau}=\{\omega; x(t,\omega)=x(t\wedge \tau(\omega),\omega), \forall t\geq0\}$ is a Borel subset of $\Omega_{0}$ hence $\mathscr{P}(\Omega_{0,\tau})\subset \mathscr{P}(\Omega_{0})$.

\begin{definition}\label{def:martsol}
Let $s\geq 0$ and $x_{0}\in L^{2}_{\sigma}$. Let $\tau\geq s$ be a $(\mathcal{B}_{t})_{t\geq s}$-stopping time. A probability measure $P\in\mathscr{P}(\Omega_{0,\tau})$  is  a martingale solution to the Navier--Stokes system \eqref{1} on $[s,\tau]$ with the initial value $x_0$ at time $s$ provided

\no{\rm (M1)} $P(x(t)=x_0,  0\leq t\leq s)=1$  and
$P(x\in L^\infty_{\rm{loc}}([0,\tau];L^2_\sigma)\cap L^2_{\rm{loc}}([0,\tau];H^\gamma))=1.$

\no{\rm (M2)} For every $e_i\in C^\infty(\mathbb{T}^3)\cap L^2_\sigma$, and for $t\geq s$ the process
$$M_{t\wedge\tau,s}^{i}:=\langle x(t\wedge\tau)-x_{0},e_i\rangle+\int^{t\wedge\tau}_s\langle \div(x(r)\otimes x(r))-\Delta x(r),e_i\rangle \dif r$$
is a continuous square integrable $(\mathcal{B}_t)_{t\geq s}$-martingale under $P$ with the quadratic variation process
given by
$\|G^*e_i\|_U^2(\cdot\wedge\tau-s).$

\no {\rm (M3)}
For every $p\in\mathbb{N}$ there exists
	$C_p\geq 0$ such that the process  $\{E^p(t\wedge\tau)-E^p(s)\}_{t\geq s}$ defined through
	$$
	E^p(t):=\|x(t)\|^{2p}_{L^2}+ 2p\int_0^{t}\|x(l)\|^{2p-2}_{L^2} \|x(l)\|^2_{H^{\gamma}} \dif l- C_pC_{G}\int_0^t\|x(l)\|^{2p-2}_{L^2}\dif l$$
	 is an almost sure $(\mathcal{B}_t)_{t\geq s}$-supermartingale under $P$ and $s$ is a regular time.
\end{definition}

The following result provides  existence as well as stability of martingale solutions. Similar results can be found in \cite{FR08,GRZ09,HZZ19}. However, as our notion of solution is weaker/different we include the proof for readers' convenience in Appendix~\ref{s:appA}.

\bt\label{convergence}
	For every $(s,x_0)\in [0,\infty)\times L_{\sigma}^2$, there exists  $P\in\mathscr{P}(\Omega_0)$ which is a martingale solution to the Navier--Stokes system \eqref{1} starting at time $s$ from the initial condition $x_0$  in the sense of Definition \ref{martingale solution}. The set of all  martingale solutions with the same $C_{p}$, $p\in\mathbb{N}$, in \emph{(M3)} of Definition  \ref{martingale solution} is denoted by $\mathscr{C}(s,x_0,C_{p})$.
	
	Let $(s_n,x_n)\rightarrow (s,x_0)$ in $[0,\infty)\times L_{\sigma}^2$ as $n\rightarrow\infty$  and let $P_n\in \mathscr{C}(s_n,x_n,C_{p})$. Then there exists a subsequence $n_k$ such that the sequence $\{P_{n_k}\}_{k\in\mathbb{N}}$  converges weakly to some $P\in\mathscr{C}(s,x_0,C_{p})$.
\et

\begin{remark}\label{r:4.6}
An important observation here is that the stability in Theorem~\ref{convergence} only holds in the class of martingale solutions with the same constants $C_{p}$, $p\in\N$.
\end{remark}

\subsection{Probabilistic extension of martingale solutions}
\label{s:ext}

In this section we collect two auxiliary results developed in \cite[Section 3]{HZZ19} and needed in order to extend the convex integration solutions to the full time interval $[0,\infty)$. Note that our current definition of martingale solution Definition~\ref{martingale solution} differs from the one in \cite[Definition 3.1]{HZZ19}. Notably, we include the supermartingale formulation of the relaxed energy  inequality which is suitable for Markov selections, unlike its form  in \cite{HZZ19}. The proof of Proposition~\ref{prop:1} below follows the  arguments of \cite[Proposition~3.2]{HZZ19} (see also  \cite[Proposition 5.10]{HZZ20}), and the proof of Proposition~\ref{prop:2}  is slightly different from \cite[Proposition~3.4]{HZZ19} due to the different formulation of (M3).

\bp\label{prop:1}
	Let $\tau$ be a bounded $(\mathcal{B}_{t})_{t\geq0}$-stopping time.
	For every $\omega\in \Omega_{0}$ there exists $Q_{\omega}\in\mathscr{P}(\Omega_{0})$ such that for $\omega\in \{x(\tau)\in L^2_{\sigma}\}$
	\begin{equation}\label{qomega}
	Q_\omega\big(\omega'\in\Omega_{0}; x(t,\omega')=\omega(t) \textrm{ for } 0\leq t\leq \tau(\omega)\big)=1,
	\end{equation}
	and
	\begin{equation}\label{qomega2}
	Q_\omega(A)=R_{\tau(\omega),x(\tau(\omega),\omega)}(A)\qquad\text{for all}\  A\in \mathcal{B}^{\tau(\omega)}.
	\end{equation}
	where $R_{\tau(\omega),x(\tau(\omega),\omega)}\in\mathscr{P}(\Omega_0)$ is a martingale solution to the Navier--Stokes system \eqref{1} starting at time $\tau(\omega)$ from the initial condition $x(\tau(\omega),\omega)$. Furthermore, for every $B\in\mathcal{B}$ the mapping $\omega\mapsto Q_{\omega}(B)$ is $\mathcal{B}_{\tau}$-measurable.
\ep

\begin{proof}
According to the stability with respect to the initial time and the initial condition in  Theorem~\ref{convergence}, for every $(s,x_0)\in [0,\infty)\times L_{\sigma}^2$ and fixed constants $C_{p}$, $p\in\N$, the set $\mathscr{C}(s,x_{0},C_{p})$ of all associated martingale  solutions  is compact with respect to the weak convergence of probability measures. Let $\mathrm{Comp}(\mathscr{P}(\Omega_0))$ denote the space of all compact subsets of $\mathscr{P}(\Omega_0)$ equipped with the Hausdorff metric. Using the stability from Theorem \ref{convergence} again together with \cite[Lemma 12.1.8]{SV79} we obtain that the map
$$
[0,\infty)\times L^{2}_{\sigma}\to\mathrm{Comp}(\mathscr{P}(\Omega_0)),\qquad (s,x_{0})\mapsto \mathscr{C}(s,x_{0},C_{p}),$$
is Borel measurable. Accordingly, \cite[Theorem 12.1.10]{SV79} gives the existence of a measurable selection. More precisely, there exists a Borel measurable map
$$
[0,\infty)\times L^{2}_{\sigma}\mapsto R_{s,x_{0}},
$$
such that $R_{s,x_{0}}\in \mathscr{C}(s,x_{0},C_{p})$.

As the next step, we recall that the canonical process $\omega$ on $\Omega_{0}$ is continuous in
$H^{-3}$,
hence $x:[0,\infty)\times\Omega_0\rightarrow H^{-3}$ is progressively measurable with respect to the canonical filtration $({\mathcal{B}}^{0}_{t})_{t\geq 0}$ and consequently it is also progressively measurable  with respect to the right continuous filtration $({\mathcal{B}}_{t})_{t\geq 0}$.
 In addition, $\tau$ is a stopping time with respect to the same filtration $({\mathcal{B}}_{t})_{t\geq 0}$. Therefore, it follows from \cite[Lemma~1.2.4]{SV79} that both $\tau$ and $x(\tau(\cdot),\cdot)1_{\{x(\tau)\in L^2_{\sigma}\}}$ are ${\mathcal{B}}_{\tau}$-measurable which by $\mathcal{B}([0,\infty)\times L^2_{\sigma})=\mathcal{B}([0,\infty))\times \mathcal{B}(L^2_{\sigma})$ implies that $(\tau,x(\tau(\cdot),\cdot)1_{\{x(\tau)\in L^2_{\sigma}\}})$ are ${\mathcal{B}}_{\tau}$-measurable. Combining this fact with the measurability of the selection $(s,x_{0})\mapsto R_{s,x_{0}}$ constructed above, we deduce that
\begin{equation}\label{eq:R1}
\Omega_{0}\to \mathscr{P}(\Omega_{0}),\qquad\omega\mapsto R_{\tau(\omega),x(\tau(\omega),\omega)1_{\{x(\tau(\omega),\omega)\in L^2_{\sigma}\}}}
\end{equation}
is ${\mathcal{B}}_{\tau}$-measurable as a composition of ${\mathcal{B}}_{\tau}$-measurable mappings. Recall that for every $\omega\in\Omega_{0}\cap \{x(\tau)\in L^2_{\sigma}\}$ this mapping gives a martingale solution starting at the deterministic time $\tau(\omega)$ from the deterministic initial condition $x(\tau(\omega),\omega)$. Hence, for $\omega\in \{x(\tau)\in L^2_{\sigma}\} $
$$
R_{\tau(\omega),x(\tau(\omega),\omega)}\big(\omega'\in\Omega_{0}; x(\tau(\omega),\omega')=x(\tau(\omega),\omega)\big)=1.
$$

Now, we apply \cite[Lemma 6.1.1]{SV79} and deduce that for every $\omega\in \Omega_{0}\cap \{x(\tau)\in L^2_{\sigma}\}$ there is a unique probability measure
\begin{equation*}
\delta_\omega\otimes_{\tau(\omega)}R_{\tau(\omega),x(\tau(\omega),\omega)}\in\mathscr{P}(\Omega_{0}),
\end{equation*}
such that for $\omega\in \Omega_{0}\cap \{x(\tau)\in L^2_{\sigma}\}$ \eqref{qomega} and \eqref{qomega2} hold.
This permits to concatenate, at the deterministic time $\tau(\omega)$, the Dirac mass $\delta_{\omega}$ with the martingale solution $R_{\tau(\omega),x(\tau(\omega),\omega)}$. 
Define 
\begin{align*}
Q_\omega=\begin{cases}
\delta_\omega\otimes_{\tau(\omega)}R_{\tau(\omega),x(\tau(\omega),\omega)},&\omega \in\{x(\tau)\in L^2_{\sigma}\} ,\\
\delta_{x(\cdot\wedge \tau(\omega))}, &\textrm{otherwise}.
\end{cases}
\end{align*}

In order to show that the mapping $\omega\mapsto Q_{\omega}(B)$  is ${\mathcal{B}}_{\tau}$-measurable for every $B\in {\mathcal{B}}$, it is enough to consider sets of the form $A=\{x(t_{1})\in \Gamma_{1},\dots, x(t_{n})\in\Gamma_{n}\}$ where $n\in\mathbb{N}$, $0\leq t_{1}<\cdots< t_{n}$, and $\Gamma_{1},\dots,\Gamma_{n}\in\mathcal{B}(H^{-3})$. Then by the definition of $\delta_\omega\otimes_{\tau(\omega)}R_{\tau(\omega),x(\tau(\omega),\omega)}$, we have
\begin{align*}
\begin{aligned}
\delta_\omega\otimes_{\tau(\omega)}R_{\tau(\omega),x(\tau(\omega),\omega)}(A)&={\bf 1}_{[0,t_{1})}(\tau(\omega))R_{\tau(\omega),x(\tau(\omega),\omega)}(A)\\
&\quad+\sum_{k=1}^{n-1}{\bf 1}_{[t_{k},t_{k+1})}(\tau(\omega)){\bf 1}_{\Gamma_{1}}(x(t_{1},\omega))\cdots {\bf 1}_{\Gamma_{k}}(x(t_{k},\omega))\\
&\qquad\quad\times R_{\tau(\omega),x(\tau(\omega),\omega)}\big(x(t_{k+1})\in \Gamma_{k+1},\dots, x(t_{n})\in\Gamma_{n}\big)\\
&\quad+{\bf 1}_{[t_{n},\infty)}(\tau(\omega)){\bf 1}_{\Gamma_{1}}(x(t_{1},\omega))\cdots {\bf 1}_{\Gamma_{n}}(x(t_{n},\omega)).
\end{aligned}
\end{align*}
Here the right hand side is ${\mathcal{B}}_{\tau}$-measurable as a consequence of the ${\mathcal{B}}_{\tau}$-measurability of \eqref{eq:R1} and $\tau$. Moreover, $\delta_{x(\cdot\wedge \tau(\omega))}$  is ${\mathcal{B}}_{\tau}$-measurable as a consequence of the ${\mathcal{B}}_{\tau}$-measurability of $x(\tau\wedge \cdot)$.
Thus the final result follows from $\{x(\tau)\in L^2_{\sigma}\}$ is ${\mathcal{B}}_{\tau}$-measurable.
\end{proof}

\bp\label{prop:2}
	Let $x_{0}\in L^{2}_{\sigma}$.
	Let $P$ be a martingale solution to the Navier--Stokes system \eqref{1} on $[0,\tau]$ starting at the time $0$ from the initial condition $x_{0}$. In addition to the assumptions of Proposition \ref{prop:1}, suppose that there exists a Borel  set $\mathcal{N}\subset\Omega_{0,\tau}$ such that $P(\mathcal{N})=0$ and for every $\omega\in \mathcal{N}^{c}$ it holds
	\begin{equation}\label{Q1}
	\aligned
	&Q_\omega\big(\omega'\in\Omega_{0}; \tau(\omega')=
	\tau(\omega)\big)=1.
	\endaligned\end{equation}
	Then the  probability measure $ P\otimes_{\tau}R\in \mathscr{P}(\Omega_{0})$ defined by
	\begin{equation}\label{eq:PR}
	P\otimes_{\tau}R(\cdot):=\int_{\Omega_{0}}Q_{\omega} (\cdot)\,P(\dif\omega)
	\end{equation}
	satisfies $P\otimes_{\tau}R= P$ on the $\sigma$-algebra $\sigma \{ x(t\wedge\tau);\, t\geq 0 \}$ and
	it is a martingale solution to the Navier--Stokes system \eqref{1} on $[0,\infty)$ with initial condition $x_{0}$.
\ep

\begin{proof}
The properties (M1) and (M2) in Definition \ref{martingale solution} follow by the same arguments  as in \cite[Proposition~3.4]{HZZ19}. In order to establish  (M3), we first deduce the $P\otimes_\tau R$-integrability of the process $E^{p}$, $p\in\mathbb{N}$. We write for $p\in\mathbb{N}$
	\begin{align*}
	\E^{P\otimes_\tau R}[|E^p(t)|]&=\int_{\Omega_{0}}\E^{Q_\omega}[|E^p(t)|]\dif P(\omega)\\
	&=\int_{\Omega_{0}}\E^{Q_\omega}[|E^p(t)|1_{\tau(\omega)>t }]\dif P(\omega)+\int_{\Omega_{0}}\E^{Q_\omega}[|E^p(t)|1_{\tau(\omega)\leq t }]\dif P(\omega)\\
	&\leq 2\E^P[|E^p(t\wedge \tau)|]+\int_{\Omega_{0}}\E^{Q_\omega}[|E^p(t)-E^p(\tau(\omega))|{1_{\tau(\omega)\leq t}}]\dif P(\omega).
	\end{align*}
	Using the supermartingale property under $Q_\omega$ we know for $t\geq \tau(\omega)$
	\begin{align*}
	&\E^{Q_\omega}\left[|E^p(t)-E^p(\tau(\omega))|{1_{\tau(\omega)\leq t}}\right]
	\leq \E^{Q_\omega}\left[\|x(t)\|_{L^2}^{2p}\right]+\E^{Q_\omega}\left[\|x(\tau(\omega))\|_{L^2}^{2p}\right]\\
	&\qquad+2p\E^{Q_\omega}\left[\int_{\tau(\omega)}^t\|x(l)\|_{L^{2}}^{2p-2}\|x(l)\|_{H^\gamma}^2\dif l\right]+C_{p}C_{G}\E^{Q_\omega}\left[\int_{\tau(\omega)}^t\|x(l)\|_{L^{2}}^{2p-2}\dif l\right]
	\\
	&\leq C\E^{Q_\omega}\left[\|x(\tau(\omega))\|_{L^2}^{2p}\right]+C\E^{Q_\omega}\left[\int_{\tau(\omega)}^{t}\|x(l)\|_{L^2}^{2p-2}\dif l\right],
	\end{align*}
where $C$ is independent of $\omega$. Then we first choose $p=1$ and by induction we deduce
		\begin{align*}
	&\E^{Q_\omega}\left[|E^p(t)-E^p(\tau(\omega))|{1_{\tau(\omega)\leq t}}\right]
	\leq  C\E^{Q_\omega}\left[\|x(\tau(\omega))\|_{L^2}^{2p}\right]+C(t),
	\end{align*}
where $C$ and $C(t)$ are independent of $\omega$.
Combining the above two inequalities we find
		\begin{align*}
	\E^{P\otimes_\tau R}[|E^p(t)|]&\leq C\E^{P}\left[\|x(\tau)\|_{L^2}^{2p}\right]+2\E^P[|E^p(t\wedge \tau)|]+C(t)\\
	&\leq C\E^{P}\left[\|x(\tau)\|_{L^2}^{2p}\right]+C\E^P\left[\|x(t\wedge \tau)\|_{L^2}^{2p}\right]\\
	&\qquad+C\E^P\left[\int_0^{t\wedge \tau}\|x( l)\|_{L^2}^{2p-2}\|x(l)\|_{H^\gamma}^2\dif l\right]+C\E^P\left[\int_0^{t\wedge \tau}\|x( l)\|_{L^2}^{2p-2}\dif l\right]+C(t)\\
	&\leq C\|x(0)\|_{L^2}^{2p}+C (t)
	<\infty,
	\end{align*}
	where we also first deduce the result for $p=1$ and use induction in the last step.

As the next step,  we  verify the a.s. supermartingale property.  To this end, let $t\geq s$, $A\in\mathcal{B}_s$, $s\notin T_{Q_\omega}$ with $T_{Q_\omega}$ denoting the exceptional set for the supermartingale property.
Note that if $t\geq\tau(\omega)$ and $s\notin T_{Q_\omega}$ then
\begin{align}\label{ine:sup}
\E^{Q_\omega} [(E^p(t)-E^p(t\wedge\tau(\omega))\vee s))1_A]\leq 0 .
\end{align}
Indeed, since  $R_{\tau(\omega),x(\tau(\omega),\omega)}$ is a martingale solution staring from $\tau(\omega)$,
using \eqref{qomega2} we find for $t\geq\tau(\omega)\geq s$
\begin{align*}
\E^{Q_\omega} [(E^p(t)-E^p((t\wedge\tau(\omega))\vee s))1_A]=\E^{Q_\omega} [(E^p(t)-E^p(\tau(\omega)))1_A]\leq 0 ,
\end{align*}
and for $t\geq s\geq \tau(\omega)$, $s\notin T_{Q_\omega}$
\begin{align*}
\E^{Q_\omega} [(E^p(t)-E^p((t\wedge\tau(\omega))\vee s))1_A]=\E^{Q_\omega} [(E^p(t)-E^p(s))1_A]\leq 0 .
\end{align*}
Hence \eqref{ine:sup} follows.

Then we have for a measurable set $T_B\subset[0,t]$
	\begin{align*}
	&\int_0^t1_{T_B}\E^{P\otimes_\tau R} [(E^p(t)-E^p(s))1_A]\dif s
	\\&=\int_0^t1_{T_B}\E^{P\otimes_\tau R} [(E^p(t)-E^p(s))1_{A\cap \{\tau>s\}}]\dif s+\int_0^t1_{T_B}\E^{P\otimes_\tau R} [(E^p(t)-E^p(s))1_{A\cap \{\tau\leq s\}}]\dif s
	\\
	&\leq\int_0^t1_{T_B}\E^{P\otimes_\tau R} [(E^p(t)-E^p(t\wedge\tau))1_{A\cap \{\tau>s\}}]\dif s+\int_0^t1_{T_B}\E^{P\otimes_\tau R} [(E^p(t\wedge\tau)-E^p(s))1_{A\cap \{\tau>s\}}]\dif s
	\\&\quad+\int_0^t\int_{\Omega_{0}}1_{T_B}\E^{Q_\omega} [(E^p(t)-E^p(s))1_{A\cap \{\tau\leq s\}}]\dif P(\omega) \dif s.
	\end{align*}
We show that this is non-positive which then completes the proof by similar argument as the argument after (A.4) in \cite{FR08}.
To this end, for the first term on the right hand side we apply Fubini's theorem together with the fact that $\tau(\omega)$ is a regular time for $R_{\tau(\omega),x(\tau(\omega),\omega)}$ and \eqref{ine:sup} to have
	\begin{align*}
	&\int_0^t1_{T_B}\int \E^{ Q_\omega} [(E^p(t)-E^p(t\wedge\tau))1_{A\cap \{\tau>s\}}]\dif P(\omega)\dif s
	\\\qquad&=\int\int_0^t1_{T_B} \E^{ Q_\omega} [(E^p(t)-E^p(t\wedge\tau(\omega)))1_{A\cap \{\tau(\omega)>s\}}]\dif s\dif P(\omega)\leq0.
	\end{align*}
For the second one we observe that for a.e. $s$
	 $$\E^{P\otimes_\tau Q} [(E^p(t\wedge\tau)-E^p(s))1_{A\cap \{\tau>s\}}]=\E^{P} [(E^p(t\wedge\tau)-E^p(s))1_{A\cap \{\tau>s\}}]\leq0.$$
	 The third term is estimated as follows.  As a consequence of \eqref{ine:sup}, for $s\notin T_{Q_\omega}$
	$$\E^{Q_\omega} [(E^p(t)-E^p(s))1_{A\cap \{\tau\leq s\}}]=\E^{Q_\omega} [(E^p(t)-E^p(t\wedge\tau(\omega))\vee s)1_{A\cap \{\tau(\omega)\leq s\}}]\leq0. $$
Then using Fubini's theorem the third term is also non-positive.
We then find that $E^p$ is an a.s. $(\cB_t)_{t\geq0}$  supermartingale and $E^p$ is $(\cB_t^0)_{t\geq0}$-adapted by definition. By Proposition \ref{prop:0},  we therefore deduce that $E^p$ is $(\cB_t^0)_{t\geq0}$ a.s. supermartingale. 	This completes the proof of (M3).
\end{proof}

\subsection{Non-uniqueness in law}
\label{s:1.2}

In this section, we combine  Theorem~\ref{Main results1} with the results of Section~\ref{s:ext} and establish non-uniqueness in law.
We use the notations of Section~\ref{s:1.1} as well as the previous subsections of Section~\ref{sec:mar}. Namely,
$B$ is   a $GG^{*}$-Wiener process on a probability space $(\Omega, \mathcal{F},\mathbf{P})$ and $(\mathcal{F}_{t})_{t\geq0}$ denotes its normal filtration, i.e. the canonical filtration  augmented by all the $\mathbf{P}$-negligible sets; $\Omega_{0}=C([0,\infty);H^{-3})\cap L^2_{\rm{loc}}([0,\infty);L^2_\sigma)$ is the canonical path space for the velocity and $(\mathcal{B}_{t})_{t\geq 0}$ is the right continuous version of the canonical filtration on $\Omega_{0}$.

As the next step, for every $\omega\in \Omega_0$ we define a process $M_{t,0}^\omega$ similarly to  Definition \ref{martingale solution}, that is,
\begin{equation}\label{eq:M}
M_{t,0}^\omega:= x(t,\omega)-x(0,\omega)+\int^t_0 [\mathbb{P} \div\big(x(r,\omega)\otimes x(r,\omega)\big)-\Delta x(r,\omega)]\dif r
\end{equation}
and for every $\omega\in \Omega_0$ we let
\begin{equation}\label{eq:Z}
Z^\omega(t):= M_{t,0}^\omega+\int_0^t\mathbb{P}\Delta e^{(t-r)\Delta} M_{r,0}^\omega \dif r.
\end{equation}
We refer to \cite[Section 3.3]{HZZ19} for a further discussion of  the main ideas behind these definitions.

In addition, by definition of $Z$ and $M$ together with the regularity of trajectories in $\Omega_{0}$, it follows that for every $\omega\in \Omega_0$, $Z^\omega\in C([0,\infty);H^{-3-\delta})$ for any $\delta>0$.
Motivated by the definition of the stopping time $\mathfrak{t}$ in \eqref{stopping time}, for $n\in\mathbb{N}$ and for  $\delta\in(0,1/12)$ to be determined below we let
\begin{equation*}
\aligned\tau^n(\omega)&=\inf\left\{t\geq 0, \|Z^\omega(t)\|_{{H^{1-\delta}}}>\frac{1}{C_S}-\frac{1}{n}\right\}\wedge \inf\left\{t>0,\|Z^\omega\|_{C_t^{\frac{1}{2}-2\delta}{L^2}}>\frac{1}{C_S}-\frac{1}{n}\right\} \\
&\qquad\wedge \inf\left\{t\geq0,\|Z^\omega(t)\|_{L^2}>\frac{a^{\beta b-b^2\beta}}{\sqrt{12}} -\frac1n\right\}\wedge 1,
\endaligned
\end{equation*}
where $C_S$ is the Sobolev constant for $\|f\|_{L^\infty}\leq C_S \|f\|_{H^{\frac{3+\sigma}{2}}}$ with $\sigma>0$.
We observe that the sequence $(\tau^{n})_{n\in\mathbb{N}}$ is nondecreasing and define
\begin{equation}\label{eq:tauL}
\tau:=\lim_{n\rightarrow\infty}\tau^n.
\end{equation}
Note that without additional regularity of the trajectory $\omega$, it holds true that $\tau^{n}(\omega)=0$.
However, under $P$ we may use the regularity assumption on $G$ to deduce that  $Z\in CH^{1-\delta}\cap C^{1/2-\delta}_{\mathrm{loc}}L^2$ $P$-a.s.
By \cite[Lemma 3.5]{HZZ19} we obtain that $\tau^{n}$ is a $(\mathcal{B}_t)_{t\geq0}$-stopping time and consequently also  $\tau$ is a $(\mathcal{B}_t)_{t\geq 0}$-stopping time as an increasing limit of stopping times.

 We denote by $P$ the law of the convex integration solution $u$ obtained from Theorem \ref{Main results1} with $e(t)=c_0+c_1t$, $c_0\geq 4$, $c_1>0$, and prove the following result.

 \bp\label{prop:ext}
	The probability measure $P$ is a martingale solution to the Navier--Stokes system \eqref{1} on $[0,\tau]$ in the sense of Definition \ref{def:martsol}, where $\tau$ was defined in \eqref{eq:tauL}. The corresponding constants $C_{p}$, $p\in\N$, in {\rm (M3)} depend on $c_{0}$, $c_{1}$ and $C_{G}$.
 \ep
 \begin{proof} By the same argument as in the proof of \cite[Proposition 3.7]{HZZ19} we obtain $\tau(u)=\mathfrak{t}$ $\mathbf{P}$-a.s., with $u$ the solution obtained in Theorem \ref{Main results1}, and it is sufficient to check
 	(M3) under $P$. By the convex integration method in  Section~\ref{s:1.1}, we know that under $P$ it holds $\|x(t)\|_{L^2}^2=c_{0}+c_1t$  as well as $\|x(t)\|_{H^\gamma}^2\leq c_2$ for $t\in[0,\tau]$. In the following we prove that for a suitable choice of the constants $C_{p}$ in the definition processes $E^{p}$, $P$-a.e. trajectory $E^p(t\wedge \tau)$ is non-increasing with respect to $t$. This implies in particular the desired supermartingale property.
 	In fact, for every $t\geq s\geq0$
 	\begin{align*}
	E^1(t\wedge\tau)-E^1(s\wedge\tau)&=
\|x(t\wedge \tau)\|^2_{L^2}-\|x(s\wedge \tau)\|^2_{L^2}+ 2\int_{s\wedge\tau}^{t\wedge\tau} \|x(l)\|^2_{H^{\gamma}} \dif l- C_1C_{G}(t\wedge\tau-s\wedge\tau)
 	\\&\leq (c_1+2c_2-C_1C_{G})(t\wedge \tau-s\wedge \tau)\leq 0,
 	\end{align*}
where the last inequality holds provided  $C_1C_{G}>c_1+2c_2$. Here,  we require in particular that  $C_G>0$. However, see Remark~\ref{r:CG} for a modification applying to the deterministic setting. Hence, under this condition,  $E^1(t\wedge \tau)$ is a supermartingale under $P$.

The same approach applies to general  $E^p$, $p\in\mathbb{N}$. Indeed, since $\tau\leq 1$ by our construction
 	\begin{align*}
&	E^p(t\wedge\tau)-E^p(s\wedge\tau)\\
&=
 	\|x(t\wedge \tau)\|^{2p}_{L^2}-\|x(s\wedge \tau)\|^{2p}_{L^2}+ 2p\int_{s\wedge\tau}^{t\wedge\tau} \|x(l)\|_{L^2}^{2p-2}\|x(l)\|^2_{H^{\gamma}} \dif l- C_pC_{G}\int_{s\wedge\tau}^{t\wedge\tau}\|x(l)\|_{L^2}^{2p-2}\dif l
 	\\
	&\leq (c_{0}+c_1(t\wedge\tau))^{2p}-(c_{0}+c_1(s\wedge\tau))^{2p}+ \big(2pc_2^{p}-C_pC_{G}c_0^{p-1}\big)(t\wedge \tau-s\wedge \tau)
 	\\
	&\leq 2p(c_{0}+c_1)^{2p-1}c_1(t\wedge\tau-s\wedge\tau)+ \big(2pc_2^{p}-C_pC_{G}c_0^{p-1}\big)(t\wedge \tau-s\wedge \tau)
 	\leq 0,
 	\end{align*}
 	provided $C_p $ is large enough satisfying $C_pC_{G}c_0^{p-1}\geq 2pc_2^p+2p(c_{0}+c_1)^{2p-1}c_1$.
 Accordingly,  $E^p(t\wedge \tau)$ is a supermartingale under $P$.
 	\end{proof}

 As the next step, we employ the results of Section~\ref{s:ext} while following the ideas of \cite[Proposition 3.8]{HZZ19} in order to extend the solution $P$ beyond the stopping time $\tau$. Consequently, we deduce non-uniqueness in law.

 \bp\label{prp:ext2}
 	The probability measure $P\otimes_{\tau}R$ is a martingale solution to the Navier--Stokes system \eqref{1} on $[0,\infty)$ in the sense of Definition \ref{martingale solution}. Furthermore, there exist constants $C_{p}\geq 0$, $p\in\N$ such that  martingale solutions to \eqref{1} associated with $C_{p}\geq 0$, $p\in\N$, are not unique.
 \ep

 \begin{proof}
After the application of Proposition~\ref{prop:1}, the key  condition \eqref{Q1} for $Q_{\omega}$ in Proposition~\ref{prop:2} can be verified by the same argument as in \cite[Proposition~3.8]{HZZ19}. As a consequence, the properties (M1)-(M3) follow and the first claim is proved.
Now, it only remains  to show the non-uniqueness in law. We present two proofs. On the one hand, we show that there are different martingale solutions which stem from different convex integration solutions. On the other hand, we show that there are convex integration solutions which differ from the usual Leray martingale solutions.

 As a consequence of the last claim in  Theorem~\ref{Main results1} together with Proposition~\ref{prop:ext}, we can construct two different  martingale solutions $P_{1}, P_{2}$ on $[0,\tau]$ with the same initial data and different energies $e_{1}, e_2$ of the form $e_{i}(t)=c_{0}+c_{1,i}t$, $i=1,2$, and with the same constants $C_{p}$, $p\in\N$ in (M3). Via the first claim in Proposition~\ref{prp:ext2}, the measures  $P_{1}, P_2$ give raise  to two different martingale solutions $P_{1}\otimes_\tau R$ and $P_2\otimes_\tau R$ on $[0,\infty)$, where $R$ in both cases was selected from the set of martingale solutions associated to the same constants $C_{p}$, $p\in\N$.

To compare with a Leray martingale solution, we choose the energy $e(t)=c_0+c_1t$ with $c_1>C_{G}$. Using Theorem \ref{Main results1} and Proposition~\ref{prop:ext}, we obtain one solution $P$ on $[0,\tau]$ such that $\|x(t)\|_{L^2}^2=e(t)$ on $[0,\tau]$ $P$-a.s. Then $P\otimes_\tau R$ is  a martingale solution on $[0,\infty)$ and taking expectation under   $P\otimes_\tau R$ we find
 $$
 \E^{ P\otimes_\tau R}[\|x(\tau)\|_{L^2}^2]=c_0+c_1\E^{ P\otimes_\tau R}[\tau]
 .$$
 Concerning the law $Q$ of the Leary solution, the classical energy inequality yields
 $$
 \E^{ Q}[\|x(\tau)\|_{L^2}^2]\leq c_0+C_{G}\E^{Q}[\tau]<c_0+c_1\E^{Q}[\tau].
 $$
Moreover, $\E^{ P\otimes_\tau R}[\tau]=\E^{Q}[\tau]$ as $\tau$ is a function of the process $Z$, which under both probability measures $P\otimes_\tau R$ and $Q$ solves the linear equation \eqref{linear} hence its law under $P\otimes_\tau R$ and $Q$ is the same. Thus, the two  martingale solutions are different.

Finally, we observe that there is a common choice of constants $C_{p}$, $p\in\N$, such that all the probability measures $P_{1}\otimes_{\tau} R$, $P_{2}\otimes_{\tau} R$, $P\otimes_{\tau} R$ and $Q$ above are martingale solutions with these constants in (M3). Precisely, each $C_{p}$ is given by the maximum of the constants associated to different probability measures.
\end{proof}

\begin{remark}\label{r:4.11}
The second claim in Proposition~\ref{prp:ext2} implies that there are multiple martingale solutions belonging to the same stability class as discussed in Remark~\ref{r:4.6}. This is needed for  non-uniqueness of Markov selections in the sequel.
\end{remark}

\subsection{Selection of non-unique  Markov solutions}\label{s:mark}

In this section, we present the abstract framework of almost sure Markov processes from \cite{FR08} and combine it with the non-uniqueness in law from Proposition~\ref{prp:ext2} in order to prove Theorem~\ref{th:ma2}. Let us start with a result on the existence of regular conditional probability distribution. Throughout this section, let $T\geq 0$.

\bt\label{rcpd}
Given $P\in \mathscr{P}(\Omega_0)$, there exists a regular conditional probability distribution $P(\cdot|\mathcal{B}_{T}^0)(\omega)$, $\omega\in\Omega_{0}$, of $P$ with respect to $\mathcal{B}_{T}^0$ such that
\begin{enumerate}
	\item For every $\omega\in\Omega_{0}$, $P(\cdot|\mathcal{B}_{T}^0)(\omega)$ is a probability measure on $(\Omega_{0},{\mathcal{B}})$.
	\item For every $A\in{\mathcal{B}}$, the mapping $\omega\mapsto P(A|\mathcal{B}_{T}^0)(\omega)$ is $\mathcal{B}^{0}_{T}$-measurable.
	\item There exists a $P$-null set $N\in\mathcal{B}_{T}^0$ such that for any $\omega\notin N$
	$$P\left(\{\tilde{\omega};\,x(s,\tilde{\omega})=x(s,\omega),0\leq s\leq T\}|\mathcal{B}_{T}^0\right)(\omega)=1.$$
	\item For any  set $A\in \mathcal{B}_{T}^0$ and any  set $B\in\mathcal{B} $
	$$P\left(x|_{[0,T]}\in A, x|_{[T,\infty)}\in B\right)=\int_{\tilde{\omega}|_{[0,T]}\in A}P\left(x|_{[T,\infty)}\in B|\mathcal{B}_{T}^0\right)(\tilde\omega)\dif P(\tilde\omega).$$
\end{enumerate}
\et

According to   \cite[Theorem~6.1.2]{SV79} we obtain the following reconstruction result.

\bt\label{th:m2}
Let $\omega\mapsto Q_\omega$ be a mapping from $\Omega_0$ to $\mathscr{P}(\Omega_0)$ such that for any $A\in {\mathcal{B}}$, $\omega\mapsto Q_\omega(A)$ is $\mathcal{B}^0_T$-measurable and for any $\omega\in \Omega_0$
$$Q_\omega\big(\tilde{\omega}\in \Omega: x(T,\tilde{\omega})=x(T,\omega)\big)=1.$$
Then for any $P\in \mathscr{P}(\Omega_0)$, there exists a unique $P\otimes_{T}Q\in \mathscr{P}(\Omega_0)$ such that
$$(P\otimes_{T} Q)(A)=P(A)\  \mbox{for all}\  A\in \mathcal{B}_{T}^0,$$
and for $P\otimes_{T} Q$-almost all $\omega\in \Omega_0$
$$\delta_\omega\otimes_{T} Q_\omega=(P\otimes_{T}Q)\left(\cdot|\mathcal{B}_{T}^0\right)(\omega).$$

\et

  We say $P\in \mathscr{P}_{L^2_\sigma}(\Omega_0)\subset \mathscr{P}(\Omega_0)$ is concentrated on the paths with values in $L^2_\sigma$ if there exists $A\in \mathcal{B}$ with $P(A)=1$ such that $A\subset \{\omega\in\Omega_0;\,\omega(t)\in L^2_\sigma \ \mbox{for all}\  t\geq0\}$.
It is clear that $\mathcal{B}(\sP_{L^2_\sigma}(\Omega_0))=\mathcal{B}(\sP(\Omega_0))\cap \sP_{L^2_\sigma}(\Omega_0)$. We also use $\mathrm{Comp}(\sP_{L^2_\sigma}(\Omega_0))$ to denote the space of all compact subsets of $\sP_{L^2_\sigma}(\Omega_0)$. Throughout the rest of this section we fix the constants $C_{p}$, $p\in\N$, as in Proposition~\ref{prp:ext2} and  we denote  by $\sC(x_0,C_{p})$ the set of dissipative  martingale solutions starting from $x_0\in L^2_\sigma$ at time $s=0$ and satisfying (M3) with  $C_{p}$, $p\in\N$.

The shift operator $\Phi_t:\Omega_0\to \Omega_0^t$ is defined by
$$\Phi_t(\omega)(s):=\omega(s-t), \quad s\geq t.$$

Now, we have all in hand to recall the definition of almost sure Markov process.

\bd\label{def:ad}
A family $(P_{x_0})_{x_0\in L^2_\sigma}$ of probability measures in $\mathscr{P}_{L^2_\sigma}(\Omega_0)$, is called an almost sure Markov family provided
\begin{enumerate}
\item  for every $A\in\mathcal{B}$, the mapping $x_0\mapsto P_{x_0}(A)$ is $\mathcal{B}(L^2_\sigma)/\mathcal{B}([0,1])$-measurable,
\item for  every $x_0\in L^2_\sigma$ there exists a set $\mathfrak{T}\subset(0,\infty)$ with zero Lebesgue measure such that for $T\notin \mathfrak{T}$
$$P_{x_0}\left(\cdot|\mathcal{B}_{T}^0\right)(\omega)=P_{\omega(T)}\circ \Phi_{T}^{-1}.$$
\end{enumerate}
\ed

An almost surely Markov family can be obtained from a so-called almost sure pre-Markov family through a selection procedure.

\bd\label{def:pre}
Let the mapping $L^2_\sigma\to \mathrm{Comp}(\sP_{L^2_\sigma}(\Omega_0))$, $  x_0\mapsto \sC(x_0,C_{p}) ,$ be Borel measurable. We say that $(\sC(x_0))_{x_0\in L^2_\sigma}$ forms an almost surely pre-Markov family if for each $x_0\in L^2_\sigma$, $P\in \sC(x_0,C_{p})$ there exists a set $\mathfrak{T}\subset (0,\infty)$ with zero Lebesgue measure such that the following holds true for $T\notin \mathfrak{T}$
\begin{enumerate}
  \item (Disintegration) there is a $P$-null set $N\in \mathcal{B}_{T}^0$ such that for $\omega\notin N$,
  $$x(T,\omega)\in L^2_\sigma, \quad P\left(\Phi_{T}(\cdot)|\mathcal{B}_{T}^0\right)(\omega)\in \sC(x(T,\omega),C_{p}),$$
  \item (Reconstruction) if a mapping $\Omega_0\to\sP_{L^2_\sigma}(\Omega_0)$, $ \omega\mapsto Q_\omega,$ satisfies the assumptions of Theorem~\ref{th:m2} and there is a $P$-null set $N\in\mathcal{B}_{T}^0$ such that for all $\omega\notin N$
      $$x(T,\omega)\in L^2_\sigma,\quad Q_\omega\circ \Phi_{T}\in \sC(x(T,\omega),C_{p}),$$
      then $P\otimes_T Q\in \sC(x_0,C_{p})$.
\end{enumerate}
\ed

Finally, we  recall the following abstract Markov selection theorem, cf.  \cite[Theorem~2.8]{FR08}.   The only missing point is  the maximization of a given functional, however, this can be achieved easily by choosing  this functional in the selection procedure as the first functional to be maximized.

\bt\label{selection}
Let $(\sC(x_0,C_{p}))_{x_0\in L^2_\sigma}$ be an almost surely pre-Markov family. Suppose that for each $x_0\in L^2_\sigma$, $\sC(x_0,C_{p})$ is non-empty and convex. Then there exists a measurable
selection $L^2_\sigma\to\sP_{L^2_\sigma}(\Omega_0)$, $ x_0\mapsto P_{x_0} ,$ such that $P_{x_0}\in\sC(x_0,C_{p})$ for every $x_0\in L^2_\sigma$ and $(P_{x_0})_{x_0\in L^2_\sigma}$ is an almost surely Markov family. In addition, if $F:L^2_\sigma\to \R$ is a bounded continuous function and $\lambda>0$ then  the selection can be chosen for every $x_0\in L^2_\sigma$ to maximize
\begin{equation}\label{eq:func}
\E^{P}\left[\int_{0}^{\infty}e^{-\lambda s}F(x(s))\dif s\right]
\end{equation}
among all martingale solutions $P$ with the initial condition $x_0$.
\et

As the next step, we verify  that $(\sC(x_0,C_{p}))_{x_0\in L^2_\sigma}$ has the disintegration as well as the reconstruction property from Definition \ref{def:pre}.

\bl
The family $(\sC(x_0,C_{p}))_{x_0\in L^2_\sigma}$ satisfies the disintegration property in Definition \ref{def:pre}.
\el

\begin{proof}
Fix $x_0\in L^2_\sigma$, $P\in \sC(x_0,C_{p})$ and let $T$ be a regular time of $P$.  Let $P(\cdot|\mathcal{B}_{T}^0)(\omega)$ be a regular conditional probability distribution of $P$ with respect to $\mathcal{B}_{T}^0$. We want to show that there is a $P$-null set $N\in \mathcal{B}_{T}^0$ such that for all $\omega\notin N$
$$x(T,\omega)\in L^2_\sigma,\quad P\left(\Phi_{T}(\cdot)|\mathcal{B}_{T}^0\right)(\omega)\in \sC(x(T,\omega),C_{p}).$$

Next, we shall verify that $P(\Phi_{T}(\cdot)|\mathcal{B}_{T}^0)(\omega)$ satisfies the conditions (M1), (M2), (M3) in Definition~\ref{martingale solution} with the initial condition $x(T,\omega)$ and the initial time $0$ or alternatively that $P(\cdot|\mathcal{B}_{T}^0)(\omega)$ satisfies (M1), (M2), (M3) in Definition~\ref{martingale solution} with the initial condition $x(T,\omega)$ and the initial time $T$.

(M1): Due to (3) from Theorem \ref{rcpd}, it follows that outside of a $P$-null set in $\mathcal{B}^{0}_{T}$, it holds
$$
P\left(\{ \tilde\omega;\,x(T,\tilde\omega)=x(T,\omega) \}  |\mathcal{B}_{T}^0\right)(\omega)=1.
$$
In other words, $P(\cdot|\mathcal{B}_{T}^0)(\omega)$ has the correct initial value at the initial time $T$.
Set
$$S_T=\{\omega\in \Omega_0: \omega|_{[0,T]}\in L^\infty([0,T];L^2_\sigma)\cap L^{2}([0,T];H^{\gamma})\},$$
$$S^T=\{\omega\in \Omega_0: \omega|_{[T,\infty)}\in L^\infty_{\rm loc}([T,\infty);L^2_\sigma)\cap L^{2}_{\rm loc}([T,\infty);H^{\gamma})\}.$$
We obtain by (M1) for $P$ that
$$1=P(S_T\cap S^T)=\int_{S_T}P(S^T|\mathcal{B}^{0}_T)(\tilde\omega)\dif P(\tilde\omega).$$
Consequently, it holds $P(S^T|\mathcal{B}^{0}_T)(\tilde\omega)=1$ for $P$-a.s. $\tilde{ \omega}$.
We denote the union of the above two  $P$-null sets by $N_1$.

(M2): Using \cite[Proposition B.1]{FR08} (cf. \cite[Lemma B.3]{GRZ09}), there exists a $P$-null set $N_2\in \mathcal{B}^{0}_{T}$
such that for all ${\omega}\notin N_2$, $P( \Phi_{T}(\cdot)|\mathcal{B}^{0}_{T})(\omega)$ satisfies (M2).

(M3): Similarly, \cite[Proposition B.4]{FR08} implies that for $P$-a.s. $\tilde{ \omega}$ that $(E_t^p)_{t\geq T}$ is an almost sure $((\mathcal{B}_t^0)_{t\geq T},P(\cdot|\mathcal{B}_T^0)(\tilde{ \omega}))$-supermartingale. This gives a $P$-null set $N_3$.

We complete the proof by choosing the null set $N=\cup_{i=1}^3N_i$.
\end{proof}

\bl
The family $(\sC(x_0,C_{p}))_{x_0\in L^2_\sigma}$ satisfies the reconstruction property in Definition \ref{def:pre}.
\el

\begin{proof}
Fix $x_0\in L^2_\sigma$, $P\in \sC(x_0,C_{p})$ and let $T$ be a regular time of $P$. Let $Q_\omega$ be as in Definition~\ref{def:pre} (2). We shall show that  $P\otimes_{T}Q\in\sC(x_0,C_{p})$ hence we need to verify the conditions (M1), (M2), (M3) from Definition~\ref{martingale solution}.
This follows by the same arguments as in  Proposition \ref{prop:2} with the choice $\tau=T$.
\end{proof}

\begin{proof}[Proof of Theorem \ref{th:ma2}]
The proof follows the lines of   \cite[Theorem 12.2.4]{SV79}. In particular, in view of the non-uniqueness in law from Proposition~\ref{prp:ext2} there exists an initial value $x_{0}\in L^{2}_{\sigma}$ giving raise to  at least two martingale solutions on $[0,\infty)$. In particular, there are   martingale solutions $P$, $Q$ and a functional $F:L^2_\sigma\to \R$ such that
$$
\E^{P}\left[\int_{0}^{\infty}e^{-\lambda s}F(x(s))\dif s\right]>\E^{Q}\left[\int_{0}^{\infty}e^{-\lambda s}F(x(s))\dif s\right].
$$
Now, applying Theorem~\ref{selection} once with $F$ and once with $-F$ we obtain selections $(P^{+}_{x_0})_{x_0\in L^2_\sigma}$ and $(P^{-}_{x_0})_{x_0\in L^2_\sigma}$, respectively. In particular, it holds
$$
\E^{P^{+}_{x_0}}\left[\int_{0}^{\infty}e^{-\lambda s}F(x(s))\dif s\right]\geq \E^{P}\left[\int_{0}^{\infty}e^{-\lambda s}F(x(s))\dif s\right],
$$
and
$$
\E^{Q}\left[\int_{0}^{\infty}e^{-\lambda s}F(x(s))\dif s\right] \geq \E^{P^{-}_{x_0}}\left[\int_{0}^{\infty}e^{-\lambda s}F(x(s))\dif s\right].
$$
In other words, the two  Markov selections are different.
\end{proof}

\subsection{Non-unique semiflows in the deterministic setting}\label{s:det}

\begin{proof}[Proof of Corollary~\ref{cor:det}]
Let us just give a brief sketch of the proof. The convex integration result of Theorem~\ref{Main results1} applies in particular to the case $G=0$. Note that in this case $z\equiv 0$ and therefore the stopping time $\mathfrak{t}$ defined through \eqref{stopping time} equals to one and does not play any role in the analysis. But actually already the seminal paper by Buckmaster and Vicol \cite{BV19a} provides the necessary non-uniqueness result of weak solutions with a prescribed energy. These solutions can be simply extended to $[0,\infty)$ by Leray solutions. In view of Remark~\ref{r:CG}, one can then formulate  the corresponding relaxed energy inequality
similarly to Definition~\ref{martingale solution} and verify the stability, shift and concatenation property of the set of solutions required by the selection result of Cardona and Kapitanskii \cite{CorKap}. This  implies existence of a semiflow. However, the semiflow property is only satisfied for a.e. time, similarly to the a.s. Markov property. By including the energy as an additional variable similarly to  \cite{Bas20,BreFeiHof19,BreFeiHof19B}, it is possible to  obtain a semiflow in the usual sense. The non-uniqueness then follows by  a similar argument as in the proof of Theorem~\ref{th:ma2}.
\end{proof}

\section{Construction of global probabilistically strong solutions}
\label{s:in}

This section is devoted to the proof of our last  main result,  Theorem~\ref{th:ma4}. The  goal is to establish existence of non-unique global-in-time probabilistically strong solutions to the Navier--Stokes system \eqref{1} for every given divergence free initial condition in $L^{2}_{\sigma}$. More precisely, we intend to overcome the limitation given by the stopping time required in almost all the available  convex integration results in the stochastic setting (except for \cite{CFF19}) and also in Section~\ref{s:1.1}. In \cite{CFF19}, stopping times were not needed due to a suitable transformation which only seems to work  in an Euler setting with a linear multiplicative noise, cf. also \cite[Remark 2.2]{HZZ20}. Recall that the stopping time was used in order to control the noise throughout the convex integration scheme uniformly in the randomness variable $\omega$.  In addition, the initial value was part of the construction and as such could not have been prescribed.

The main idea of this section is to use convex integration  to construct probabilistically strong solutions for every given initial condition in $L^{2}_{\sigma}$. While these solutions also only exist up to a suitable stopping time, we may repeat the construction: we use the final value at the stopping time of  any such convex integration solution  as a new initial condition for the convex integration procedure. This way we are able  to extend the convex integration solutions as probabilistically strong solutions defined on the whole time interval $[0,\infty)$. It is therefore not necessary to pass to the framework of martingale solutions as done in Section~\ref{sec:mar}.

In order to prove  Theorem~\ref{th:ma4} we modify the convex integration scheme developed in Section \ref{s:1.1}. The notation remains mostly the same and we explain the necessary changes in the sequel. First of all, we intend to prescribe an arbitrary random initial condition $u_{0}\in L^{2}$ $\mathbf{P}$-a.s. independent of the given Wiener process $B$. Therefore, let $(\mathcal{F}_{t})_{t\geq0}$ be the augmented joint canonical filtration on $(\Omega,\mathcal{F})$ generated by $B$ and $u_{0}$. Then $B$ is a $(\mathcal{F}_{t})_{t\geq0}$-Wiener process and $u_{0}$ is $\mathcal{F}_0$-measurable.
We include the initial value into the linear part $z$, namely, we let  $z$ satisfy the stochastic Stokes equation \eqref{linear} with $z(0)=u_0$.
As before, the iteration is indexed by a parameter $q\in\mathbb{N}_{0}$. At each step $q$, a pair $(v_q, \mathring{R}_q)$ is constructed solving the following system
\begin{equation}\label{induction ps}
\aligned
\partial_tv_q-\Delta v_q +\div((v_q+{z_q})\otimes (v_q+{z_q}))+\nabla p_q&=\div \mathring{R}_q,
\\
\div v_q&=0,\\
v_{q}(0)&=0.
\endaligned
\end{equation}
{Here we decompose $z=z^{in}+Z$ with $z^{in}(t)=e^{t\Delta}u_{0}$ and define $z_q:=z^{in}+Z_q=z^{in}+\mP_{\leq f(q)}Z$ with $f(q)=\lambda_{q+1}^{\alpha/8}$.} Note that the above yields the correct initial value $u(0)=v(0)+z(0)=\lim_{q\to\infty}v_{q}(0)+u_{0}=u_{0}$. 
The parameters $\lambda_q$, $ \delta_q$, $q\in\mathbb{N}_{0}$, as well as the mollification parameter $\ell$ all  retain the same structure as in Section \ref{s:1.1}, but the value of the determining parameters $a$,  $b$, $\alpha$, $\beta$ will be possibly different as further conditions need to be satisfied. Details are given in Section~\ref{s:par} below.

Let $L\geq 1$ sufficiently large be given.   Let $N\geq 1$ be given  and assume in addition that $\mathbf{P}$-a.s.
\begin{equation}\label{eq:u0}
\|u_{0}\|_{L^{2}}\leq N.
\end{equation}
We keep this additional assumption on the initial condition throughout the convex integration step in Proposition~\ref{p:iteration} and relax it later in the proof of Theorem~\ref{thm:6.1}.
 Define the following stopping time for $0<\delta<\frac1{12}$ 
\begin{equation}\label{stopping time ps}
\aligned
T_L:=&\inf\{t\geq0, \|Z(t)\|_{H^{1-\delta}}\geq L/ C_S\}\wedge \inf\{t\geq0,\|Z\|_{C_t^{1/2-2\delta}L^2}\geq L/C_S\}
\wedge  L,
\endaligned
\end{equation}
which is $\mathbf{P}$-a.s. strictly positive.
Moreover, for  $t\in[0, T_L]$ it holds
\begin{equation}\label{z ps}
\| Z_q(t)\|_{L^\infty}\leq L\lambda_{q+1}^{\frac\alpha8}, \quad\|\nabla Z_q(t)\|_{L^\infty}\leq L\lambda_{q+1}^{\frac\alpha4}, \quad \|Z_q\|_{C_t^{\frac{1}{2}-2\delta}L^\infty}\leq L\lambda_{q+1}^{\frac\alpha4}.
\end{equation}
We also suppose that there is a deterministic constant $M_{L}(N)\geq (L+N)^2$ such that $\|z_q\|_{L^2}\leq \|z^{in}\|_{L^{2}}+\|Z_q\|_{L^{2}}\leq M_{L}(N)^{1/2}$. In the following we write $M_L$ instead of $M_L(N)$ for simplicity.
We denote $A=4M_L$, $\sigma_{q}=2^{-q}$, $q\in\mathbb{N}_{0}\cup\{-1\}$,   $\gamma_q=2^{-q}$, $q\in\N_{0}\setminus\{3\}$, and  $\gamma_3=K$. Here  $K\geq1$ is a large constant which used  in the proof of Theorem~\ref{thm:6.1} to distinguish different solutions.

Under the above assumptions, our main iteration  reads as follows.

\begin{proposition}\label{p:iteration}
Let $L,N\geq 1$ and assume \eqref{eq:u0}. There exists a choice of parameters $a, b, \beta$ such that the following holds true: Let $(v_{q},\mathring{R}_{q})$ for some $q\in\N_{0}$ be an $(\mathcal{F}_{t})_{t\geq 0}$-adapted solution to \eqref{induction ps} satisfying
\begin{equation}\label{inductionv ps}
\|v_{q}(t)\|_{L^{2}}\leq\begin{cases}
M_0 (M_L^{1/2}\sum_{ r=1}^{q}\delta_{r}^{1/2}+\sum_{r=1}^{q}\gamma_{r}^{1/2})+3M_0(M_L+qA)^{1/2},&t\in (\frac{\sigma_{q-1}}2\wedge T_L,  T_L],\\
0, &t\in [0,\frac{\sigma_{q-1}}2\wedge T_L],
\end{cases}
\end{equation}
for a universal constant $M_0$,
\begin{align}\label{inductionv C1}
 \|v_q\|_{C^1_{t,x}}&\leq \lambda_q^4M_L^{1/2},\quad t\in[0, T_L],
\end{align}
\begin{align}\label{eq:R}
\|\mathring{R}_q(t)\|_{L^1}\leq
\delta_{q+1}M_L,\quad t\in (\sigma_{q-1}\wedge T_L,T_L],
\end{align}
\begin{align}\label{bd:R}
\|\mathring{R}_{q}(t)\|_{L^1}\leq M_L+qA,\quad t\in[0, T_L].
\end{align}
 Then    there exists an $(\mathcal{F}_{t})_{t\geq 0}$-adapted process $(v_{q+1},\mathring{R}_{q+1})$ which solves \eqref{induction ps} and satisfies
 \begin{equation}\label{iteration ps}
\|v_{q+1}(t)-v_q(t)\|_{L^2}\leq \begin{cases}
M_0 (M_L^{1/2}\delta_{q+1}^{1/2}+\gamma_{q+1}^{1/2}),& t\in (4\sigma_q\wedge T_L,T_L],\\
M_0 ((M_L+qA)^{1/2}+\gamma_{q+1}^{1/2}),& t\in (\frac{\sigma_q}2\wedge T_L,4\sigma_q\wedge T_L],\\
0,
&t\in [0,\frac{\sigma_q}2\wedge T_L],
\end{cases}
\end{equation}
\begin{equation}\label{iteration R}
\|\mathring{R}_{q+1}(t)\|_{L^1}\leq\begin{cases}
M_L\delta_{q+2},& t\in (\sigma_q\wedge T_L,T_L],\\
M_L\delta_{q+2}+\sup_{s\in[t-\ell,t]}\|\mathring{R}_{q}(s)\|_{L^1},&t\in (\frac{\sigma_q}2\wedge T_L,\sigma_q\wedge T_L],\\
\sup_{s\in[t-\ell,t]}\|\mathring{R}_{q}(s)\|_{L^1}+A&t\in [0,\frac{\sigma_q}2\wedge T_L].
\end{cases}
\end{equation}
Consequently, $(v_{q+1},\mathring{R}_{q+1})$ obeys \eqref{inductionv ps},  \eqref{inductionv C1}, \eqref{eq:R} and \eqref{bd:R} at the level $q+1$.
Furthermore, for $1<p=\frac{31}{30}<\frac{16}{15}$, $t\in [0,T_L]$ it holds
\begin{align}\label{induction w}
\|v_{q+1}(t)-v_q(t)\|_{W^{\frac12,p}}\leq
M_0M_L^{1/2}\delta_{q+1}^{1/2}
\end{align}
and for $t\in (4\sigma_q\wedge T_L,T_L]$ we have
\begin{align}\label{p:gamma}
\big|\|v_{q+1}\|_{L^2}^2-\|v_q\|_{L^2}^2-3\gamma_{q+1}\big|\leq 7M_L\delta_{q + 1}.
\end{align}

\end{proposition}

The proof of this result is presented in Section~\ref{s:it} below.

\begin{remark}
Note that based on \eqref{iteration ps} together with the fact that $v_{0}=0$ on $[0,T_{L}]$,  the first bound in \eqref{inductionv ps} on the level $q+1$, $q\in\mathbb{N}$, may be improved by replacing $qA$ on the right hand side by $(q-1)A$. Since this is never needed in the proof of Proposition~\ref{p:iteration} and is not necessarily valid in case $q-1=-1$, i.e. for $q=0$, we postulated \eqref{inductionv ps} in the present form.
\end{remark}

\begin{remark}\label{r:5.3}
Comparing to Theorem C in \cite{BMS20}, we only obtain solutions  in $L^{p}(0,T_{L};L^{2})$ for every $p\in[1,\infty)$ but not for $p=\infty$. This is a consequence of the term $qA$ present in \eqref{bd:R} which propagates to all the other estimates and comes from the control of the $z$-part, cf. the case {\bf III.} in Section~\ref{s:R} below. In our approach, the initial condition is included into $z$ whereas \cite{BMS20} include it into $v_{0}$. Consequently, the $C^{1}$-norm of their iterations $v_{q}$ is generally infinite, i.e. \eqref{inductionv C1} fails, and hence they need to control the smallness of $R_{q}-R_{\ell}$ directly. This is possible in the deterministic setting  using uniform continuity but in the stochastic setting one would require a compact range uniformly in $\omega$ in the spirit of \cite{BFH20,HZZ20}. This is the reason why we proceeded differently at this point.
\end{remark}

We intend to start the iteration from $v_{0}\equiv 0$ on $[0,T_{L}]$. In that case, we have
$\mathring{R}_0=z_{0}\mathring\otimes z_{0}$
so that
\begin{align*}
\|\mathring{R}_0(t)\|_{L^1}\leq M_L
\end{align*}
and \eqref{eq:R} as well as \eqref{bd:R} are satisfied on the level  $q=0$, since $\delta_{1}=1$.

We deduce the following result.

\bt\label{thm:6.1}
There exists a $\mathbf{P}$-a.s. strictly positive stopping time $T_L$, arbitrarily large by choosing $L$ large, such that for any   initial condition $u_0\in L^2_\sigma$ $\mathbf{P}$-a.s. independent of the Brownian motion $B$
the following holds true: There exists an $(\mathcal{F}_t)_{t\geq0}$-adapted process $u$ which belongs to $L^p(0,T_L;L^2)\cap C([0,T_L],W^{\frac12,\frac{31}{30}})$ $\mathbf{P}$-a.s. for all $p\in[1,\infty)$ and is an analytically weak solution to \eqref{1} with $u(0)=u_0$.  There are infinitely many such solutions $u$.
\et

\begin{proof}
Let the additional assumption  \eqref{eq:u0} be satisfied  for some $N\geq 1$. Letting $v_{0}\equiv 0$, we repeatedly apply Proposition~\ref{p:iteration} and obtain $(\mathcal{F}_{t})_{t\geq0}$-adapted processes $(v_{q},\mathring{R}_{q})$, $q\in\N$, such that  $v_q \to v$ in $C([0,T_L];W^{\frac12,\frac{31}{30}})$ as a consequence of  \eqref{induction w}. Moreover, using {\eqref{iteration ps}} we
have for every $p\in[1,\infty)$
$$
   \int_0^{T_L}  \| v_{q + 1} - v_q \|_{L^2}^p d t
 \leqslant \int_{\sigma_q / 2 \wedge T_L}^{4 \sigma_q \wedge
   T_L} \| v_{q + 1} - v_q \|_{L^2}^p d t +
   \int_{4 \sigma_q \wedge T_L}^{T_L} \| v_{q + 1} - v_q \|_{L^2}^p d t
   $$
$$ \lesssim \int_{\sigma_q / 2 \wedge T_L}^{4 \sigma_q \wedge T_L}
   M^p_0 ((M_L + q A)^{1/2} + \gamma^{1/2}_{q + 1})^{p} d t
    + \int_{4 \sigma_q \wedge T_L}^{T_L} M_0^p (M_L^{1/2} \delta^{1/2}_{q + 1}
   + \gamma^{1/2}_{q + 1})^{p} d t
   $$
$$ \lesssim M_{0}^{p}\left( 2^{- q} ((M_L + q A)^{1/2}  + \gamma_{q+1}^{1/2})^{p}
   + T_{L}(M_L^{1 / 2}\delta_{q+1}^{1/2} + \gamma_{q+1}^{1/2})^{p} \right).
   $$
Thus, the sequence $v_{q}$, $q\in\N$, is  Cauchy hence converging in $L^{p}(0,T_{L};L^{2})$ for all $p\in[1,\infty)$.
Accordingly, $v_q\to v$ also in $L^p(0,T_{L};L^{2})$.
Furthermore, by \eqref{eq:R}, \eqref{bd:R} we know for all $p\in[1,\infty)$
	$$
	\int_0^{T_L}\|\mathring{R}_q(t)\|^{p}_{L^1}\dif t\lesssim M_L^{p}\delta^{p}_{q+1}T_L+(M_L+qA)^{p}2^{-q} \to 0, \quad\mbox{as}\quad q\to\infty.
	$$
Thus, the process $u=v+z$ satisfies \eqref{1} before $T_L$ in the analytic weak sense. Since $v_q(0)=0$ we deduce $v(0)=0$, which implies that $u(0)=u_0$.
	
Next, we prove non-uniqueness of  the constructed solutions, still  under the additional assumption \eqref{eq:u0}.	In view of  \eqref{p:gamma}, we have on $t\in (4\sigma_{0}\wedge T_L,T_L]$
\begin{align}\label{eq:K}
\begin{aligned}
\big|\|v\|_{L^2}^2-3K\big|&\leq \left|\sum_{q=0}^\infty(\|v_{q+1}\|_{L^2}^2-\|v_{q}\|_{L^2}^2-3\gamma_{q+1})\right|+3\sum_{q\neq2}\gamma_{q+1}
\\
&\leq 7M_L\sum_{q=0}^{\infty}\delta_{q+1}+3\sum_{q\neq 2}\gamma_{q+1}=:c
\end{aligned}
\end{align}
and this implies non-uniqueness by choosing different $K$. More precisely, for a given  $L\geq 1$ sufficiently large it holds $\mathbf{P}(4\sigma_{0}<T_{L})>0$. The parameters $L,N$  determine $M_{L}(N)$ and consequently by choosing different $K=K(L,N)$ and $K'=K'(L,N)$ so that $3|K-K'|>2c$ we deduce that the corresponding solutions $v_{K}$ and $v_{K'}$ have different $L^{2}$-norms on the set
$\{4\sigma_0<T_L\}.$ The  velocities $u_{K}=v_{K}+z$ and $u_{K'}=v_{K'}+z$ are therefore different as well.
	
For a general initial condition $u_{0}\in L^{2}_{\sigma}$ $\mathbf{P}$-a.s., define $\Omega_{N}:=\{N-1\leq \|u_{0}\|_{L^{2}}< N\}\in\mathcal{F}_{0}$. Then the first part of this proof gives the existence of infinitely many adapted solutions $u^{N}$ on each $\Omega_{N}$. Letting $u:=\sum_{N\in\mathbb{N}}u^{N}1_{\Omega_{N}}$ concludes the proof.
	\end{proof}

Now we have all in hand to complete the proofs of Theorem \ref{th:ma4} and Corollary~\ref{cor:law}.

\begin{proof}[Proof of Theorem \ref{th:ma4}]
By Theorem~\ref{thm:6.1} we constructed a probabilistically strong solution $u$ before the stopping time $T_L$ starting from the given initial condition $u_0\in L^{2}_{\sigma}$ $\mathbf{P}$-a.s. Since $T_L>0$ $\mathbf{P}$-a.s., we know for $\mathbf{P}$-a.e. $\omega$ that there exists $q_0(\omega)$ such that $4\sigma_{q_0(\omega)}<T_L(\omega)$.
By \eqref{iteration ps} we find
	\begin{align*}
	\|u(T_L)\|_{L^2}&\leq \|z(T_L)\|_{L^2}+\sum_{0\leq q< q_0}\|v_{q+1}(T_L)-v_q(T_L)\|_{L^2}+\sum_{ q\geq q_0}\|v_{q+1}(T_L)-v_q(T_L)\|_{L^2}
	\\
	&\lesssim M_L^{{1/2}}+M_0q_0(M_L+q_0A)^{1/2}+M_0(K^{1/2}+1)+M_0M_L^{1/2}<\infty.
	\end{align*}
	This implies that $\|u(T_L)\|_{L^2}<\infty$ $\mathbf{P}$-a.s. hence we can use the value $u(T_{L})$ as a new initial condition in Theorem~\ref{thm:6.1}. More precisely, now we aim at solving the original Navier--Stokes system \eqref{1} on the time interval $[T_{L},T_{L}+\hat{T}_{L+1}]$ by applying Theorem~\ref{thm:6.1} with some  stopping time $\hat{T}_{L+1}>0$ $\mathbf{P}$-a.s.
	
	To this end, we define $u_1(0)=u(T_L), \hat{B}_t=B_{T_L+t}-B_{T_L}$ and  $\hat{\mathcal{F}}_t=\sigma(\hat{B}_s,s\leq t)\vee \sigma(u(T_L))$. Let  $\hat{Z}$ solve equation \eqref{linear} with $\hat{Z}(0)=0$ and $B$ replaced by $\hat{B}$. It is straightforward to check that $Z(t+T_L)-e^{t\Delta}Z(T_L)$ satisfies the same equation as $\hat{Z}$ with the same initial data. As a result, $\hat{Z}(t)=Z(t+T_L)-e^{t\Delta}Z(T_L)$, which we use for the definition of the stopping time
	\begin{equation}\label{stopping time ps1}
\aligned
\hat T_{L+1}&:=\inf\{t\geq 0, \|\hat{Z}(t)\|_{{H^{1-\delta}}}\geq 2(L+1)/ C_S\}\\
&\qquad\wedge \inf\{t\geq0,\|\hat{Z}\|_{C_t^{1/2-2\delta}{L^2}}\geq 2(L+1)/C_S\}
\wedge  (L+1).
\endaligned
\end{equation}
Then we find $T_{L+1}-T_L\leq \hat{T}_{L+1}$ since for $t\leq T_{L+1}-T_L$ it holds
	$$\|\hat{Z}(t)\|_{{H^{1-\delta}}}\leq \|Z(t+T_L)\|_{{H^{1-\delta}}}+\|Z(T_L)\|_{{H^{1-\delta}}}\leq (2L+1)/C_S< 2(L+1)/C_S$$
	and
		\begin{align*}
		\|\hat{Z}\|_{C_t^{1/2-2\delta}{L^2}}&\leq \|Z(\cdot+T_L)\|_{C_t^{1/2-2\delta}{L^2}}+\|e^{\Delta\cdot}Z(T_L)\|_{C_t^{1/2-2\delta}{L^2}}
		\\&\leq (L+1)/C_{S}+\|Z(T_L)\|_{{H^{1-\delta}}}\leq  (2L+1)/C_S< 2(L+1)/C_S.
		\end{align*}
Similarly to the proof of Theorem~\ref{thm:6.1} we define $\Omega_N:=\{N-1\leq \|u(T_L)\|_{L^2}< N\}\in \hat{\mathcal{F}}_{0}$ and   construct a solution $u_1^N$ on $\Omega_N$ which solves the system with $B$ replaced by $\hat B$ on $[0,\hat{T}_{L+1}]$.
Finally, we deduce that $\bar{u}_1=\sum_{N\in\N} u_1^N1_{\Omega_N}$ is a  solution  to  \eqref{1}  on $[0,\hat{T}_{L+1}]$ with $B$ replaced by $\hat{B}$, which is still adapted to  $(\hat{\mathcal{F}}_t)_{t\geq0}$.

	Thus, define $u_{1}(t)=u(t)1_{t\leq T_L}+ \bar{u}_1(t-T_L)1_{t>T_L}$. For $t\leq T_L$, it is easy to see $u_1$ satisfies the equation \eqref{1} before $T_L$. For $T_L<t\leq T_{L+1} \leq T_{L}+\hat T_{L+1}$ we know
	\begin{align*}
	\bar{u}_1(t-T_L)&=u(T_L)+\int_0^{t-T_L}(\Delta \bar{u}_1-\mathbb{P}\div(\bar{u}_1\otimes \bar{u}_1))(s)\,\dif s+B(t)-B({T_L})
	\\&=u_0+\int_0^{T_L}(\Delta u-\mathbb{P}\div({u}\otimes {u}))(s)\,\dif s+\int_{T_L}^t(\Delta u_1-\mathbb{P}\div(u_1\otimes u_1))(s)\,\dif s+B(t)
	\\&=u_0+\int_{0}^t(\Delta u_1-\mathbb{P}\div(u_1\otimes u_1))(s)\,\dif s+B(t).
	\end{align*}
As a consequence,  $u_{1}$ satisfies the stochastic Navier--Stokes equation before $T_{L+1}$.

	 Now, we prove that $u_{1}$ is adapted to the natural filtration $(\mathcal{F}_t)_{t\geq0}$ generated by the Brownian motion $B$ together with the initial condition $u_{0}$. It is easy to see that $\hat{\mathcal{F}}_t\subset \mathcal{F}_{T_L+t}$. This implies that $\bar{u}_1(t-T_L)1_{t>T_L}$ is measurable with respect to $\mathcal{F}_{t}$. Indeed, we define
	$T_L^n:=\sum_{k=1}^\infty\frac{k}{2^n}1_{\{\frac{k-1}{2^n}\leq T_L<\frac{k}{2^n}\}}$ and we know that $T_L^n$ is stopping time and $T_L^n\downarrow T_L$.
	For every $t\geq 0$ $\bar{u}_1(t)\in \mathcal{F}_{T_L+t}\subset \mathcal{F}_{T_L^n+t}$. This implies that for every $t\geq0$ and every closed domain $D$ in $W^{\frac12,\frac{31}{30}}$
	$$\{\bar{u}_1(t)\in D\}\cap \{T_L^n+t\leq s\}\in \mathcal{F}_s.$$
	Then we find
	\begin{align*}
	&\{t>T_L^n\}\cap \{\bar{u}_1(t-T_L^n)\in D\}\\
&	\quad=\{t>T_L^n\}\cap \left(\cup_{\frac{k}{2^n}<t}\left(\left\{T_L^n=\frac{k}{2^n}\right\}\cap \left\{\bar{u}_1\left(t-\frac{k}{2^n}\right)\in D\right\}\cap \left\{T_L^n+t-\frac{k}{2^n}\leq t\right\}\right)\right)\in \mathcal{F}_t,
	\end{align*}
	which yields that $\bar{u}_1(t-T_L^n)1_{t>T_L^n}$ is measurable with respect to $\mathcal{F}_{t}$. Since $\bar{u}_1(t-T_L)1_{t>T_L}=\lim_{n\to\infty}\bar{u}_1(t-T_L^n)1_{t>T_L^n}$, we know $\bar{u}_1(t-T_L)1_{t>T_L}$ is measurable with respect to $\mathcal{F}_{t}$.
	
Now, we can iterate the above steps, i.e. starting from $u_k(T_{L+k})$ and constructing solutions $u_{k+1}$ before the stopping time $T_{L+k+1}$.
 Define $\bar u=u1_{t\leq T_L}+\sum_{k=1}^\infty u_k1_{\{T_{L+k-1}< t\leq T_{L+k} \}}$
	 and obtain $\bar u\in L^p_{\rm{loc}}([0,\infty); L^2)\cap C([0,\infty);W^{\frac12,\frac{31}{30}}) $, for all $p\in[1,\infty)$, is a probabilistically strong solution. We emphasize that $\bar u$ does not blow up at any finite time $T$ since for
	 any time $T$ we could find $k_0$ such $T<T_{L+k_0}$ and the infinite sum becomes a finite sum.
	Furthermore, as in the proof of Theorem~\ref{thm:6.1} we obtain infinitely many probabilistically strong solutions by choosing different $K$.
		\end{proof}

\begin{proof}[Proof of Corollary~\ref{cor:law}]
Let $L\geq 1$ be such that $\mathbf{P}(4\sigma_{0}<T_{L})>1/2$. With the notation from the proof of Theorem~\ref{thm:6.1} and particularly in view of \eqref{eq:K}, we choose again $K,$ $K'$ so that $3|K-K'|>2c$. The corresponding solutions then satisfy
$$
\mathbf{P}\left(\|v_{K}(T_L)\|_{L^{2}}^{2}\in [3K-c,3K+c]\right) >1/2,\qquad \mathbf{P}\left(\|v_{K'}(T_L)\|_{L^{2}}^{2}\in [3K'-c,3K'+c]\right) >1/2.
$$
Since the two intervals $[3K-c,3K+c]$ and $[3K'-c,3K'+c]$ are disjoint, the laws of $v_{K}$ and $v_{K'}$ are different. This carries over to the solutions $u_{K}$ and $u_{K'}$ as well as to the final solutions obtained at the end of the proof of Theorem~\ref{thm:6.1}.
\end{proof}

\subsection{Proof of Proposition~\ref{p:iteration}}
\label{s:it}

\subsubsection{Choice of parameters}
\label{s:par}

We take again the first two assumptions from \eqref{ell} as well as \eqref{ell1}.
For  the choice of $\ell$ we require
$\ell\leq \ell^{1/2}\vee (4\ell)\leq \sigma_q$, which is satisfied if  $a>2$.  In addition, we postulate $M_L+K+qA\leq \ell^{-1}$, which can be deduced by choosing $a$ large enough. We also need
$\ell^{1/2}\lambda_q^4<\lambda_{q+1}^{-\beta}$ and $\ell^{1/4}(M_L+qA+K)<\lambda_{q+1}^{-2b\beta}$, which requires $b(3\alpha-4\beta)>12$ and $3\alpha>16\beta b$. We further require {$M_0M_L^{1/2}(M_L+qA+K)<
\lambda_{q+1}^{\frac\alpha8(1-\delta)-2\beta b}$, which needs $\alpha>18\beta b$}. For given $\alpha>0$ satisfying $161\alpha<1/7$ we can choose $b\in 7\cdot 8\mN$ large enough such that $\alpha b\in 8\mN$. Consequently, \eqref{ell} is satisfied and assuming further that $\beta$ is small enough so that $4\beta<\alpha$, the condition $b(3\alpha-4\beta)>12$ holds.  Finally, we  make $\beta$ possibly smaller to fulfill {$\alpha>18\beta b$}. Then all the conditions are satisfied by choosing $a$ large enough.  We point out that by choosing different $M_{L},K$ we get different $a$.
However, the  initial value of the solution does not change.
Moreover,   the space $W^{\frac12,\frac{31}{30}}$, where  the convergence $v_{q}\to v$ in the proof of Theorem~\ref{thm:6.1} holds true, depends only on the structure of the parameters $r_\perp, r_{\|}$, i.e. the corresponding powers of $\lambda_{q+1}$, but remains  unchanged for different values of $a$, $b$, $\beta$.

\subsubsection{Construction of $v_{q+1}$}

Now, we extend $v_q, z_q, \mathring{R}_q$ and $p_q$ to $t<0$ by taking them equal to the value at $t=0$. Then $(v_q,z_q,\mathring{R}_q)$ also satisfies equation \eqref{induction ps} for $t<0$ as $\p_tv_q(0)=0$.
By  (\ref{inductionv C1}) and a similar argument as in  \eqref{error}  we know for $t\in[0, T_L]$
\begin{equation}\label{error ps}
\|v_q-v_\ell\|_{C_tL^2}\leq \ell \lambda_q^4M_L^{1/2}\leq \frac{1}{4} \delta_{q+1}^{1/2}M_L^{1/2},
\end{equation}
\begin{equation}\label{error ps1}
\|v_q-v_\ell\|_{C_tC^{\frac12}}\lesssim \ell^{1/2}\|v_q\|_{C_{t,x}^1}\lesssim \ell^{1/2}\lambda_q^4M_L^{1/2}\leq \frac{1}{4} \delta_{q+1}^{1/2}M_L^{1/2},
\end{equation}
where we  used the fact that $\ell^{\frac12}\lambda_q^4<\lambda_{q+1}^{-\beta}$ and we  chose $a$ large enough in order to absorb the implicit constant.
In addition, by \eqref{inductionv ps} it holds for $t\in(\frac{\sigma_q}2\wedge T_L, T_L]$
\begin{equation}\label{eq:vl ps}
\|v_\ell(t)\|_{L^2}\leq \|v_q\|_{C_{t}L^2}\lesssim  M_0(M_L^{1/2}+K^{1/2})+3M_0(M_L+qA)^{1/2},
\end{equation}
with some universal implicit constant.
For $N\geq1$, $t\in [0,T_L]$ we have by \eqref{ell}
\begin{equation}\label{eq:vl2 ps}
\|v_\ell\|_{C^N_{t,x}}\lesssim \ell^{-N+1}\|v_q\|_{C^1_{t,x}}\leq \ell^{-N+1}\lambda_q^4M_L^{1/2} \leq  \ell^{-N}\lambda_{q+1}^{-\alpha}M_L^{1/2},
\end{equation}
where we have chosen $a$ large enough to absorb the implicit constant.

Now we change the definition of $\rho$ with $\gamma_\ell$ replaced by the constant $\gamma_{q+1}$, namely,
$$
\rho:=2\sqrt{\ell^2+|\mathring{R}_\ell|^2}+\frac{\gamma_{q+1}}{(2\pi)^3},
$$
which changes \eqref{rho} into
\begin{equation}\label{rho ps}
\|\rho(t)\|_{L^p}\leq 2\ell (2\pi)^{3/p}+2\|\mathring{R}_\ell(t)\|_{L^p}+\gamma_{q+1}.
\end{equation}
In view of \eqref{eq:R} which holds on $(2\sigma_q\wedge T_L,T_L]$ and since $\supp\varphi_\ell\subset [0,\ell]$, we obtain for $t\in (4\sigma_q\wedge T_L,T_L]$
\begin{equation}\label{rho0 sp}
 \|\rho\|_{C^{0}_{[4\sigma_q\wedge T_L,t],x}}\lesssim \ell^{-4}\delta_{q+1}M_L+\gamma_{q+1},
\end{equation}
where we also used the embedding $W^{4,1}\subset L^\infty$.
Then, we deduce similarly as \eqref{rhoN} for $N\geq1$ and $t\in (4\sigma_q\wedge T_L,T_L]$
\begin{equation}\label{rhoN sp}
\aligned
\|\rho\|_{C^N_{[4\sigma_q\wedge T_L,t],x}}& \lesssim \ell^{-4-N}M_L\delta_{q+1}+\ell^{-N+1}(\ell^{-5}M_L\delta_{q+1})^N+\gamma_{q+1}\lesssim
\ell^{2 - 7 N} \delta_{q + 1} M_L+\gamma_{q+1}.
\endaligned
\end{equation}
For a general $t\in [0,T_L]$, we have by \eqref{bd:R}
\begin{equation}\label{rho0 sp1}
\|\rho\|_{C^{0}_{t,x}}\lesssim \ell^{-4}(M_L+qA)+\gamma_{q+1},
\end{equation}
 and for $N\geq 1$
\begin{equation}\label{rhoN sp1}
\aligned
\|\rho\|_{C^N_{t,x}}& \lesssim
\ell^{2 - 7 N}  (M_L+qA)+\gamma_{q+1},
\endaligned
\end{equation}
where we used $M_L+qA\leq \ell^{-1}$.

Next, we define the amplitude functions as in \eqref{amplitudes} with $\rho$ replaced by the above.
By using (\ref{rho ps}) for $t\in(4\sigma_q\wedge T_L, T_L]$
\begin{equation}\label{estimate a ps}
\begin{aligned}
\|a_{(\xi)}(t)\|_{L^2}&\leq \|\rho(t)\|_{L^1}^{1/2}\|\gamma_\xi\|_{C^0(B_{1/2}(\Id))}\\
&\leq\frac{M}{8|\Lambda|(1+8\pi^{3})^{1/2}}\left(2(2\pi)^3\ell+2\delta_{q+1}M_L+\gamma_{q+1}\right)^{1/2}\\
&\leq\frac{M}{4|\Lambda|}(M_L^{1/2}\delta_{q+1}^{1/2}+\gamma_{q+1}^{1/2}),
\end{aligned}
\end{equation}
where   $M$ denotes the universal constant from Lemma~\ref{geometric}.
From \eqref{rho0 sp}, \eqref{rhoN sp} and similarly to \eqref{eq:rho12} we deduce for $t\in (4\sigma_q\wedge T_L,T_L]$
\begin{align*}
\|\rho^{1/2}\|_{C^0_{[4\sigma_q\wedge T_L,t],x}}\lesssim \ell^{-2}\delta_{q+1}^{1/2}M_L^{1/2}+\gamma_{q+1}^{1/2},
\end{align*}
and for $m=1,\dots, N$ using $K\leq \ell^{-1}$
\begin{align*}
\|\rho^{1/2}\|_{C^m_{[4\sigma_q\wedge T_L,t],x}}\lesssim \ell^{1-7m}\delta_{q+1}^{1/2}M_L^{1/2}+\ell^{1/2-m}\gamma_{q+1}^{m}\leq \ell^{1-7m}(\delta_{q+1}^{1/2}M_L^{1/2}+\gamma_{q+1}^{1/2}).
\end{align*}
This implies  for $N\in \mN_{0}$ as in \eqref{estimate aN}
\begin{equation}\label{estimate aN ps}
\begin{aligned}
\|a_{(\xi)}\|_{C^N_{[4\sigma_q\wedge T_L,t],x}}
&\lesssim \ell^{-8-7N}(\delta_{q+1}^{1/2}M_L^{1/2}+\gamma_{q+1}^{1/2}).
\end{aligned}
\end{equation}
For a general $t\in [0,T_L]$ we have for $N\in \mN_{0}$
\begin{equation}\label{estimate aN ps1}
\begin{aligned}
\|a_{(\xi)}\|_{C^N_{t,x}}
&\lesssim \ell^{-8-7N}((M_L+qA)^{1/2}+\gamma_{q+1}^{1/2}),
\end{aligned}
\end{equation}
where we used $M_L+qA\leq \ell^{-1}$ and $K\leq \ell^{-1}$.

 Let us introduce a smooth cut-off function
\begin{align*}
\chi(t)=\begin{cases}
0,& t\leq \frac{\sigma_q}2,\\
\in (0,1),& t\in (\frac{\sigma_q}2,{\sigma_q} ),\\
1,&t\geq {\sigma_q}.
\end{cases}
\end{align*}
Note that
$\|\chi'\|_{C^{0}_{t}}\leq 2^{q+1}$ which has to be taken into account in the estimates of the $C^{1}_{t,x}$-norms in \eqref{principle est2 ps}, \eqref{correction est2 ps}, \eqref{temporal est2 ps} below.

Let  $ w^{(p)}_{q+1},$ $  w_{q+1}^{(c)},$  $w_{q+1}^{(t)}$ be given as in Section \ref{s:con} with $a_{(\xi)}$ replaced by the above.
We define the perturbations $\tilde w^{(p)}_{q+1},$ $ \tilde w_{q+1}^{(c)},$ $ \tilde w_{q+1}^{(t)}$ as follows
$$\tilde w^{(p)}_{q+1}:=w^{(p)}_{q+1}\chi,\quad \tilde w^{(c)}_{q+1}:=w^{(c)}_{q+1}\chi,\quad \tilde w^{(t)}_{q+1}:=w^{(t)}_{q+1}\chi^2.$$
Using Lemma \ref{lem:Lp} similarly to \eqref{estimate wqp} we obtain
 for $t\in(4\sigma_q\wedge T_L, T_L]$ and some universal constant $M_0\geq 1$
\begin{equation}\label{estimate wqp ps}
\|\tilde w_{q+1}^{(p)}(t)\|_{L^2}\lesssim \sum_{\xi\in\Lambda}\frac{1}{4|\Lambda|}M(M_L^{1/2}\delta_{q+1}^{1/2}+\gamma_{q+1}^{1/2})\|W_{(\xi)}\|_{C_tL^2}\leq \frac{M_0}{2}(M_L^{1/2}\delta_{q+1}^{1/2}+\gamma_{q+1}^{1/2}),
\end{equation}
where we used $150\alpha<\frac17$,
and for $t\in (\frac{\sigma_q}2\wedge T_L,4\sigma_q\wedge T_L]$
\begin{equation}\label{estimate wqp ps1}
\|\tilde w_{q+1}^{(p)}(t)\|_{L^2}\leq \frac{M_0}{2}((M_L+qA)^{1/2}+\gamma_{q+1}^{1/2}).
\end{equation}

For general $L^p$-norms we apply (\ref{bounds}) and (\ref{estimate aN ps}) to deduce for $t\in(4\sigma_q\wedge T_L, T_L]$, $p\in(1,\infty)$
\begin{equation}\label{principle est1 ps}
\aligned
\|\tilde w_{q+1}^{(p)}(t)\|_{L^p}&\lesssim \ell^{-8}(M_L^{1/2}\delta_{q+1}^{1/2}+\gamma_{q+1}^{1/2})r_\perp^{2/p-1}r_\|^{1/p-1/2},
\endaligned
\end{equation}
\begin{equation}\label{correction est ps}
\aligned
\|\tilde w_{q+1}^{(c)}(t)\|_{L^p}&\lesssim \ell^{-22}(M_L^{1/2}\delta_{q+1}^{1/2}+\gamma_{q+1}^{1/2})r_\perp^{2/p}r_\|^{1/p-3/2},
\endaligned
\end{equation}
and
\begin{equation}\label{temporal est1 ps}
\aligned
\|\tilde w_{q+1}^{(t)}(t)\|_{L^p}&\lesssim  \ell^{-16}(M_L\delta_{q+1}+\gamma_{q+1})r_\perp^{2/p-1}r_\|^{1/p-2}\lambda_{q+1}^{-1},
\endaligned
\end{equation}
\begin{equation}\label{corr temporal ps}
\aligned
&\|\tilde w_{q+1}^{(c)}(t)\|_{L^p}+\|\tilde w_{q+1}^{(t)}(t)\|_{L^p}\lesssim\ell^{-2} (M_L^{1/2}\delta_{q+1}^{1/2}+\gamma_{q+1}^{1/2})r_\perp^{2/p-1}r_\|^{1/p-1/2}.
\endaligned
\end{equation}
For a general $t\in (\frac{\sigma_q}2\wedge T_L,4\sigma_q\wedge T_L]$ we have
\begin{equation}\label{principle est1 ps1}
\aligned
\|\tilde w_{q+1}^{(p)}(t)\|_{L^p}&\lesssim \ell^{-8}((M_L+qA)^{1/2}+\gamma_{q+1}^{1/2})r_\perp^{2/p-1}r_\|^{1/p-1/2},
\endaligned
\end{equation}
\begin{equation}\label{correction est ps1}
\aligned
\|\tilde w_{q+1}^{(c)}(t)\|_{L^p}&\lesssim \ell^{-22}((M_L+qA)^{1/2}+\gamma_{q+1}^{1/2})r_\perp^{2/p}r_\|^{1/p-3/2},
\endaligned
\end{equation}
and
\begin{equation}\label{temporal est1 ps1}
\aligned
\|\tilde w_{q+1}^{(t)}(t)\|_{L^p}&\lesssim  \ell^{-16}((M_L+qA)+\gamma_{q+1})r_\perp^{2/p-1}r_\|^{1/p-2}\lambda_{q+1}^{-1},
\endaligned
\end{equation}
\begin{equation}\label{corr temporal ps1}
\aligned
&\|\tilde w_{q+1}^{(c)}(t)\|_{L^p}+\|\tilde w_{q+1}^{(t)}(t)\|_{L^p}\lesssim \ell^{-2}((M_L+qA)^{1/2}+\gamma_{q+1}^{1/2})r_\perp^{2/p-1}r_\|^{1/p-1/2},
\endaligned
\end{equation}
where we used $M_L+qA<\ell^{-1}$ and $K<\ell^{-1}$.

Define ${w}_{q+1}:=\tilde{w}_{q+1}^{(p)}+\tilde{w}_{q+1}^{(p)}+\tilde{w}_{q+1}^{(t)}$. Combining (\ref{estimate wqp ps}), (\ref{correction est ps}) and (\ref{temporal est1 ps}) we obtain for $t\in(4\sigma_q\wedge T_L, T_L]$
\begin{equation}\label{estimate wq ps}
\aligned
\|w_{q+1}(t)\|_{L^2}&\leq (M_L^{1/2}\delta_{q+1}^{1/2}+\gamma_{q+1}^{1/2})\left(\frac{M_0}{2}+C\lambda_{q+1}^{44\alpha-2/7}+C(M_L^{1/2}\delta_{q+1}^{1/2}+\gamma_{q+1}^{1/2})\lambda_{q+1}^{32\alpha-1/7}\right)
\\&\leq \frac34M_{0}(M_L^{1/2}\delta_{q+1}^{1/2}+\gamma_{q+1}^{1/2}),
\endaligned
\end{equation}
and
for $t\in(\frac{\sigma_q}2\wedge T_L, 4\sigma_q\wedge T_L]$
\begin{equation}\label{estimate wq ps1}
\aligned
\|w_{q+1}(t)\|_{L^2}
&\leq \frac34M_{0}((M_L+qA)^{1/2}+\gamma_{q+1}^{1/2}).
\endaligned
\end{equation}

With these bounds, we have all in hand to complete the proof of Proposition~\ref{p:iteration}. We split the details it into several subsections.

\subsubsection{Proof of \eqref{iteration ps}}
First, \eqref{estimate wq ps} together with \eqref{error ps} yields for $t\in(4\sigma_q\wedge T_L, T_L]$
$$
\|v_{q+1}(t)-v_{q}(t)\|_{L^{2}}\leq \|w_{q+1}(t)\|_{L^{2}}+\|v_{\ell}(t)-v_{q}(t)\|_{L^{2}}\leq M_0(M_L^{1/2} \delta_{q+1}^{1/2}+\gamma_{q+1}^{1/2}).
$$
For $t\in (\frac12\sigma_q\wedge T_L, 4\sigma_q\wedge T_L]$
we use \eqref{estimate wq ps1}, \eqref{error ps} to obtain
$$
\|v_{q+1}(t)-v_{q}(t)\|_{L^{2}}\leq \|w_{q+1}(t)\|_{L^{2}}+\|v_{\ell}(t)-v_{q}(t)\|_{L^{2}}\leq M_0((M_L+qA)^{1/2}+\gamma_{q+1}^{1/2}).
$$
For $t\in  [0,\frac12\sigma_q\wedge T_L]$ it holds $\chi(t)=0$ as well as $v_{q}(t)=0$ by \eqref{inductionv ps} implying
$$
\|v_{q+1}-v_{q}\|_{C_{t}L^{2}}=\|v_{\ell}-v_{q}\|_{C_{t}L^{2}}=0.%
$$
Hence \eqref{iteration ps} follows.

\subsubsection{Proof that \eqref{iteration ps} implies \eqref{inductionv ps} on the level $q+1$}
From \eqref{iteration ps} we find for $t\in [0,\frac12\sigma_q\wedge T_L]$ by \eqref{iteration ps} that
$$
\|v_{q+1}\|_{C_{t}L^{2}}\leq \sum_{r=0}^{q}\|v_{r+1}-v_{r}\|_{C_{t}L^{2}}=0,%
$$
proving the second bound in \eqref{inductionv ps} on the level $q+1$.
For the first bound in \eqref{inductionv ps} on the level $q+1$, we obtain in view of \eqref{iteration ps} for $t\in (\frac12\sigma_q\wedge T_L,T_{L}]$
$$
\begin{aligned}
\|v_{q+1}(t)\|_{L^{2}}&\leq \sum_{0\leq r\leq q}\|v_{r+1}(t)-v_{r}(t)\|_{L^{2}}\\
&\leq M_{0} \left(M_{L}^{1/2}\sum_{0\leq r\leq q}\delta_{r+1}^{1/2}+\sum_{0\leq r\leq q}(M_L+rA)^{1/2}\mathbf{1}_{t\in (\frac{\sigma_{r}}{2}\wedge T_{L},4\sigma_{r}\wedge T_{L}]}+\sum_{0\leq r\leq q}\gamma_{r+1}^{1/2}\right)\\&\leq M_{0} \left(M_{L}^{1/2}\sum_{0\leq r\leq q}\delta_{r+1}^{1/2}+3(M_L+qA)^{1/2}+\sum_{0\leq r\leq q}\gamma_{r+1}^{1/2}\right),
\end{aligned}
$$
where we used the fact that by the definition of $\sigma_{r}=2^{-r}$ each $t\in [0,T_L]$ only belongs to three intervals $(\frac{\sigma_{r}}{2}\wedge T_{L},4\sigma_{r}\wedge T_{L}]$.  Hence \eqref{inductionv ps} follows.

\subsubsection{Proof of \eqref{inductionv C1} on the level $q+1$}
Using \eqref{estimate aN ps1} we find for $t\in [0,T_L]$ as in \eqref{principle est2}, \eqref{correction est2}, \eqref{temporal est2}
\begin{equation}\label{principle est2 ps}
\aligned
\|\tilde w_{q+1}^{(p)}\|_{C^1_{t,x}}&\lesssim \ell^{-15}((M_L+qA)^{1/2} +\gamma_{q+1}^{1/2})r_\perp^{-1}r_\|^{-1/2}\lambda_{q+1}^2,
\endaligned
\end{equation}
\begin{equation}\label{correction est2 ps}
\aligned
\|\tilde w_{q+1}^{(c)}\|_{C^1_{t,x}}
&\lesssim \ell^{-29}((M_L+qA)^{1/2}+\gamma_{q+1}^{1/2})r_\|^{-3/2}\lambda_{q+1}^2,
\endaligned
\end{equation}
and
\begin{equation}\label{temporal est2 ps}
\aligned
\|\tilde w_{q+1}^{(t)}\|_{C^1_{t,x}}
\lesssim \ell^{-24}(M_L+qA+\gamma_{q+1})r_\perp^{-1}r_\|^{-2}\lambda_{q+1}^{1+\alpha}.
\endaligned
\end{equation}
In particular, we see that the fact that the time derivative of $\chi$  behaves like $2\sigma_{q}^{-1}\lesssim \ell^{-1}$ does not pose any problems as the $C^{0}_{t,x}$-norms of $\tilde{w}_{q+1}^{(p)}$, $\tilde{w}_{q+1}^{(c)}$ and $\tilde{w}_{q+1}^{(t)}$ always contain smaller powers of $\ell^{-1}$.

Combining   \eqref{eq:vl2 ps} and \eqref{principle est2 ps}, \eqref{correction est2 ps}, \eqref{temporal est2 ps} with (\ref{ell}) we obtain for $t\in[0, T_L]$
\begin{equation*}
\aligned
\|v_{q+1}\|_{C^1_{t,x}}&\leq \|v_\ell\|_{C^1_{t,x}}+\|w_{q+1}\|_{C^1_{t,x}}\\
&\leq  ({M}_L+qA+\gamma_{q+1})^{1/2}\left(\lambda_{q+1}^\alpha+C\lambda_{q+1}^{30\alpha+22/7}+C\lambda_{q+1}^{58\alpha+20/7}+C\lambda_{q+1}^{50\alpha+3}\right)
\leq M_L^{1/2}\lambda_{q+1}^4,
\endaligned
\end{equation*}
where we used ${M}_L+qA+\gamma_{q+1}\leq \ell^{-1}$.
This implies  \eqref{inductionv C1}.

\subsubsection{Proof of \eqref{induction w}}
Similarly, we derive  the following estimates: for $t\in[0, T_L]$ it follows from (\ref{ell}), (\ref{estimate aN ps1})  that
\begin{equation}\label{principle est22 ps}
\aligned
\|\tilde w_{q+1}^{(p)}+\tilde w_{q+1}^{(c)}\|_{C_tW^{1,p}}
&\lesssim  \ell^{-8}((M_L+qA)^{1/2}+\gamma_{q+1}^{1/2})r_\perp^{2/p-1}r_\|^{1/p-1/2}\lambda_{q+1},
\endaligned
\end{equation}
and
\begin{equation}\label{corrector est2 ps}
\aligned
\|\tilde w_{q+1}^{(t)}\|_{C_tW^{1,p}}
&\lesssim \ell^{-16}(M_L+qA+\gamma_{q+1})r_\perp^{2/p-2}r_\|^{1/p-1}\lambda_{q+1}^{-2/7}.
\endaligned
\end{equation}
For $p=\frac{31}{30}<\frac{16}{15}$ we find for $t\in [0,T_L]$
\begin{align*}
\|w_{q+1}\|_{C_tW^{1,p}}\lesssim \ell^{-16}(M_L+qA+\gamma_{q+1})\lambda_{q+1}^{-\frac{15}{217}}
\leq \frac34 M_L^{1/2}\delta_{q+1}^{1/2},
\end{align*}
where we used the condition for $\alpha, \beta$ in the second step, which combined with \eqref{error ps1} implies
\eqref{induction w}.

\subsubsection{Proof of  \eqref{p:gamma}}

We control the energy similarly to Section \ref{s:energy}.
By definition, we find
	\begin{align}\label{eq:deltaE ps}
	\begin{aligned}
	&\big|\|v_{q+1}\|_{L^2}^2-\|v_q\|_{L^2}^2-3\gamma_{q+1}\big|
	\leq \big|\|\tilde w_{q+1}^{(p)}\|_{L^2}^2-3\gamma_{q+1}\big|+\|\tilde w_{q+1}^{(c)}+\tilde w_{q+1}^{(t)}\|_{L^2}^2\\
	&\ +2\|v_\ell(\tilde w_{q+1}^{(c)}+\tilde w_{q+1}^{(t)})\|_{L^1}
+2\|v_\ell \tilde w_{q+1}^{(p)}\|_{L^1}+2\|\tilde w_{q+1}^{(p)}(w_{q+1}^{(c)}+\tilde w_{q+1}^{(t)})\|_{L^1}+|\|v_\ell\|_{L^2}^2-\|v_q\|_{L^2}^2|.
	\end{aligned}
	\end{align}
	Let us begin with the bound of the first term on the right hand side of \eqref{eq:deltaE ps}. We use \eqref{can} and the fact that $\mathring{R}_\ell$ is traceless to deduce for $t\in (4\sigma_q\wedge T_L,T_L]$
	\begin{align*}
	|\tilde w_{q+1}^{(p)}|^2-\frac{3\gamma_{q+1}}{(2\pi)^3}=6\sqrt{\ell^2+|\mathring{R}_\ell|^2}+\sum_{\xi\in \Lambda}a_{(\xi)}^2P_{\neq0}|W_{(\xi)}|^2
	,
	\end{align*}
	hence
	\begin{align}\label{eq:gg ps}
	|\|\tilde w_{q+1}^{(p)}\|_{L^2}^2-3\gamma_{q+1}|\leq 6\cdot(2\pi)^3\ell+6\|\mathring{R}_\ell\|_{L^1}+\sum_{\xi\in \Lambda}\Big|\int a_{(\xi)}^2P_{\neq0}|W_{(\xi)}|^2\Big|.
	\end{align}
	Here we estimate each term separately. Using \eqref{ell1}  we find
	\begin{align*}
	6\cdot(2\pi)^3\ell\leq 6\cdot(2\pi)^3\lambda_{q+1}^{-{3\alpha}/2}\leq  \frac{1}{ 48}\lambda_{q+1}^{-2\beta }M_L\leq\frac{1}{ 48}\delta_{q+1}M_L,
	\end{align*}
	which requires $2\beta <{3\alpha}/2$ and choosing $a$ large to absorb the constant.
	Using \eqref{eq:R} on $\mathring{R}_q$ and $\supp\varphi_\ell\subset [0,\ell]$ we know for $t\in (4\sigma_q\wedge T_L,T_L]$
	\begin{align*}
	6\|\mathring{R}_\ell(t)\|_{L^1}\leq  6\delta_{q+1}M_L.
	\end{align*}
	For the last term in \eqref{eq:gg ps} we use similar argument as in Section \ref{s:energy} to have since $M_{L}\geq 1$
		\begin{align*}
	\sum_{\xi\in\Lambda}&\Big|\int a_{(\xi)}^2\mathbb{P}_{\neq0}|W_{(\xi)}|^2\Big|\lesssim\lambda_{q+1}^{158\alpha-1/7}(M_L+K)\lesssim\lambda_{q+1}^{160\alpha-1/7}
	\leq\frac1{24}\lambda_1^{2\beta}\lambda_{q+1}^{-2\beta} M_L=\frac1{24}\delta_{q+1}M_L,
	\end{align*}
	where we used $M_{L}+K\leq\ell^{-1}\leq\lambda_{q+1}^{2\alpha}$ as well as $160\alpha+2\beta<1/7$.
This completes the bound for \eqref{eq:gg ps}.

	Going back to \eqref{eq:deltaE ps}, we control the  remaining terms as follows.
Using  the estimates \eqref{correction est ps}, \eqref{temporal est1 ps} and \eqref{ell} we  have for $t\in (4\sigma_q\wedge T_L,T_L]$
	\begin{align*}
	\|\tilde w_{q+1}^{(c)}+\tilde w_{q+1}^{(t)}\|_{L^2}^2&
	\lesssim (M_L+\gamma_{q+1})\lambda_{q+1}^{88\alpha-4/7}
	+(M_L^2+\gamma_{q+1}^2)\lambda_{q+1}^{64\alpha-2/7}\leq \frac{1}{48}
	\lambda_{q+1}^{-2\beta }M_L\leq \frac{\delta_{q+1}}{48}M_L,
	\end{align*}
	where we use  $M_L+\gamma_{q+1}\leq \ell^{-1}$ to absorb $M_L$ and $\gamma_{q+1}$. Similarly we use \eqref{eq:vl ps} together with \eqref{estimate wqp ps} to have for $t\in (4\sigma_q\wedge T_L,T_L]$
	\begin{align*}
	&2\|v_\ell(\tilde w_{q+1}^{(c)}+\tilde w_{q+1}^{(t)})\|_{L^1}+2\|\tilde w_{q+1}^{(p)}(\tilde w_{q+1}^{(c)}+\tilde w_{q+1}^{(t)})\|_{L^1}
	\lesssim M_{0}((M_L+qA)^{1/2}+K^{1/2})\|\tilde w_{q+1}^{(c)}+\tilde w_{q+1}^{(t)}\|_{L^2}\\
	&\qquad \lesssim M_{0}((M_L+qA)^{1/2}+K^{1/2})\left((M_L^{1/2}+\gamma_{q+1}^{1/2})\lambda_{q+1}^{44\alpha-2/7}
	+(M_L+\gamma_{q+1})\lambda_{q+1}^{32\alpha-1/7}\right)
	\\
	&\qquad\leq \frac{1}{48}
	\lambda_{q+1}^{-2\beta }M_L\leq \frac{\delta_{q+1}}{48}M_L,
	\end{align*}
where we used $M_L+qA+K\leq \ell^{-1}$ and we possibly increased $a$ to absorb $M_{0}$.
	We use \eqref{ell} and \eqref{principle est1 ps} and $\|v_\ell\|_{C^1_{t,x}}\leq \|v_q\|_{C^1_{t,x}} $ to have for every $\eps>0$
	\begin{align*}
	2\|v_\ell \tilde w_{q+1}^{(p)}\|_{L^1}
	&\lesssim\|v_\ell\|_{L^\infty}\|\tilde w_{q+1}^{(p)}\|_{L^1}\lesssim M_L^{1/2}\lambda_q^4\ell^{-8}(M_L^{1/2}\delta_{q+1}^{1/2}+\gamma_{q+1}^{1/2})r_{\perp}^{1-\eps}r_{\|}^{\frac12(1-\eps)}\\
	&\lesssim (M_L+K)\lambda_{q+1}^{17\alpha-\frac87(1-\eps)}\leq \frac{1}{96}
	\lambda_{q+1}^{-2\beta }M_L\leq \frac{\delta_{q+1}}{96}M_L.
	\end{align*}
	For the last  terms by \eqref{error ps} we have
	\begin{align*}
	|\|v_\ell\|_{L^2}^2-\|v_q\|_{L^2}^2|&\leq \|v_\ell-v_q\|_{L^2}(\|v_\ell\|_{L^2}+\|v_q\|_{L^2})\\
	&\lesssim \ell\lambda_q^4 M_L^{1/2} M_{0}(M_L+qA+K)^{1/2}\leq \lambda_{q+1}^{-\alpha}M_{0}(M_L+qA+K)\\
	&\leq \frac{1}{96}
	\lambda_{q+1}^{-2\beta }M_L\leq \frac{\delta_{q+1}}{96}M_L .
	\end{align*}
	which requires
	$M_{L}+qA+K\leq \lambda_{q+1}^{\alpha-2\beta}$ and $a$ large enough to absorb the extra constant.
	
	Combining the above estimates \eqref{p:gamma} follows.

\subsubsection{Proof of \eqref{iteration R}}\label{s:R}

The new stress $\mathring{R}_{q+1}$ is defined in the spirit of Section~\ref{s:def}, notably, as in \eqref{stress} by replacing the three correctors $w_{q+1}^{(p)}$, $w_{q+1}^{(c)}$ and $w_{q+1}^{(t)}$ by their respective truncated counterparts $\tilde w_{q+1}^{(p)}$, $\tilde w_{q+1}^{(c)}$ and $\tilde w_{q+1}^{(t)}$.
In order to establish  the iterative estimate \eqref{iteration R}, we distinguish  three cases corresponding to the three time intervals in \eqref{iteration R}.

\textbf{I.} Let $t\in (\sigma_q\wedge T_L,T_L]$. Note that if $T_{L}\leq \sigma_{q}$ then there is nothing to estimate here, hence we assume that $\sigma_{q}<T_{L}$ and $t\in (\sigma_{q},T_{L}]$.
Then we have $\chi(t)=1$ and therefore  we  use similar calculations  as in Section~\ref{sss:R} to derive the first inequality in \eqref{iteration R} by changing $\bar e$ into $M_L$. More precisely, we notice that  the bounds in \eqref{principle est1 ps1}, \eqref{correction est ps1}, \eqref{temporal est1 ps1} hold for $t\in [0,T_L]$ and the main difference between these bounds and the corresponding bounds \eqref{principle est1}, \eqref{correction est}, \eqref{temporal est1} in Section \ref{sss:v} is the change from $\bar{e}^{1/2}\delta_{q+1}^{1/2}$ to $(M_L+qA)^{1/2}+\gamma_{q+1}^{1/2}$, which can be bounded by $\ell^{-1/2}$. In the estimate of $\mathring{R}_{q+1}$ in Section~\ref{sss:R} we always omit $\delta_{q+1}^{1/2}$. As a result, we could  estimate similarly for most of the terms.
 The strongest requirement comes from the bound of $R_{\rm{cor}}$, namely, we have
 $$
 \|R_{\rm cor}(t)\|_{L^{p}}\lesssim((M_L+qA)^{1/2}+\gamma_{q+1}^{1/2})^3\lambda_{q+1}^{49\alpha-1/7}\leq \ell^{-3/2}\lambda_{q+1}^{49\alpha-1/7}
\leq\frac{M_{L}}{5}\lambda_{q+1}^{-2\beta b}\leq \frac{M_{L}}{5}\delta_{q+2},
 $$
which is satisfied provided $ \lambda_{q+1}^{52\alpha-\frac17}\leq\frac{M_{L}}{5}\lambda_{q+1}^{-2\beta b}.$
Another difference comes from the estimate involving the process $z$ since its initial value is not smooth. We estimate the corresponding terms in what follows.

The first part comes from  $z_\ell\mathring\otimes w_{q+1}$ and $w_{q+1}\mathring\otimes z_\ell$ in  $R_{\mathrm{lin}}$.
According to Lemma 9 in  \cite{DV15}
we know
$$
\|z^{{in}}(t)\|_{L^{\infty}}= \|e^{t\Delta}u_{0}\|_{L^{\infty}}\lesssim (t^{-3/4}+1)\|u_{0}\|_{L^{2}},
$$
whereas $\|Z_q(t)\|_{L^{\infty}}\leq L\lambda_{q+1}^{\alpha/8}$ is controlled uniformly in time due to the stopping time.
Hence we use \eqref{principle est1 ps1} and \eqref{corr temporal ps1} to have for $p=\frac{32}{32-7\alpha}$
\begin{align*}
\|z_\ell\mathring\otimes w_{q+1}(t)\|_{L^1}&\leq \|z_\ell(t)\|_{L^{\infty}}\|w_{q+1}(t)\|_{L^p}\lesssim \big(\sup_{s\in[t-\ell,t]}s^{-3/4}+\lambda_{q+1}^{\alpha/8}\big)M_L^{1/2}\|w_{q+1}(t)\|_{L^p}
\\
&\lesssim (\ell^{-1}+{\lambda_{q+1}^{\alpha/8}})M_{L}^{1/2}((M_L+qA)^{1/2}+\gamma^{1/2}_{q+1})\ell^{-8}r_{\perp}^{2/p-1}r_{\|}^{1/p-1/2}\\
&\lesssim \lambda_{q+1}^{21\alpha-8/7}\leq \frac{M_{L}}{20}\lambda_{q+1}^{-2\beta b}\leq \frac{M_{L}}{20}\delta_{q+2},
\end{align*}
where we used $t> \sigma_q\geq 2\ell$ which is satisfied by our assumptions at the beginning of Section~\ref{s:par}.

 The second part comes from $R_{\mathrm{com}}$ and $R_{\mathrm{com1}}$. In view of the bound for $\kappa\in (0,1)$ and $t\in (\sigma_q,T_L]$
\begin{align*}
&\|f_\ell\otimes g_\ell-(f\otimes g)*_x\phi_\ell*_t\varphi_\ell\|_{L^1}\lesssim \|f_\ell\otimes g_\ell-f\otimes g\|_{L^1}+\|f\otimes g-(f\otimes g)*_x\phi_\ell*_t\varphi_\ell\|_{L^1}\\
&\qquad\lesssim \|f_\ell-f\|_{L^2}\|g_\ell\|_{L^2}+\|f\|_{L^2}\|g_\ell-g\|_{L^2}+\|f\otimes g-(f\otimes g)*_x\phi_\ell*_t\varphi_\ell\|_{L^1}
\\
&\qquad\lesssim \ell^{\kappa}(\|f\|_{C_{[\frac{\sigma_q}2,t]}^\kappa L^2}+\|f\|_{C_{[\frac{\sigma_q}2,t]} C^\kappa})\|g\|_{L_t^\infty L^2}+\ell^{\kappa}(\|g\|_{C_t^\kappa L^2}+\|g\|_{C_t C^\kappa})\|f\|_{L_t^\infty L^2},
	\end{align*}
we estimate
for $\delta<\frac1{12}$ and $t\in (\sigma_q,T_L]$
\begin{align*}
&\|R_{\mathrm{com}}(t)\|_{L^1}\leq (\ell \|v_q\|_{C^1_{t,x}}+\ell^{1/2-2\delta}(\|{Z_q}\|_{C_t^{1/2-2\delta}L^\infty}+\|{Z_q}\|_{C_tC_x^{1/2-2\delta}}))(\|v_q\|_{L^\infty_tL^2}+\|{z_q}\|_{C_tL^2})\\
&\qquad\qquad\qquad+\ell^{1/2}(\|z^{in}\|_{C_{[\frac{\sigma_q}2,t]}^{1/2}L^2}+\|z^{in}\|_{C_{[\frac{\sigma_q}2,t]}{H}^{1/2}})(\|v_q\|_{L^\infty_tL^2}+\|{z_q}\|_{C_tL^2}).
\end{align*}	
Due to  Lemma 9 in  \cite{DV15} it holds
$$
\|z^{in}(t)\|_{H^{1/2}}\lesssim (t^{-1/4}+1)\|u_0\|_{L^2}=(t^{-1/4}+1)\|u_0\|_{L^2},
$$
for $|t-s|\leq 1$
\begin{align*}
\|z^{in}(t)-z^{in}(s)\|_{L^{2}}	&=\|e^{s\Delta}(e^{(t-s)\Delta}u_0-u_0)\|_{L^2}
	=\Big\|e^{s\Delta}\int_0^{t-s}\Delta e^{r\Delta}u_0\dif r\Big\|_{L^2}
	\\&	\lesssim\int_0^{t-s}(1+r^{-1/2})\| e^{s\Delta}u_0\|_{H^{1}}\dif r\lesssim |t-s|^{1/2}(1+s^{-1/2})\|u_0\|_{L^2},
\end{align*}
and for $|t-s|\geq1$
\begin{align*}
\|z^{in}(t)-z^{in}(s)\|_{L^{2}}\leq 2\|u_0\|_{L^2}	\leq 2 |t-s|^{1/2}\|u_0\|_{L^2}.
\end{align*}
Hence, in view of our assumption $\sigma_{q}\geq \ell^{1/2}$ we deduce for $t\geq \sigma_q$
$$
\begin{aligned}
&\|R_{\mathrm{com}}(t)\|_{L^1}\\
&\qquad\lesssim M_0(\ell\lambda_q^4+\ell^{1/2-2\delta}\lambda_{q+1}^{{\alpha/4}})(M_L+qA+K)^{1/2}M_L^{1/2}+\ell^{1/2}(\sigma_q^{-1/2}+1)M_0(M_L+qA+K)^{1/2}M_L^{1/2}
\end{aligned}
$$
and we require
$$
M_0\ell \lambda_q^4(M_L+K+qA)^{1/2}M_L^{1/2}\leq M_0\lambda_{q+1}^{-\alpha}(M_L+K+qA)^{1/2}M_L^{1/2}< \frac{M_L}{30}\lambda_{q+1}^{-2\beta b}\leq\frac{M_L }{30}\delta_{q+2},
$$
$$M_0\ell^{1/2-2\delta}{\lambda_{q+1}^{{\alpha/4}}}(M_L+K+qA)^{1/2}M_L^{1/2}\leq M_0 \ell^{1/3}{\lambda_{q+1}^{{\alpha/4}}}(M_L+K+qA)^{1/2}M_L^{1/2}< \frac{M_L}{30}\lambda_{q+1}^{-2\beta b}
\leq\frac{M_L}{30}\delta_{q+2},
$$
$$
 \ell^{1/2}\sigma_q^{-1/2}M_0(M_L+K+qA)^{1/2}M_L^{1/2}\leq\ell^{1/4}M_0(M_L+K+qA)^{1/2}M_L^{1/2}< \frac{M_L}{30}\lambda_{q+1}^{-2\beta b}
\leq\frac{M_L}{30}\delta_{q+2},
$$
where we used $\alpha>18\beta b$ and $\ell^{1/4}M_0(M_L+K+qA)<\frac{M_{L}}{30}\lambda_{q+1}^{-2\beta b}$ and $\ell^{1/3}M_0(M_L+K+qA)^{1/2}M_L^{1/2}<\frac{M_{L}}{30}\lambda_{q+1}^{-2\beta b-\frac\alpha4}$.

For $R_{\mathrm{com}1}$ we know
{\begin{align*}
\|R_{\mathrm{com}1}(t)\|_{L^1}\lesssim&  M_0(M_L+K+qA)^{1/2}\|z_\ell(t)-z_{q+1}(t)\|_{L^2}
\\\lesssim&  M_0(M_L+K+qA)^{1/2}(\|z_\ell(t)-z_{q}(t)\|_{L^2}+\|z_q(t)-z_{q+1}(t)\|_{L^2})
\\\lesssim&(\ell^{1/2-2\delta}+\ell^{1/2}\sigma_q^{-1/2}+\lambda_{q+1}^{-\frac\alpha8(1-\delta)})M_0(M_L+K+qA),
\end{align*}}
which can be bounded by $\frac{\delta_{q+2}M_L}{10}$ by the same choice of the parameters and $\alpha >18b\beta$ to have $\lambda_{q+1}^{-\frac\alpha8(1-\delta)}M_0(M_L+K+qA)<\frac{\delta_{q+2}M_L}{30}$.

\textbf{II.} Let  $t\in (\frac{\sigma_q}2\wedge T_L, \sigma_q\wedge T_L]$. If $T_{L}\leq \frac{\sigma_{q}}2$ then there is nothing to estimate, hence we may assume $\frac{\sigma_{q}}2<T_{L}$ and  $t\in (\frac{\sigma_q}2, \sigma_q\wedge T_L]$. Then
we decompose $\mathring{R}_\ell=\chi^2\mathring{R}_\ell+(1-\chi^2)\mathring{R}_\ell$. The first part $\chi^2\mathring{R}_\ell$ is canceled (up to the oscillation error $\chi^{2}R_{\rm osc}$) by
$\tilde w_{q+1}^{(p)}\otimes \tilde w_{q+1}^{(p)}=\chi^{2} w_{q+1}^{(p)}\otimes  w_{q+1}^{(p)}$
 and $\chi^2\p_tw_{q+1}^{(t)}=\p_t\tilde w_{q+1}^{(t)}-(\chi^2)'w_{q+1}^{(t)}$. So in this case in the definition of $\mathring{R}_{q+1}$ most terms are similar to in the case \textbf{I.}  and can be estimated similarly as in Section \ref{sss:R}.
 We only have to consider $(1-\chi^2)\mathring{R}_\ell$, and
 $$\div R_{\mathrm{cut}}:=\chi'(t)(w_{q+1}^{(p)}(t)+w_{q+1}^{(c)}(t))+(\chi^2)'(t)w_{q+1}^{(t)}(t).$$
We know
$$\|(1-\chi^2)\mathring{R}_\ell(t)\|_{L^1}\leq\sup_{s\in [t-\ell,t]} \|\mathring{R}_q(s)\|_{L^1}.$$
For $R_{\mathrm{cut}}$ we realize that the bounds \eqref{principle est1 ps1} and \eqref{corr temporal ps1} also hold for $w_{q+1}^{(p)}$, $w_{q+1}^{(c)}$ and $w_{q+1}^{(t)}$. Then we have for $\eps>0$
\begin{align*}
 \|R_{\mathrm{cut}}(t)\|_{L^1}&\leq\|\chi'(t)(w_{q+1}^{(p)}(t)+w_{q+1}^{(c)}(t))\|_{L^1}+\|(\chi^2)'(t)w_{q+1}^{(t)}(t)\|_{L^1}\\
 &\lesssim\frac1{\sigma_q}\|w_{q+1}^{(p)}(t)\|_{L^1}+\frac1{\sigma_q}\|w_{q+1}^{(c)}(t)\|_{L^1}+\frac1{\sigma_q}\|w_{q+1}^{(t)}(t)\|_{L^1}
 \\
 &\leq\frac1{\sigma_q}\ell^{-8}((M_L+qA)^{1/2}+\gamma_{q+1}^{1/2})r_\perp^{1-2\eps}r_{\|}^{1/2-\eps} \leq \frac{M_L}{10}\delta_{q+2},
\end{align*}
where we use $\sigma_q^{-1}\leq \ell^{-1}$. For $R_{\mathrm{com}}$ and $R_{\mathrm{com1}}$ since $t\geq \sigma_q/2$ and $4\ell\leq \sigma_q$ we have a similar bound as in the first case.

\textbf{III.} For $t\in [0,\frac{\sigma_q}2\wedge T_L]$ we know $v_\ell(t)=v_{q+1}(t)=0$ and
$$\mathring{R}_{q+1}=\mathring{R}_\ell+R_{\mathrm{com1}}+R_{\mathrm{com}}.$$
We estimate $\|R_{\mathrm{com1}}(t)\|_{L^1}$, $\|R_{\mathrm{com}}(t)\|_{L^1}$ directly and have
\begin{align*}
\|R_{\mathrm{com1}}(t)\|_{L^1}+\|R_{\mathrm{com}}(t)\|_{L^1}\leq 4\|{z_q}\|_{C_tL^2}^2
\leq 4M_L.
\end{align*}

As a result, it follows
$$\|\mathring{R}_{q+1}(t)\|_{L^1}\leq \sup_{s\in [t-\ell,t]} \|\mathring{R}_q(s)\|_{L^1}
+4M_L,$$
which completes the proof of \eqref{iteration R}.

\subsubsection{Proof of \eqref{eq:R} and \eqref{bd:R} on the level $q+1$}
Finally, we observe that \eqref{eq:R} on the level $q+1$ is just \eqref{iteration R} and \eqref{bd:R} on the level $q+1$ follows from \eqref{iteration R} and \eqref{bd:R}.

With this,  the proof of Proposition~\ref{p:iteration} is complete.

 \appendix
  \renewcommand{\appendixname}{Appendix~\Alph{section}}
   \renewcommand{\theequation}{A.\arabic{equation}}
   \section{Proof of Theorem \ref{convergence}}
   \label{s:appA}

  \begin{proof}[Proof of Theorem \ref{convergence}] We first consider the second result giving the stability of martingale solutions with respect to the initial time and initial condition.  By (M3) with $p=1, 2$ and \cite[Corollary~B.3]{FR08} we have that for all $T>0$
 	$$\E^{P_n}\Big[\sup_{t\in [0,T]}\|x(t)\|_{L^2}^2+\int_{s_n}^T\|x(r)\|_{H^\gamma}^2\dif r\Big]\leq C.$$
 	By the same argument as in the proof of \cite[Theorem 3.1]{HZZ19} we therefore deduce that the set
 $\{P_n\}_{n\in\mathbb{N}}$ is tight in $\mathbb{S}:=C([0,\infty);H^{-3})\cap L^2_{\textrm{loc}}([0,\infty);L^2_\sigma)$.
 	
 	Without loss of generality, we  assume that $P_n$  converges weakly to some probability measure $P\in \mathscr{P}(\Omega_{0})$. It remains to prove that $P\in \mathscr{C}(s,x_0, C_{p})$.
 	By Skorokhod's representation theorem, there exists a probability space $(\tilde{\Omega},\tilde{\mathcal{F}},\tilde{P})$ and $\mathbb{S}$-valued
 	random variables $\tilde{x}_n$ and $\tilde{x}$ such that
 	\begin{enumerate}
 		\item[(i)] $\tilde{x}_n$ has the law $P_n$ for each $n\in\mathbb{N}$,
 		\item[(ii)] $\tilde{x}_n\rightarrow \tilde{x}$ in $\mathbb{S}$ $\tilde{P}$-a.s., and $\tilde{x}$ has the law $P$.
 	\end{enumerate}
 	(M1) follows from by the same argument as in the proof of \cite[Lemma A.3]{FR08} with $\|\cdot\|_V$ replaced by $\|\cdot\|_{H^\gamma}$.
 	(M2) follows by exactly the same argument as in the
 	the proof of \cite[Theorem 3.1]{HZZ19}. (M3) also follows from a similar argument as in the proof of \cite[Lemma A.3]{FR08}. The main difference is the change of initial time. It is sufficient to prove that for $t>0$ there is a set $T_t\subset (0,t)$ of zero Lebesgue measure such that for all $r\notin T_t$ and all bounded positive and real-valued $\mathcal{B}_r^0$-measurable continuous functions $g$ on $\mS$,
 	\begin{align}\label{sup}
 	\E^P[g(x) (E^1(t)-E^1(s))]\leq \E^P[g(x)(E^1(r)-E^1(s))].
 	\end{align}
 	For the case  $r>s$, \eqref{sup} follows as  in the proof of \cite[Lemma A.3]{FR08}. We only need to prove that
 	\begin{align}\label{sup1}
 	\E^P[g(x) (E^1(t)-E^1(s))]\leq 0.
 	\end{align}
 	For $P_n$ we know
 	\begin{align*}
 	\E^{P_n}[g(x_n) (E^1(t)-E^1(s_n))]\leq 0,
 	\end{align*}
 	i.e.
 	\begin{align}\label{sup3}
 	\E^{P_n}\Big[g(x_n) \Big(\|x(t)\|_{L^2}^2+2\int_{s_n}^t\|x(r)\|_{H^\gamma}^2\dif r-C_1C_{G}(t-s_{n})\Big)\Big]\leq g(x_n)\|x_n\|_{L^2}^2.
 	\end{align}
 	It is sufficient to prove that
 	\begin{align}\label{sup2} \int_{s}^t\|\tilde{x}(r)\|_{H^\gamma}^2\dif r\leq\liminf_{n\to\infty}\int_{s_n}^t\|\tilde{x}_n(r)\|_{H^\gamma}^2\dif r.
 	\end{align}
 	If \eqref{sup2} holds, \eqref{sup1} follows by taking limit on the both side of \eqref{sup3}.
 	
 	It is easy to see that $\tilde{x}_n1_{[s_n,t]}\to \tilde{x}1_{[s,t]}$ in $L^2(0,T;L^2_\sigma)$ $P$-a.s. Then by lower semicontinuity \eqref{sup2} follows.
 	
 	A similar argument also applies to $E^p$.
 	Consequently,  (M3) and the stability result follows.

 	The first result giving existence of a martingale solution can be easily deduced by Galerkin approximation and the same arguments as in \cite{FR08, GRZ09}. Since (M3) in  Definition \ref{martingale solution} is different, we first check that the law of  $u_n$ satisfies (M3) with $u_n$ satisfying
 	the following Galerkin approximation
 	\begin{equation*}
 	\aligned
 	\dif u_n-\Delta u_n \dif t+\Pi_n\mathbb{P}\div(u_n\otimes u_n)\dif t&=\Pi_n\dif B,
 	\\\div u_n&=0,
 	\\u_n(t)&=\Pi_nx_0,\quad 0\leq t\leq s,
 	\endaligned
 	\end{equation*}
 	where $\Pi_n$ and $\mathbb{P}$, respectively, is the Galerkin and Leray projection operator, respectively. It is classical to show that a solution exists on some probability space $(\Omega,\mathcal{F},\bf{P})$ with a $GG^{*}$-Wiener process $B$. For notational simplicity and without loss of generality, we may assume that the probability space and the Wiener process do not depend on $n$. We denote the law of $u_n$ on $\Omega_0$ by $P_n$. By using It\^o's formula we find that the process
 	\begin{align*}
 	M_{n,u_n}^{E}(t):=\|u_n(t)\|_{L^2}^2-\|u_n(s)\|_{L^2}^2+2\int_s^t\|\nabla u_n(l)\|_{L^2}^2\dif l-\|\Pi_nG\|_{L_2(U;L^2_\sigma)}^2(t-s)
 	\end{align*}
 	is a martingale. Recalling that we have normalized the norm on $H^{\gamma}$ to satisfy $\|f\|_{H^{\gamma}}\leq \|\nabla f\|_{L^{2}}$, it follows  for $t\geq r\geq s$
 	\begin{align*}
 	\|u_n(t)\|_{L^2}^2-\|u_n(r)\|_{L^2}^2+2\int_r^t\| u_n(l)\|_{H^\gamma}^2\dif l-\|\Pi_nG\|_{L_2(U;L^2_\sigma)}^2(t-r)\leq M_{n,{u_{n}}}^E(t)-M_{n,{u_{n}}}^E(r).
 	\end{align*}
 	Then under $P_n$
 	\begin{align*}
 	E_t^1-E_r^1\leq M_{n,x}^{E}(t)-M_{n,x}^{E}(r).
 	\end{align*}
 	Taking the conditional expectation with respect to $\mathcal{B}_r^0$ on  both sides, we find that under $P_n$ the process
 	$\{E^1(t)-E^1(s)\}_{t\geq s}$ is a $(\mathcal{B}_t^0)_{t\geq s}$-supermartingale. Similarly, we find that under $P_n$ the process
 	$\{E^p(t)-E^p(s)\}_{t\geq s}$ is a $(\mathcal{B}_t^0)_{t\geq s}$-supermartingale for all $p\in\mathbb{N}$. Finally, the existence of martingale solutions follows from tightness and the above proof of stability.
 \end{proof}

  \renewcommand{\theequation}{B.\arabic{equation}}

 \section{Intermittent jets}
\label{s:B}

In this  part we recall the construction of  intermittent jets from \cite[Section 7.4]{BV19}.
We point out that the construction is entirely deterministic, that is, none of the functions below depends on $\omega$.
Let us begin with  the following geometric lemma which can be found in  \cite[Lemma 6.6]{BV19}.

\bl\label{geometric}
	Denote by $\overline{B_{1/2}}(\mathrm{Id})$ the closed ball of radius $1/2$ around the identity matrix $\mathrm{Id}$, in the space of $3\times 3$ symmetric matrices. There
	exists $\Lambda\subset \mathbb{S}^2\cap \mathbb{Q}^3$ such that for each $\xi\in \Lambda$ there exists a  $C^\infty$-function $\gamma_\xi:\overline{B_{1/2}}(\mathrm{Id})\rightarrow\mathbb{R}$ such that
	\begin{equation*}
	R=\sum_{\xi\in\Lambda}\gamma_\xi^2(R)(\xi\otimes \xi)
	\end{equation*}
	for every symmetric matrix satisfying $|R-\mathrm{Id}|\leq 1/2$.
	For $C_\Lambda=8|\Lambda|(1+8\pi^3)^{1/2}$, where $|\Lambda|$ is the cardinality of the set $\Lambda$, we define
	the constant
	\begin{equation*}
	M=C_\Lambda\sup_{\xi\in \Lambda}(\|\gamma_\xi\|_{C^0}+\sum_{|j|\leq N}\|D^j\gamma_\xi\|_{C^0}).
	\end{equation*}
	For each $\xi\in \Lambda$ let us define $A_\xi\in \mathbb{S}^2\cap \mathbb{Q}^3$ to be an orthogonal vector to $\xi$. Then for each $\xi\in\Lambda$ we have that $\{\xi, A_\xi, \xi\times A_\xi\}\subset \mathbb{S}^2\cap \mathbb{Q}^3$ form an orthonormal basis for $\mathbb{R}^3$.
	We label by $n_*$ the smallest natural such that
	\begin{equation*}\{n_*\xi, n_*A_\xi, n_*\xi\times A_\xi\}\subset \mathbb{Z}^3\end{equation*}
	for every $\xi\in \Lambda$.
\el

Let $\Phi:\mathbb{R}^2\rightarrow\mathbb{R}$ be a smooth function with support in a ball of radius $1$. We normalize $\Phi$ such that
$\phi=-\Delta \Phi$ obeys
\begin{equation}\label{eq:phi}
\frac{1}{4\pi^2}\int_{\mathbb{R}^2}\phi^2(x_1,x_2)\dif x_1\dif x_2=1.
\end{equation}
By definition we know $\int_{\mathbb{R}^2}\phi dx=0$. Define $\psi:\mathbb{R}\rightarrow\mathbb{R}$ to be a smooth, mean zero function with support in the ball of radius $1$ satisfying
\begin{equation}\label{eq:psi}
\frac{1}{2\pi}\int_{\mathbb{R}}\psi^2(x_3)\dif x_3=1.
\end{equation}
For parameters $r_\perp, r_\|>0$ such that
\begin{equation*}r_\perp\ll r_\|\ll1,\end{equation*}
we define the rescaled cut-off functions
\begin{equation*}\phi_{r_\perp}(x_1,x_2)=\frac{1}{r_\perp}\phi\left(\frac{x_1}{r_\perp},\frac{x_2}{r_\perp}\right),\quad
\Phi_{r_\perp}(x_1,x_2)=\frac{1}{r_\perp}\Phi\left(\frac{x_1}{r_\perp},\frac{x_2}{r_\perp}\right),\quad \psi_{r_\|}(x_3)=\frac{1}{r_\|^{1/2}}\psi\left(\frac{x_3}{r_\|}\right).\end{equation*}
We periodize $\phi_{r_\perp}, \Phi_{r_\perp}$ and $\psi_{r_\|}$ so that they are viewed as periodic functions on $\mathbb{T}^2, \mathbb{T}^2$ and $\mathbb{T}$ respectively.

Consider a large real number $\lambda$ such that $\lambda r_\perp\in\mathbb{N}$, and a large time oscillation parameter $\mu>0$. For every $\xi\in \Lambda$ we introduce
\begin{equation*}\aligned
\psi_{(\xi)}(t,x)&:=\psi_{\xi,r_\perp,r_\|,\lambda,\mu}(t,x):=\psi_{r_{\|}}(n_*r_\perp\lambda(x\cdot \xi+\mu t))
\\ \Phi_{(\xi)}(x)&:=\Phi_{\xi,r_\perp,\lambda}(x):=\Phi_{r_{\perp}}(n_*r_\perp\lambda(x-\alpha_\xi)\cdot A_\xi, n_*r_\perp\lambda(x-\alpha_\xi)\cdot(\xi\times A_\xi))\\
\phi_{(\xi)}(x)&:=\phi_{\xi,r_\perp,\lambda}(x):=\phi_{r_{\perp}}(n_*r_\perp\lambda(x-\alpha_\xi)\cdot A_\xi, n_*r_\perp\lambda(x-\alpha_\xi)\cdot(\xi\times A_\xi)),
\endaligned\end{equation*}
where $\alpha_\xi\in\mathbb{R}^3$ are shifts to ensure that $\{\Phi_{(\xi)}\}_{\xi\in\Lambda}$ have mutually disjoint support.

The intermittent jets $W_{(\xi)}:\mathbb{R}\times\mathbb{T}^3 \rightarrow\mathbb{R}^3$ are defined as in \cite[Section 7.4]{BV19}.
\begin{equation}\label{intermittent}W_{(\xi)}(t,x):=W_{\xi,r_\perp,r_\|,\lambda,\mu}(t,x):=\xi\psi_{(\xi)}(t,x)\phi_{(\xi)}(x).\end{equation}
By the choice of $\alpha_\xi$ we have that
\begin{equation}\label{Wxi}
W_{(\xi)}\otimes W_{(\xi')}\equiv0, \textrm{ for } \xi\neq \xi'\in\Lambda,
\end{equation}
and by the normalizations \eqref{eq:phi} and \eqref{eq:psi} we obtain
$$
\frac1{(2\pi)^3}\int_{\mathbb{T}^3}W_{(\xi)}(t,x)\otimes W_{(\xi)}(t,x)\dif x=\xi\otimes\xi.
$$
These facts combined with Lemma \ref{geometric} imply that
\begin{equation}\label{geometric equality}
\frac1{(2\pi)^3}\sum_{\xi\in\Lambda}\gamma_\xi^2(R)\int_{\mathbb{T}^3}W_{(\xi)}(t,x)\otimes W_{(\xi)}(t,x)\dif x=R,
\end{equation}
for every symmetric matrix $R$ satisfying $|R-\textrm{Id}|\leq 1/2$. Since $W_{(\xi)}$ are not divergence free, we  introduce the corrector term
\begin{equation}\label{corrector}
W_{(\xi)}^{(c)}:=\frac{1}{n_*^2\lambda^2}\nabla \psi_{(\xi)}\times \textrm{curl}(\Phi_{(\xi)}\xi)
=\textrm{curl\,curl\,} V_{(\xi)}-W_{(\xi)}.
\end{equation}
with
\begin{equation*}
V_{(\xi)}(t,x):=\frac{1}{n_*^2\lambda^2}\xi\psi_{(\xi)}(t,x)\Phi_{(\xi)}(x).
\end{equation*}
Thus we have
\begin{equation*}
\div\left(W_{(\xi)}+W_{(\xi)}^{(c)}\right)\equiv0.
\end{equation*}

Finally, we  recall the key   bounds from \cite[Section 7.4]{BV19}. For $N, M\geq0$ and $p\in [1,\infty]$ the following holds
\begin{equation}\label{bounds}\aligned
&\|\nabla^N\partial_t^M\psi_{(\xi)}\|_{C_{{[-2,t]}}L^p}\lesssim r^{1/p-1/2}_\|\left(\frac{r_\perp\lambda}{r_\|}\right)^N
\left(\frac{r_\perp\lambda \mu}{r_\|}\right)^M,\\
&\|\nabla^N\phi_{(\xi)}\|_{L^p}+\|\nabla^N\Phi_{(\xi)}\|_{L^p}\lesssim r^{2/p-1}_\perp\lambda^N,\\
&\|\nabla^N\partial_t^MW_{(\xi)}\|_{C_{{[-2,t]}}L^p}+\frac{r_\|}{r_\perp}\|\nabla^N\partial_t^MW_{(\xi)}^{(c)}\|_{C_{{[-2,t]}}L^p}+\lambda^2\|\nabla^N\partial_t^MV_{(\xi)}
\|_{C_{{[-2,t]}}L^p}\\&\lesssim r^{2/p-1}_\perp r^{1/p-1/2}_\|\lambda^N\left(\frac{r_\perp\lambda\mu}{r_\|}\right)^M,
\endaligned\end{equation}
where the implicit constants may depend on $p, N$ and $M$, but are independent of $\lambda, r_\perp, r_\|, \mu$.

%

\def\cprime{$'$} \def\ocirc#1{\ifmmode\setbox0=\hbox{$#1$}\dimen0=\ht0
  \advance\dimen0 by1pt\rlap{\hbox to\wd0{\hss\raise\dimen0
  \hbox{\hskip.2em$\scriptscriptstyle\circ$}\hss}}#1\else {\accent"17 #1}\fi}

\end{document}